\documentclass[a4paper,11pt]{article}
\usepackage[utf8]{inputenc}
\usepackage[hmargin={20mm,20mm},vmargin={25mm,25mm}]{geometry}
\usepackage{amsmath,amsfonts,amsthm,amssymb,mathtools}
\usepackage{cancel}
\usepackage{comment}
\usepackage{enumerate}
\usepackage{graphicx}
\usepackage[ocgcolorlinks,allcolors={blue}]{hyperref}
\usepackage{placeins}

\usepackage{xcolor}

\usepackage{authblk}
\usepackage[percent]{overpic}
\usepackage{epstopdf}
\usepackage{graphicx}
\usepackage{subcaption}
\usepackage{tabularx}
\usepackage{multirow}
\usepackage{booktabs}
\usepackage{tcolorbox}

\newtheorem{theorem}{Theorem}
\newtheorem{proposition}[theorem]{Proposition}
\newtheorem{lemma}[theorem]{Lemma}
\newtheorem{corollary}[theorem]{Corollary}
\theoremstyle{remark}
\newtheorem{remark}[theorem]{Remark}
\theoremstyle{definition}
\newtheorem{assumption}{Assumption}

\usepackage{stackengine}

\usepackage{newtxtext,newtxmath}
\usepackage[normalem]{ulem}
\newcounter{corr}
\definecolor{violet}{rgb}{0.580,0.,0.827}
\newcommand{\corr}[3]{\typeout{Warning : a correction remains in page \thepage}
\stepcounter{corr}        
{\color{magenta}\ifmmode\text{\,\sout{\ensuremath{#1}}\,}\else\sout{#1}\fi}{\color{red}#2}{\color{violet} [#3]}
}
\definecolor{mygreen}{rgb}{0,0.4,0.2}

\def\b{\boldsymbol}

\newcommand{\red}[1]{{\color{red}#1}}

\newcommand{\R}{\mathbb{R}}    
\newcommand{\N}{\mathbb{N}} 
\newcommand{\Pk}{\mathbb{P}} 
\newcommand{\M}{\mathbb{M}} 
\newcommand{\Pkw}{\widehat{\Pk}}

\newcommand{\GRAD}{\b\nabla}                    %gradient
\newcommand{\GRADs}{\b{\epsilon}}         %symmetric gradient
\newcommand{\divs}{\nabla{\cdot}} 
\newcommand{\DIV}{\b\nabla{\cdot}}               %divergence
\newcommand{\norm}[2][]{\|#2\|_{#1}}

\newcommand{\tri}{|\!|\!|}

\newcommand{\VDL}{\b U_h(E)}
\newcommand{\VDG}{\b U_h}

\newcommand{\ZDG}{\b Z_h}
\newcommand{\QCG}{P}

\newcommand{\QDG}{P_h}

\def\P0{{\Pi^{0, E}_k}}

\def\PP0{{\boldsymbol{\Pi}^{0, E}_{k-1}}}

\newcommand{\email}[1]{\href{mailto:#1}{#1}}

\def\w{{\b w}}
\def\DOF{\boldsymbol{\chi}}  % DoF vector and DoF operators 
 
\def\mesh{\Omega_h}
\def\Proj{\P0}

\title{Virtual Element methods for non-Newtonian shear-thickening fluid flow problems}
\author[1]{P. F. Antonietti\footnote{\email{paola.antonietti@polimi.it}}}
\author[2,5]{L. Beir\~{a}o da Veiga\footnote{\email{lourenco.beirao@unimib.it}}}
\author[1]{M. Botti\footnote{\email{michele.botti@polimi.it}}}
\author[3]{A. Harnist\footnote{\email{andre.harnist@utc.fr}}}
\author[4]{G. Vacca\footnote{\email{giuseppe.vacca@uniba.it}}}
\author[1]{M. Verani \footnote{\email{marco.verani@polimi.it}}}

\affil[1]{MOX-Laboratory for Modeling and Scientific Computing,  Dipartimento di Matematica, Politecnico di Milano, 
Piazza Leonardo da Vinci 32 - 20133 Milano, Italy}

\affil[2]{Dipartimento di Matematica e Applicazioni, 
Universit\`a degli Studi di Milano-Bicocca, 
Via Roberto Cozzi 55  - 20125 Milano, Italy}

\affil[3]{Universit\'e de technologie de Compi\`egne, Alliance Sorbonne Universit\'e, LMAC, Compi\`egne, France}

\affil[4]{Dipartimento di Matematica, 
Universit\`a degli Studi di Bari, 
Via Edoardo Orabona 4  - 70125 Bari, Italy}

\affil[5]{IMATI-CNR, 
Via Ferrata 1  - 27100 Pavia, Italy}

\begin{document}
%% REMOVE KEYS%%%%
%\norefnames

\maketitle
\abstract{In this work, we present a comprehensive theoretical analysis for Virtual Element discretizations of incompressible non-Newtonian flows governed by the Carreau-Yasuda constitutive law, in the shear-thickening regime ($r>2$)  including both degenerate ($\delta=0$) and non-degenerate ($\delta>0$) cases. The proposed Virtual Element method features two distinguishing advantages: the construction of an exactly divergence-free discrete velocity field and compatibility with general polygonal meshes.
The analysis presented in this work extends the results of \cite{Antonietti_et_al_2024}, where only shear-thinning behavior ($1<r<2$) was considered. Indeed, the theoretical analysis of the shear-thickening setting requires several novel analytical tools, including: an inf–sup stability analysis of the discrete velocity-pressure coupling in non-Hilbertian norms, a stabilization term specifically designed to address the nonlinear structure as the exponent $r>2$; and the introduction of a suitable discrete norm tailored to the underlying nonlinear constitutive relation. Numerical results demonstrate the practical performance of the proposed formulation.
}

\section{Introduction} 

Numerous applications, such as polymer processing, additive manufacturing, material deposition, concentrated suspensions, and high-shear biological fluids, as well as various materials science problems, involve fluids that exhibit non-Newtonian behavior. A nonlinear relation between the strain rate and the shear stress characterizes this behavior. A paradigmatic example is the Carreau-Yasuda model, introduced in \cite{Yasuda.Armstrong.ea:81}, where the relation between the shear stress $\b{\sigma}(\cdot,\GRADs)$ and the strain rate $\GRADs$ is given by 
\begin{equation*}
\b{\sigma}(\cdot,\GRADs) = 
\mu(\cdot) (\delta^\alpha + |\GRADs|^\alpha)^{\frac{r-2}\alpha} \GRADs,
\end{equation*}
for $\mu$, $\alpha$, $\delta$ and the power-law index $r$ to be specified later on.  The Carreau-Yasuda model is a generalization of the Carreau model, which corresponds to the choice $\alpha=2$. The Carreau-Yasuda model models both shear-thinning (pseudo-plastic) behavior for $1 <r < 2$ and shear-thickening (dilatant) behavior for $r > 2$.   When $r=2$, the model reduces to the standard Newtonian fluid case. The case $\delta=0$ corresponds to the classical power-law model. 
From a mathematical and numerical analysis perspective, the shear-thickening regime presents distinct challenges, because the value of the power-law index ($r>2$) and the possible presence of the degenerate case $\delta=0$ lead to a significantly different mathematical structure. Indeed, in such cases, designing robust discretization schemes and deriving stability and \emph{a priori} estimates are particularly challenging.  \\

Numerical methods for non-Newtonian flows %in the shear-thinning (pseudo-plastic) regime 
have a long history, beginning with the seminal work of \cite{BN:1990}, which proposed a Finite Element approximation of a non-Newtonian flow model governed by either the Carreau or the power-law model. Still in the Finite Element framework, sharp error estimates were subsequently established in \cite{Sandri:93, BL:1993, Barrett.Liu:94}. In particular, the pioneering studies \cite{BL:1993} and \cite{Barrett.Liu:94} derived (in some cases, optimal) velocity and pressure error bounds in appropriate quasi-norms for models incorporating Carreau or power-law models. We also refer to, e.g., \cite{Belenki.Berselli.ea:12, Hirn:13, Kreuzer.Suli:16, Kalte.Ruzi:23} for more recent contributions.
%{\color{blue} [Marco: for example ref 8 treats case $r>1$, not only $1<r<2$, the same applies to reference 9]} \PA{[PA: Tolto riferimento a shear-thinning, così teniamo tutte le citazioni]}. 
In practical applications, computational domains often exhibit complex geometric features, necessitating discretization schemes that efficiently accommodate flexible, possibly adapted, grids. Consequently, discretization methods that can support general polyhedral meshes have been recently investigated, including Virtual Element Methods (VEM) for non-Newtonian incompressible fluids in the shear-thinning regime ($1 <r < 2$), see \cite{Antonietti_et_al_2024}, discontinuous Galerkin \cite{Malkmus.Ruzi.ea:18} and hybridizable discontinuous Galerkin \cite{Gatica.Sequeira:15} methods, and Hybrid High-Order schemes for non-Newtonian fluids governed by the Stokes equations \cite{Botti.Castanon-Quiroz.ea:21} and Navier-Stokes equations with nonlinear convection in \cite{Castanon.Di-Pietro.ea:21}. \\

In this work, we focus on the Virtual Element method, originally introduced in \cite{volley} for second-order elliptic problems and subsequently generalized to a broad class of differential problems. In the context of fluid flow problems, a key advantage of VEM is its ability to construct divergence-free discrete velocity spaces on general polygonal or polyhedral meshes, thereby eliminating the need for ad hoc stabilization and ensuring mass conservation. For this reason, over the past decade, VEM has been extensively developed for approximating Newtonian fluid flow problems. In \cite{AntoniettiBeiraoMoraVerani_2014}, a novel stream-function-based VEM formulation for the Stokes problem, relying on a suitably designed stream function space that characterizes the divergence-free subspace of the discrete velocity field, has been proposed and analyzed. Divergence-free virtual elements have been introduced in \cite{BLV:2017}.
Further developments for the Stokes problem can be found in \cite{Cangiani-Gyrya-Manzini:2016, Caceres-Gatica_2017, BeiraodaVeiga-Dassi-Vacca_2020, Chernov-Marcati-Mascotto_2021, Lepe-Rivera_2021, Bevilacqua-Scacchi_2022, VEM_Oseen_stab, Bevilacqua-Dassi-Scacchi-Zampini_2024}; see also \cite{Frerichs-Merdon_2022, BeiraodaVeiga-Dassi-DiPietro-Droniou_2022}, where arbitrary-order pressure-robust VEMs have been investigated. The Virtual Element discretization of the Navier--Stokes equations was first studied in \cite{BLV:2018} and further investigated in \cite{Gatica-Munar-Sequeira_2018, BeiraodaVeiga-Mora-Vacca_2019, Adak-Mora-Natarajan_2021, Gatica-Sequeira_2021}; we also refer to \cite{Caceres-Gatica-Sequeira_2017, Adak-Mora-Silgado_2024, Li-Hu-Feng_2024}, for quasi-Newtonian Stokes flows. 
Recently, Virtual Element approximations to the coupled Navier-Stokes and heat equations have been analyzed in \cite{Antonietti-Vacca-Verani_2023, BMS-23}.  Moreover, least-squares Virtual Element discretizations of the Stokes and Navier-Stokes systems have recently been analyzed in \cite{Li-Hu-Feng_2022} and \cite{Wang-Wang_2024}, respectively. For a comprehensive review of recent advances in VEM, we refer to the monograph \cite{bookVEM}. \\

In this paper, we analyze a Virtual Element discretization for steady incompressible non-Newtonian flow governed by the Carreau-Yasuda constitutive law, explicitly addressing the shear-thickening range ($r>2$) and both the degenerate ($\delta=0$) and non-degenerate ($\delta>0$) cases. 
To the best of the authors' knowledge, previous $W^{1,r}$ \emph{a priori} estimates on polytopal methods for non-Newtonian flow have focused on the shear-thinning (pseudoplastic) regime, with the exception of the results in \cite{Botti.Castanon-Quiroz.ea:21}, where a Hybrid High Order discretization of non-Newtonian fluids with $r\in(1, \infty)$ is studied, leading to velocity and pressure error bounds of order $\frac{k+1}{r-1}$, being $k$ the polynomial order of the discretization. Notice that here we get a similar dependence of $r$ in the convergence order. However, we provide a better convergence order ($2k/r$) of the $W^{1,r}$ error norm under more regularity of the flux contribution.
The Virtual Element formulation is based on employing the divergence-free Virtual Element spaces introduced in \cite{BLV:2017,BLV:2018} and we focus on the extension of the stability and convergence analysis in the case $r>2$ and, possibly, $\delta=0$, addressing thereby the theoretical analysis not covered in \cite{Antonietti_et_al_2024}, where only the shear thinning regimes was considered.  The proposed method offers two principal advantages: it accommodates general polygonal meshes and leads to a discrete velocity field that is exactly divergence-free.  We point out that, the mathematical analysis of the shear-thickening range $r>2$ and the degenerate case $\delta=0$ presented in this work requires several novel technical contributions, including the establishment of inf-sup stability of the discrete velocity-pressure coupling in non-Hilbertian norms and the introduction of a stabilization term tailored to the distinct analytical structure that arises when the exponent crosses the threshold $r=2$. 
Furthermore, we note that, in the presence of non convex-elements $E$, the definition of the aforementioned Virtual Elements does not guarantee that the local spaces are in the natural Sobolev space $W^{1,r}(E)$ for large values of $r$. As a consequence, and in order to include in our analysis all type of polygonal meshes, we evaluate the stability and the error of the scheme in a discrete norm. Whenever the elements are convex, the standard norm is immediately recovered.
%%
% Moreover, in addition to the methodological advances mentioned above, a key contribution is the introduction of a discrete norm tailored to the model's nonlinear structure.  This norm is central to stability and error analysis and is introduced here for the first time to address the theoretical analysis. 
%%
We prove that if $r>2$, and the underlying solution is sufficiently smooth, the order of convergence of the velocity and the pressure is $\frac{k}{r-1}$ for $\delta \ge 0$. In the non-degenerate case, i.e. $\delta > 0$, we show that the order of convergence is $\frac{2k}{r}$. To the authors' best knowledge, this is the first result for a polytopal discretization of non-Newtonian flows in the degenerate power-law/shear-thickening regime.  \\

The remainder of the manuscript is organized as follows. In Section~\ref{sec:model}, after introducing the notation used throughout the paper,  we present the weak formulation of the continuous problem and discuss its well-posedness. Section~\ref{sec:VEM} describes the proposed divergence-free Virtual Element discretization and presents the well-posedness analysis. The \emph{a priori} error analysis is detailed in Section~\ref{sec:error_analysis}. Section~\ref{sec:tests} reports numerical experiments illustrating the method's performance. Finally, Section~\ref{sec:conclusions} contains a summary of the obtained results and some concluding remarks.

\section{Model problem}\label{sec:model}
In this section, we introduce the notation used throughout the paper, present the model problem, and discuss its well-posedness.
%\subsection{Notation}

The vector spaces considered hereafter are over $\mathbb{R}$. We denote by $\mathbb{R}_{+}$ the set of non-negative real numbers. Given a vector space $V$ with norm $\norm[V]{\cdot}$, the notation $V'$ denotes its dual space and ${}_{V'}\langle\cdot,\cdot\rangle_V$ the duality between $V$ and $V'$. The notation $\b v\cdot\b w$ and $\b v\times\b w$ designates the scalar and vector products of two vectors $\b v,\b w\in\mathbb{R}^d$, and $\vert \b v \vert$ denotes the Euclidean norm of $\b v$ in  $\mathbb{R}^d$.
The inner product in $\mathbb{R}^{d\times d}$ is defined for $\b\tau,\b\eta\in\mathbb{R}^{d\times d}$ by $\b\tau:\b\eta\coloneq\sum_{i,j=1}^d \tau_{i,j}\eta_{i,j}$ and the induced norm is given by $|\b\tau|= \sqrt{\b\tau:\b\tau}$.\\

Let $\Omega\subset\mathbb{R}^d$ denote a bounded, connected, polyhedral open set with Lipschitz boundary $\partial\Omega$ and let $\b n$ be the outward unit normal to $\partial\Omega$. To simplify the exposition, we restrict the presentation to the two-dimensional case, i.e., $d=2$, but the analysis remains valid in the three-dimensional case $d=3$ as well, with minor technical differences.
We denote with $\b x = (x_1, x_2)$ the independent variable.
We assume that the boundary is partitioned into two disjoint subsets $\partial\Omega=\Gamma_D\cup\Gamma_N$, 
with $|\Gamma_D|>0$, such that a Dirichlet condition is given on $\Gamma_D$ and a Neumann condition on $\Gamma_N$. \\

Throughout the article, spaces of functions, vector fields, and tensor fields, defined over any $X\subset\overline{\Omega}$ are denoted by italic capitals, boldface Roman capitals, and special Roman capitals, respectively. The subscript $\rm{s}$ 
%appended to a special  capital 
denotes a space of symmetric tensor fields. For example, $L^2(X)$, $\b L^2(X)$, and $\mathbb{L}^2_s(X)$ denote the spaces of square-integrable functions, vector fields, and symmetric tensor fields, respectively. 
The notation $W^{m,r}(X)$, for $m \geq 0$ and $r\in[1,+\infty]$, with the convention that $W^{0,r}(X) =L^r(X)$, and $W^{m,2}(X) = H^m(X)$, designates the classical Sobolev spaces.  
The trace map is denoted by $\gamma:W^{1,r}(\Omega)\to W^{1-\frac1r,r}(\partial\Omega)$.
Finally, given $\Gamma\subset\partial\Omega$, we denote by $W^{1,r}_{0,\Gamma}(\Omega)$ the subspace of $W^{1,r}(\Omega)$ spanned by functions having zero trace on $\Gamma$.
The symbol $\nabla$ denotes the gradient for scalar functions, while  
$\GRAD$, $\b{\epsilon(\cdot)} = \frac{\GRAD(\cdot) + \GRAD^{\rm T}(\cdot)}2$, and $\divs$ denote the gradient, the symmetric gradient operator, and the divergence operator, respectively, whereas $\DIV$ denotes the vector-valued divergence operator for tensor fields.\\

We consider the creeping flow of a non-Newtonian fluid occupying $\Omega$ and subjected to a volumetric force $\b{f}:\Omega\to\mathbb{R}^d$ and a traction $\b{g}:\Gamma_N\to\mathbb{R}^d$ described by the non-linear Stokes equations
\begin{equation}\label{eq:stokes.continuous}
  \begin{aligned} 
    -\DIV\b{\sigma}(\cdot,\b{\epsilon}(\b{u}))+\GRAD p &= 
    \b{f} &\qquad& \text{ in } \Omega,  \\
    \divs \b{u} &= 0 &\qquad& \text{ in }  \Omega, \\
    \b{\sigma}(\cdot,\b\epsilon(\b{u}))\b{n} - p\b{n} &= \b{g} &\qquad& \text{ on } \Gamma_N, \\
    \b{u} &= \b{0} &\qquad& \text{ on } \Gamma_D, 
  \end{aligned}
\end{equation} 
where $\b{u}:\Omega\to\mathbb{R}^d$ and $p:\Omega\to\mathbb{R}$ denote the velocity field and the pressure field, respectively. Non-homogeneous Dirichlet conditions can be considered in place of \eqref{eq:stokes.continuous} up to minor modifications.\\

In this work, we consider the Carreau--Yasuda model,  introduced in \cite{Yasuda.Armstrong.ea:81}, as a reference model for the non-linear shear stress-strain rate relation, i.e.,
\begin{equation}\label{eq:Carreau}
\b{\sigma}(\b{x},\GRADs(\b{v})) = 
\mu(\b x) (\delta^\alpha + |\GRADs(\b{v})|^\alpha)^{\frac{r-2}\alpha} \GRADs(\b{v}),
\end{equation}
where $\mu:\Omega\to[\mu_{-}, \mu_{+}]$, with $0<\mu_{-}<\mu_{+}<\infty$, $\alpha\in (0,\infty)$, and $\delta\ge 0$ and  $r\in[2,\infty)$. 
%\textbf{\color{red}Andre: $\alpha$ could also be between $0$ and $1$, or even a function between $\alpha_- > 0$ and $\alpha_+ > 0$.}
The Carreau--Yasuda law is a generalization of the Carreau model corresponding to the case $\alpha=2$.
The case $\delta = 0$ corresponds to the classical power-law model. 
%For $r=2$, problem \eqref{eq:stokes.continuous} reduces to the standard Stokes system for Newtonian fluids.
The stress-strain law \eqref{eq:Carreau} satisfies the following assumption:
\begin{assumption}\label{ass:hypo}
The shear stress-strain rate law $\b\sigma:\Omega\times\mathbb{R}^{d\times d}_{\rm s} \to \mathbb{R}^{d\times d}_{\rm s}$ appearing in \eqref{eq:stokes.continuous} is a Caratheodory function satisfying $\b{\sigma}(\cdot,\b{0})=\b{0}$ and for a fixed $r\in [2,\infty)$ there exist real numbers $\delta \in [0,+\infty)$ and $\sigma_{\rm{c}}, \sigma_{\rm{m}}\in(0,+\infty)$ such that the following conditions hold:
\begin{subequations}\label{eq:hypo}
  \begin{alignat}{2} 
  \label{eq:hypo.continuity}
  &|\b{\sigma}(\b{x},\b{\tau})-\b{\sigma}(\b{x},\b{\eta})| \le \sigma_{\rm c} \left(\delta^r + |\b\tau|^r + |\b\eta|^r \right)^\frac{r-2}{r}| \b\tau-\b\eta|,
  &\quad&\text{(H\"older continuity)}
  \\
  \label{eq:hypo.monotonicity}
  &\left(\b{\sigma}(\b x,\b\tau)-\b{\sigma}(\b x,\b\eta)\right):\left(\b\tau-\b\eta\right) \ge \sigma_{\rm m}\left(\delta^r+|\b\tau|^r + |\b\eta|^r \right)^\frac{r-2}{r}|\b\tau-\b\eta|^{2},
  &\quad& \text{(strong monotonicity)}
  \end{alignat}
\end{subequations}
for almost every $\b{x}\in\Omega$ and all $\b{\tau},\b{\eta}\in\mathbb{R}^{d\times d}_{\rm{s}}$.
\end{assumption}
 The constants $\sigma_{\rm{c}}, \sigma_{\rm{m}}>0$ in \eqref{eq:hypo} for the Carreau-Yasuda model \eqref{eq:Carreau} 
satisfy 
  \[
    \sigma_{\rm{c}} = \mu_+(r-1)2^{\left(\frac{1}{\alpha}-\frac{1}{r}\right)^\oplus(r-2)}
    \quad\text{and}\quad
    \sigma_{\rm{m}} = \dfrac{\mu_-}{r-1}2^{\left[-\left(\frac{1}{\alpha}-\frac{1}{r}\right)^\ominus-1\right](r-2)-1} \,,
    \]
where $\xi^\oplus\coloneq\max(0,\xi)$ and $\xi^\ominus\coloneq-\min(0,\xi)$ for all $\xi \in \R$  (cf. \cite{Botti.Castanon-Quiroz.ea:21}).
Finally, for further use, we adopt the short-hand notation $\mathsf{a} \lesssim  \mathsf{b}$ to denote the inequality $\mathsf{a} \leq C \mathsf{b}$, for a constant $C>0$ that may depend on $\sigma_{\rm{c}}, \sigma_{\rm{m}},\delta$ (or related parameters) in Assumption~\ref{ass:hypo} and $r$, but is independent of the discretization parameters.
The obvious extensions $\mathsf{a} \gtrsim  \mathsf{b}$ and $\mathsf{a} \simeq  \mathsf{b}$ hold.

\subsection{Weak formulation}

In this section, we provide the variational formulation of \eqref{eq:stokes.continuous}. For $r\in(1,\infty)$, we introduce the conjugate index  defined as $r'= \frac{r}{r-1}$
and recall Korn's first inequality (see, e.g., \cite[Theorem 1.2]{Ciarlet.Ciarlet:05} and \cite[Theorem 1]{Geymonat.Suquet:86}): there is $C_{\rm K} > 0$ depending only on $\Omega$ and $r$ such that for all $\b v \in \b{W}^{1,r}_{0,\Gamma_D}(\Omega)$,
\begin{equation}\label{eq:Korn}
\|\b v\|_{\b{W}^{1,r}(\Omega)} \le C_{\rm K} \|\GRADs (\b v)\|_{\mathbb{L}^r(\Omega)}.
\end{equation}

Let $r\in [2,\infty)$ be the Sobolev exponent dictated by the non-linear stress-strain law characterizing problem \eqref{eq:stokes.continuous} and satisfying Assumption~\ref{ass:hypo}.
We define the velocity and pressure spaces incorporating the boundary condition on $\Gamma_D$ and the zero-average constraint in the case $\Gamma_D=\partial\Omega$, respectively: 
\[
\b U = \b{W}^{1,r}_{0,\Gamma_D}(\Omega) %\coloneqq \left\{\b v \in W^{1,r}(\Omega,\mathbb{R}^d)\ : \ \b v_{|{\Gamma_D}} = \b 0 \right\},
\qquad
P = \begin{cases} L^{r'}(\Omega)&\text{if } |\Gamma_D|<|\partial\Omega| \\ L^{r'}_0(\Omega) \coloneqq \left\{q \in L^{r'}(\Omega)\ : \ \textstyle\int_\Omega q = 0 \right\} &\text{if } \Gamma_D =\partial\Omega. \end{cases} 
\]
Assuming $\b f \in \b{L}^{r'}(\Omega)$ and $\b g \in \b{L}^{r'}(\Gamma_N)$, the weak form of \eqref{eq:stokes.continuous} reads:
Find $(\b u,p) \in \b U \times P$ such that
\begin{equation}\label{eq:stokes.weak}
  \begin{aligned}
     a(\b u,\b v)+b(\b v,p) &= \int_\Omega \b f \cdot \b v + \int_{\Gamma_N} \b g \cdot\b v&\qquad \forall \b v \in \b U, \\
     -b(\b u,q) &= 0 &\qquad \forall q \in P,   
  \end{aligned}
\end{equation}
where $a : \b U \times \b U \to \mathbb{R}$ and $b : \b U \times P \to \mathbb{R}$ are defined  for all $\b v,\b w \in \b U$ and all $q \in L^{r'}(\Omega)$ by
\begin{equation}\label{eq:a.b}
  a(\b w,\b v) \coloneqq \int_\Omega \b\sigma(\cdot,\b\epsilon(\b w)) : \b\epsilon(\b v),\qquad
  b(\b v,q)  \coloneqq -\int_\Omega (\divs \b v) q.
\end{equation}

Introducing the subset $\b Z = \{ \b v \in \b U \quad \text{s.t.} \quad \divs  \b v = 0  \}\subset \b U$ corresponding to the kernel of $b(\cdot,\cdot)$, problem~\eqref{eq:stokes.weak} can be formulated in the equivalent kernel form:
Find $\b u \in \b Z$ such that
\begin{equation}\label{eq:stokes.weak.Z}
     a(\b u, \b v) = \int_\Omega \b f \cdot \b v + \int_{\Gamma_N} \b g \cdot  \b v \qquad \forall \b v \in \b Z \,.
\end{equation}

\subsection{Well-posedness}

In this section, we report the properties of the functions $a(\cdot,\cdot)$ and $b(\cdot,\cdot)$ defined in \eqref{eq:a.b} and we prove the well-posedness of  \eqref{eq:stokes.weak}. 

\begin{lemma}[Continuity and monotonicity of $a$]
%\label{lemm:properties_a_continuous}
For all $\b u, \b v, \b w \in \b U$, setting $\b e = \b u - \b w$, it holds
\begin{subequations}%\label{eq:a_properties}
  \begin{alignat}{2} 
  \label{eq:a.continuity}
  |a(\b u, \b v) - a(\b w, \b v)| &\lesssim \sigma_{\rm c}
  \left(\delta^r +\norm[\b{W}^{1,r}(\Omega)]{\b u}^r + \norm[\b{W}^{1,r}(\Omega)]{\b w}^r\right)^{\frac{r-2}r} \norm[\b{W}^{1,r}(\Omega)]{\b e} \norm[\b{W}^{1,r}(\Omega)]{\b v},
  \\
  \label{eq:a.monotonicity}
  a(\b u, \b e) - a(\b w, \b e) &\gtrsim \sigma_{\rm m} 
  %\left(\delta^r +\norm[\b{W}^{1,r}(\Omega)]{\b u}^r + \norm[\b{W}^{1,r}(\Omega)]{\b w}^r\right)^{\frac{r-2}r} 
  \norm[\b{W}^{1,r}(\Omega)]{\b e}^r.
  \end{alignat}
\end{subequations}
\end{lemma}
\begin{proof} 
We follow the lines of \cite[Lemma 7.3]{Botti.Castanon-Quiroz.ea:21}. Let $\b u, \b v, \b w \in \b U$ and set $\b e = \b u - \b w$. 

\noindent
\emph{(i) H\"older continuity.}
Recalling the H\"older continuity property \eqref{eq:hypo.continuity} and using the generalized H\"older inequality with exponents $(\frac{r}{r-2},r, r)$, we have 
%(here ${\b e}= {\b u} - {\b w}$)
\begin{equation}
\label{eq:for-holder}
\begin{aligned}
|a(\b u, \b v) - a(\b w, \b v)| &\le
\int_\Omega \left|[\b\sigma(\cdot,\b\epsilon(\b u)) - \b\sigma(\cdot,\b\epsilon(\b w)) ] : \b\epsilon(\b v)  \right| \\
%&\le 
%\sigma_{\rm c} \int_\Omega \left(\delta^r + |\b\epsilon(\b u)|^r + |\b\epsilon(\b w)|^r \right)^\frac{r-2}{r} |\b\epsilon(\b e)| |\b\epsilon(\b v)| \\
%&\lesssim 
%\sigma_{\rm c} \int_\Omega |\b\epsilon(\b e)|^{r-1} |\b\epsilon(\b v)|  
&\lesssim
\sigma_{\rm c} \left(\int_\Omega \delta^r + |\b\epsilon(\b u)|^r + |\b\epsilon(\b w)|^r \right)^\frac{r-2}{r} \norm[\b{W}^{1,r}(\Omega)]{\b e} \norm[\b{W}^{1,r}(\Omega)]{\b v}.
\end{aligned}
\end{equation}
\emph{(ii) Strong monotonicity.} 
First, we observe that for all $\b{\tau},\b{\eta}\in\mathbb{R}^{d\times d}_{\rm{s}}$ the triangle inequality implies 
$
2^{1-r} \vert\b\tau -\b\eta\vert^r \le 
\vert\b\tau\vert^r +\vert\b\eta\vert^r$.
Thus, since $\delta\ge0$ and $r \ge 2$, 
\begin{equation}\label{eq:pre_holdercont}
%\begin{aligned}
%\left(\delta^r +\vert\b\tau\vert^r +\vert\b\eta\vert^r \right)^{\frac{r-2}{r}}
%\le \left(2^{1-r} \vert\b\tau -\b\eta\vert^r\right)^{\frac{r-2}{r}}
%\le 2^{\frac{(1-r)(r-2)}{r}} \vert\b\tau -\b\eta\vert^{r-2}
%\lesssim \vert\b\tau -\b\eta\vert^{r-2},
%\qquad\text{if }\, r\le2,
%\\
\left(\delta^r +\vert\b\tau\vert^r +\vert\b\eta\vert^r \right)^{\frac{r-2}{r}}
\ge \left(2^{1-r} \vert\b\tau -\b\eta\vert^r\right)^{\frac{r-2}{r}}
\gtrsim \vert\b\tau -\b\eta\vert^{r-2}.
%\qquad\text{if }\, r>2.
%\end{aligned}
\end{equation}
Using Korn's inequality \eqref{eq:Korn} together with 
%the monotonicity property of $\b \sigma$ in 
\eqref{eq:hypo.monotonicity} and the inequality in \eqref{eq:pre_holdercont}, we obtain
$$
\begin{aligned}
\sigma_{\rm m} \norm[\b{W}^{1,r}(\Omega)]{\b e}^r &\le C_{\rm K}^r
\left(\int_\Omega \sigma_{\rm m}\vert \b\epsilon(\b u-\b w) \vert^r\right) 
\lesssim
\int_\Omega \sigma_{\rm m}
\vert\b\epsilon(\b u) -\b\epsilon(\b w)\vert^2 
\vert\b\epsilon(\b u) -\b\epsilon(\b w)\vert^{r-2}
\\
&\lesssim
\int_\Omega \left(\delta^r+|\b\epsilon(\b u)|^r + |\b\epsilon(\b w)|^r \right)^{\frac{2-r}{r} + \frac{r-2}{r}}
(\b\sigma(\cdot,\b\epsilon(\b u)) - \b\sigma(\cdot,\b\epsilon(\b w)) ) : \b\epsilon(\b u -\b w) \\
&\lesssim
%\left(|\Omega|\delta^r +\norm[\b{W}^{1,r}(\Omega)]{\b u}^r + \norm[\b{W}^{1,r}(\Omega)]{\b w}^r \right)^\frac{2-r}{r} 
a(\b u, \b e) - a(\b w, \b e).
\end{aligned}
$$
\end{proof}

The following result is needed to infer the existence of a unique pressure $p\in P$ solving problem \eqref{eq:stokes.weak} from the well-posedness of problem \eqref{eq:stokes.weak.Z}. For its proof, we refer to \cite[Theorem 1]{Bogovski:79}.
\begin{lemma}[Inf-sup condition]%\label{lem:inf_sup}
For any $r\in[2,\infty)$ there exists a positive constant $\beta(r)$ such that the bilinear form $b(\cdot,\cdot)$ defined in \eqref{eq:a.b} satisfies
\begin{equation}\label{eq:inf_sup}
\inf_{q\in P}\;\sup_{\b w\in \b{U}\setminus\{\b 0\}}\; 
\frac{b(\b w,q)}{\norm[L^{r'}(\Omega)]{q}\norm[\b{W}^{1,r}(\Omega)]{\b w}} \ge \beta(r) > 0.
\end{equation}
\end{lemma}

We are now ready to establish the well-posedness of problem \eqref{eq:stokes.weak}.

\begin{proposition}[Well-posedness]\label{prop:well-posed_cont}
For any $r\in[2,\infty)$, there exists a unique solution $(\b u, p)\in \b{U}\times P$ to problem \eqref{eq:stokes.weak} satisfying the a priori estimates
\begin{subequations}
\begin{align}\label{eq:a-priori_cont.u}
    \norm[\b{W}^{1,r}(\Omega)]{\b u}&\lesssim
\sigma_{\rm m}^{-\frac1{r-1}}
\left(\| \b f \|_{\b L^{r'}(\Omega)} +
\| \b g \|_{\b L^{r'}(\Gamma_N)}\right)^{\frac1{r-1}},
\\
\label{eq:a-priori_cont.p}
\norm[L^{r'}(\Omega)]{p} &\lesssim
\left(1+\frac{\sigma_{\rm c}}{\sigma_{\rm m}}\right)  
\left(\| \b f \|_{\b L^{r'}(\Omega)} + \| \b g \|_{\b L^{r'}(\Gamma_N)}\right) + \sigma_{\rm c}\delta^{r-2}
\left(\frac{\| \b f \|_{\b L^{r'}(\Omega)} +
\| \b g \|_{\b L^{r'}(\Gamma_N)}}{\sigma_{\rm m}}\right)^{\frac1{r-1}} \hspace{-3mm}.
\end{align}
\end{subequations}
\end{proposition}
\begin{proof}
We focus here on the \emph{a priori} bounds \eqref{eq:a-priori_cont.u} and \eqref{eq:a-priori_cont.p}, while for the uniqueness and existence we refer to \cite{Beirao-da-Veiga:09, Berselli.Diening.ea:10}. Owing to \eqref{eq:a.monotonicity} with $\b w = \b 0$, taking $\b v = \b u$ in \eqref{eq:stokes.weak.Z}, and using the H\"older inequality together with the continuity of the trace map, one has
$$
\sigma_{\rm m} 
%\left(\delta^r +\norm[\b{W}^{1,r}(\Omega)]{\b u}^r \right)^{\frac{r-2}r} 
\norm[\b{W}^{1,r}(\Omega)]{\b u}^r \lesssim 
a(\b u, \b u) = \int_\Omega \b f \cdot \b u + \int_{\Gamma_N} \b g \cdot \b\gamma(\b u)
\lesssim 
\left(\| \b f \|_{\b L^{r'}(\Omega)} +
\| \b g \|_{\b L^{r'}(\Gamma_N)}\right)
\norm[\b{W}^{1,r}(\Omega)]{\b u}.
$$
From the previous bound, the velocity estimate in \eqref{eq:a-priori_cont.u} follows.
Concerning the estimate of the pressure field: owing to the inf-sup condition \eqref{eq:inf_sup} and equation \eqref{eq:stokes.weak}, it is inferred that
$$
\beta(r) \norm[L^{r'}(\Omega)]{p} \le 
\sup_{\b v\in\b U\setminus\{\b 0\}}\frac{b(\b v, p)}{\norm[\b{W}^{1,r}(\Omega)]{\b v}} = 
\sup_{\b v\in\b U\setminus\{\b 0\}}\frac{\int_\Omega \b f \cdot \b v + \int_{\Gamma_N} \b g \cdot \b v - a(\b u, \b v)}{\norm[\b{W}^{1,r}(\Omega)]{\b v}}.
$$
Applying the H\"older inequality, the continuity of $a$ in \eqref{eq:a.continuity} with $\b{w}=\b0$, and the \emph{a priori} estimate of the velocity \eqref{eq:a-priori_cont.u}, we obtain 
$$
\begin{aligned}
\norm[L^{r'}(\Omega)]{p} &\lesssim
\| \b f \|_{\b L^{r'}(\Omega)} + \| \b g \|_{\b L^{r'}(\Gamma_N)} +
\sigma_{\rm c}
  \left(\delta^r +\norm[\b{W}^{1,r}(\Omega)]{\b u}^r \right)^{\frac{r-2}r} \norm[\b{W}^{1,r}(\Omega)]{\b u}
  \\ &\lesssim 
\left(1+\frac{\sigma_{\rm c}}{\sigma_{\rm m}}\right)  
\left(\| \b f \|_{\b L^{r'}(\Omega)} + \| \b g \|_{\b L^{r'}(\Gamma_N)}\right) + \sigma_{\rm c}\delta^{r-2}
\sigma_{\rm m}^{-\frac1{r-1}}
\left(\| \b f \|_{\b L^{r'}(\Omega)} +
\| \b g \|_{\b L^{r'}(\Gamma_N)}\right)^{\frac1{r-1}}.
\end{aligned}
$$
\end{proof}

\section{Virtual Element method}\label{sec:VEM}

In the present section, we initially review divergence-free Virtual Elements of general order \cite{BLV:2017,BLV:2018}. Afterterwards, we design the computable forms and formulate the discrete problem that approximates the nonlinear equation \eqref{eq:stokes.weak}, finally establishing its well-posedness.

%-------------------------------------------------------------------
\subsection{Mesh and discrete spaces}
\label{sub:preliminaries}
Let $\{\Omega_h\}_h$ be a sequence of decompositions of the domain $\Omega \subset \R^2$ into general polytopal elements. 
For each $E \in \Omega_h$, we denote with $h_{E}$ the diameter, with $|E|$ the area, with $\b x_{E} = (x_{E, 1}, x_{E, 2})$ the centroid and we set $h = \sup_{E \in \Omega_h} h_{E}$. 
We suppose that $\{\Omega_h\}_h$ fulfills the following assumption.
\begin{assumption}\label{ass:mesh}
\textbf{(Mesh assumptions).} There exists a positive constant $\rho$ such that for any $E \in \{\Omega_h\}_h$ 
\begin{itemize}
\item $E$ is star-shaped with respect to a ball $B_E$ of radius $ \geq\, \rho \, h_E$;
\item any edge $e$ of $E$ has length  $ \geq\, \rho \, h_E$.
\end{itemize}
\end{assumption}

Given $\omega \subset \overline\Omega$ and $n \in \N$, we denote by $\Pk_n(\omega)$ the set of polynomials on $\omega$ of degree less or equal to $n$, with the convention that $\Pk_{-1}(\omega)=\{ 0 \}$.
A natural basis for the space $\Pk_n(E)$ is the set of normalized
monomials
$
\M_n(E) = \left\{ 
m_{\boldsymbol{\alpha}},
\,\,  \text{with} \,\, \boldsymbol{\alpha} = (\alpha_1, \alpha_2) \in \N^2 \,\, \text{such that} \,\,
|\boldsymbol{\alpha}| \leq n
\right\},
$
where
\[
m_{\boldsymbol{\alpha}} =
\prod_{i=1}^2  
\left(\frac{x_i - x_{E, i}}{h_{E}} \right)^{\alpha_i}
\qquad \text{and} \qquad
|\boldsymbol{\alpha}|= \sum_{i=1}^2 \alpha_i \,.
\]
For any $e$ edge of $\Omega_h$, let $\boldsymbol{t}_e$ and  ${\boldsymbol{n}}_e$ denote the tangent and the normal vectors to the edge $e$ respectively. 
Moreover, the normalized monomial set $\M_n(e)$ is defined analogously as the span of all one-dimensional normalized monomials of degree up to $n$.
For any $m \leq n$, we denote with
\[
\Pkw_{n \setminus m}(E) = {\rm span}  
\left\{ 
m_{\boldsymbol{\alpha}},
\,\,\,  \text{with} \,\,\,
m +1 \leq |\boldsymbol{\alpha}| \leq n
\right\} \,.
\]
For any $E\in\Omega_h$, the $L^2$-projection $\Pi_n^{0, E} \colon L^2(E) \to \Pk_n(E)$ is defined such that
\begin{equation}
\label{eq:P0_k^E}
\int_{E} q_n (v - \, {\Pi}_{n}^{0, E}  v) \, {\rm d} E = 0 \qquad  \text{for all $v \in L^2(E)$  and $q_n \in \Pk_n(E)$,} 
\end{equation} 
with obvious extension $\Pi^{0, E}_{n} \colon \b L^2(E) \to [\Pk_n(E)]^d$ and $\boldsymbol{\Pi}^{0, E}_{n} \colon \mathbb{L}^2(E) \to [\Pk_n(E)]^{d\times d}$ for vector and tensor functions, respectively.
Moreover, the elliptic projection ${\Pi}_{n}^{\nabla,E} \colon W^{1,2} (E) \to \Pk_n(E)$ is given by 
\begin{equation*}
\left\{
\begin{aligned}
& \int_{E} \nabla  \,q_n \cdot \nabla ( v - \, {\Pi}_{n}^{\nabla,E}   v)\, {\rm d} E = 0 \quad  \text{for all $v \in W^{1,2}(E)$ and  $q_n \in \Pk_n(E)$,} \\
& \int_{\partial E}(v - \,  {\Pi}_{n}^{\nabla, E}  v) \, {\rm d}s = 0 \, ,
\end{aligned}
\right.
\end{equation*}
with extension for vector fields $\Pi^{\nabla, E}_{n} \colon \b W^{1,2}(E) \to [\Pk_n(E)]^d$.

We also recall the following useful results:
\begin{itemize}
\item Trace inequality with scaling \cite{brenner-scott:book}:
For any $E \in \Omega_h$ and for any  function $v \in W^{1,r}(E)$ it holds 
\begin{equation}
\label{eq:cont_trace}
\|v\|^r_{L^r(\partial E)} \lesssim h_E^{-1}\|v\|^r_{L^r(E)} + h_E^{r-1} \|\nabla v\|^r_{\b{L}^r(E)} \,.
\end{equation}

\item
Polynomial inverse estimate \cite[Theorem 4.5.11]{brenner-scott:book}: 
Let $1 \leq q, \ell \leq \infty$ and let $s\geq 0$, then for any $E \in \Omega_h$
\begin{equation}
\label{eq:inverse}
\Vert  p_n \Vert_{W^{s,q}(E)} \lesssim h_E^{2/q - 2/ \ell -s} \Vert  p_n \Vert_{L^{\ell }(E)}
\quad \text{for any $p_n \in \Pk_n(E)$.}
\end{equation}

\end{itemize}

At the global level, given $n\in \mathbb{N}$, $m \in \R_+$, and $l \in [1,+\infty)$, we introduce the piecewise regular spaces
\begin{itemize}
\item $\Pk_n(\Omega_h) = \{q \in L^2(\Omega) \quad \text{s.t} \quad q|_{E} \in  \Pk_n(E) \quad \text{for all $E \in \Omega_h$}\}$,
\item $W^{m,l}(\Omega_h) = \{v \in L^l(\Omega) \quad \text{s.t} \quad v|_{E} \in  W^{m,l}(E) \quad \text{for all $E \in \Omega_h$}\}$,
\end{itemize}
equipped with the broken norm and seminorm
\begin{equation}
\label{eq:normebroken}
\begin{aligned}
\norm[W^{m,l}(\Omega_h)]{v}^l
&= \sum_{E \in \Omega_h} \norm[W^{m,l}(E)]{v}^l\,,
&\quad
|v|^l_{W^{m,l}(\Omega_h)} &= \sum_{E \in \Omega_h} |v|^l_{W^{m,l}(E)}\,, 
&\qquad &\text{if $1 \leq l < \infty$.}
\end{aligned}
\end{equation}
Further we define the operator $\Pi^0_{n}\colon L^2(\Omega) \to \Pk_n(\Omega_h)$ such that $\Pi^0_{n}\vert_E=\Pi^{0,E}_n$ for any $E \in \Omega_h$.

Let $k \geq 1$ be the polynomial order of the method.
We consider on each polygonal element $E \in \Omega_h$ the ``enhanced'' virtual space \cite{k1,BLV:2017,BLV:2018,vacca:2018}:
\begin{equation}
\label{eq:v-loc}
\begin{aligned}
\VDL = \biggl\{  
\b v_h \in [C^0(\overline{E})]^2 \,\, \text{s.t.} \,\,
(i)& \,\,  
  \boldsymbol{\Delta}    \b v_h  +  \nabla s \in \b x ^\perp \Pk_{k-1}(E), 
\,\,\text{ for some $s \in L_0^2(E)$,} 
\\
(ii)&  
\,\,   \divs \, \b v_h \in \Pk_{k-1}(E) \,, 
\\
(iii) &
\,\,  {\b v_h }_{|e}\cdot \b n_e \in \Pk_{\max\{2,k\}}(e)\,,
\,  {\b v_h }_{|e}\cdot \b t_e \in \Pk_k(e)
 \,\,\, \forall e \in \partial E, 
\\
(iv) &
\,\,   (\b v_h - \Pi^{\nabla,E}_k \b v_h, \, \b x ^\perp \, \widehat{p}_{k-1} )_E = 0
\,\,\, \text{$\forall \widehat{p}_{k-1} \in \widehat{\Pk}_{(k-1) \setminus (k-3)}(E)$}
\biggr\},
\end{aligned}
\end{equation}
where $\b x ^\perp = (x_2, -x_1)$.
Next, we summarize the main properties of the space $\VDL$.

\begin{itemize}
\item [\textbf{(P1)}] \textbf{Polynomial inclusion:} $[\Pk_k(E)]^2 \subseteq \VDL$;
\item [\textbf{(P2)}] \textbf{Degrees of freedom:}
the following linear operators $\mathbf{D_{\boldsymbol{U}}}$ constitute a set of DoFs for $\VDL$:
\begin{itemize}
\item[$\mathbf{D_{\boldsymbol{U}}1}$] the values of $\b v_h$ at the vertexes of the polygon $E$,
\item[$\mathbf{D_{\boldsymbol{U}}2}$] 
 the edge moments of $\b v_h$  for every edge $e \in \partial E$,
\[
\begin{aligned}
&\frac{1}{|e|} \int_e {\b v_h} \cdot {\boldsymbol{t}}_e m_\alpha \,{\rm d}s \,, 
&\qquad 
&\text{for any $m_\alpha$} \in \M_{k-2}(e) \, ,
\\
&\frac{1}{|e|} \int_e {\b v_h} \cdot {\boldsymbol{n}}_e m_\alpha \,{\rm d}s &\qquad 
&\text{for any $m_\alpha$} \in \M_{\max\{2,k\}-2}(e) \, ,
\end{aligned}
\]

\item[$\mathbf{D_{\boldsymbol{U}}3}$] the moments of $\b v_h$ 
$$
\frac{1}{|E|}\int_E  \b v_h \cdot \frac{m_{\boldsymbol{\alpha}}}{h_E}(x_2 - x_{2,E}, -x_1 + x_{1,E})  \, {\rm d}E  
\qquad 
\text{for any $m_{\boldsymbol{\alpha}} \in \M_{k-3}(E)$,}
$$

\item[$\mathbf{D_{\boldsymbol{U}}4}$] the moments of $\divs \b v_h$ 
$$
\frac{h_E}{|E|}\int_E (\divs \b v_h) \, m_{\boldsymbol{\alpha}} \, {\rm d}E  
\qquad \text{for any $m_{\boldsymbol{\alpha}} \in \M_{k-1}(E)$ with $|\boldsymbol{\alpha}| > 0$;}
$$
\end{itemize}
\item [\textbf{(P3)}] \textbf{Polynomial projections:}
the DoFs $\mathbf{D_{\boldsymbol U}}$ allow us to compute the following linear operators:
\[
%\PN \colon \VDL \to [\Pk_k(E)]^2, \qquad
\P0 \colon \VDL \to [\Pk_{k}(E)]^2, \qquad
\PP0 \colon \GRAD \VDL \to [\Pk_{k-1}(E)]^{2 \times 2} \,.
\]
\end{itemize} 

The global velocity space $\VDG= \{\b v_h \in C^0(\Omega) \quad \text{s.t.} \quad {\b v_h}|_E \in \VDL \quad \text{for all $E \in \Omega_h$} \}$ is defined by gluing the local spaces with the obvious associated sets of global DoFs.

The discrete pressure space $\QDG$ is given by the piecewise polynomial functions of degree $k-1$:
\begin{equation}
\label{eq:q-glo}
\QDG = \{q_h \in \QCG \quad \text{s.t.} \quad {q_h}_{|E} \in \Pk_{k-1}(E) \quad \text{for all $E \in \Omega_h$} \}\,.
\end{equation}
The couple of spaces $(\VDG, \, \QDG)$ is well known to be inf-sup stable in the classical Hilbertian setting \cite{BLV:2017,BLV:2018}. The inf-sup stability for $r > 2$ is proven below (Lemma \ref{inf-sup:vem}).
Let us introduce the discrete kernel
\begin{equation}
\label{eq:z-glo}
\ZDG = \{ \b v_h \in \VDG \quad \text{s.t.} \quad  b(\b v_h, q_h) = 0 \quad \text{for all $q_h \in \QDG$}\}
\end{equation}
and observe that, owing to $(ii)$ in \eqref{eq:v-loc} and \eqref{eq:q-glo},  $\divs \b v_h = 0$, for all $\b v_h \in \ZDG$.

%Uniquely for the purpose of defining our interpolant in the space $\VDG$, we consider also the alternative set of edge degrees of freedom (which can substitute $\mathbf{D_{\boldsymbol{U}}2}$)
%\begin{itemize}
%\item[$\mathbf{D_{\boldsymbol{U}}2'}$] the moments of $\b v_h$
%\begin{equation}
%\label{eq:per-infsup2}
%\frac{1}{|e|} \int_e {\b v_h} \cdot {\boldsymbol{t}}_e m_\alpha \,{\rm d}s \,, \qquad 
%\frac{1}{|e|} \int_e {\b v_h} \cdot {\boldsymbol{n}}_e m_\alpha \,{\rm d}s \qquad 
%\text{for any $m_\alpha$} \in \M_{k-2}(e) \, ,
%\end{equation}
%where $ {\boldsymbol{t}}_e$ and  ${\boldsymbol{n}}_e$ denote the tangent and the normal vectors to the edge $e$ respectively.
%\end{itemize}

Given any $\b v \in \b W^{s,p}(E)$, with $p \in (1,\infty)$ and $s \in {\mathbb R}_+$, $s>2/p$, we define its approximant 
${\b v_I} \in \VDG$ as the unique function in $\VDG$ that interpolates $\b v$ with respect to the DoF set  
$\mathbf{D_{\boldsymbol{U}}}$.
It is easy to check that, whenever $\nabla\cdot \b v = 0$, then $\b v_I \in \ZDG$.
Furthermore, the following approximation property is a trivial generalization of the results in \cite{MBM_23}. 

\begin{lemma}\label{lem:approx-interp}
Let $E \in \Omega_h$, $n \in \mathbb{N}$, $\ell  \in [1,\infty]$, $s \in {\mathbb R}_{+}$ and $\b v \in \b W^{s, \ell }(E)$. Then
$$
| \b v - \Pi^{0, E}_{n} \b v |_{\b W^{m, \ell }(E)} \lesssim h_E^{s-m} |\b v|_{\b W^{s, \ell }(E)} 
\quad \text{for $0 \le m \le s \le n+1$ .} \\
$$
Furthermore, given $\b v \in \b W^{s, 2}(E)$, $s>1$, let ${\b v_I} \in \VDG$ be the interpolant of $\b v$ defined above. It holds
$$
| \b v - \b v_I |_{\b W^{m,2}(E)} \lesssim 
h_E^{s-m} |\b v|_{\b W^{s,2}(E)}
\quad  \text{for $1  < s \le k+1 , \ m \in \{0,1\}$ .}
$$
The first bound above extends identically to the scalar and tensor-valued cases.
\end{lemma}

\begin{remark}\label{rem:nonincl}
We must note that, by known regularity results on Lipschitz domains, definition \eqref{eq:v-loc} guarantees $\VDL \subset \b W^{1,r}(E)$ for all $r \in (1,\infty)$ only if the polygonal element $E$ is convex \cite{Dauge:89, Kellogg.Osborn:76} or has a small Lipschitz constant \cite{Galdi-Simader-Sohr_1994}. Otherwise, if $d=2$ and $r > 4$ or $d=3$ and $r > 3$, such inclusion does not hold (see, e.g., \cite{Geng-Kilty_2015, Savare:1998}). This is the reason why, in the following analysis, we make use of discrete norms (see definition \eqref{eq:discrete_norm} below). 
Additionally, assuming that all the mesh elements are convex, the \emph{a priori} estimates established in Section~\ref{sec:error_analysis} below imply error bounds with respect to standard $\b W^ {1,r}$-norms (see Corollary~\ref{cor:W1r.estimate} below). 
%Indeed, in this case, due to the elliptic regularity results established in \cite{Kellogg.Osborn:76} together with Sobolev embedding theorems, we would have $\VDG(E) \subset \b W^{2,2}(E) \subset \b W^{1,r}(E)$ for all $E \in \Omega_h$ and $r \in [2,\infty)$, which combined with the global continuity of $\VDG$ implies $\VDG \subset \b W^{1,r}(\Omega)$.
\end{remark}

\subsection{Discrete problem}
\label{sub:vem problem}

To formulate the VEM approximation of \eqref{eq:stokes.weak}, we aim to construct a discrete version of the nonlinear form $a(\cdot, \cdot)$ in \eqref{eq:a.b} and the approximation of the loading term $\b f$. For the latter, we define the discrete volumetric force as
\begin{equation}
\label{eq:forma fh}
\b f_h= \Pi^0_k \b f \,,
\end{equation}
owing to the fact that $(\b f_h, \b v_h)$ is computable by property \textbf{(P3)}.
The discrete nonlinear form $a_h(\cdot,\cdot)$ is given as the sum of a consistency term and a stabilizing form that is suited for the non-linearity under consideration. Following \cite{Antonietti_et_al_2024},  we consider a non-linear \texttt{dofi-dofi} stabilization $S(\cdot,\cdot) \colon \VDG \times \VDG \to \R$ defined as 
\begin{equation}
\label{eq:Sglobal}    
S(\b v_h, \b w_h)=\sum_{E\in\Omega_h} S^E(\b v_h, \b w_h) \qquad \text{for all $\b v_h$, $\b w_h \in \VDG$,}
\end{equation} 
with $S^E(\cdot, \cdot) \colon \VDL \times \VDL \to \R$
resembling the nonlinear law in \eqref{eq:Carreau}, i.e.
\begin{equation}
\label{eq:dofi}
\begin{aligned}
S^E(\b v_h,\w_h) & = \overline{\mu}_E \, \big( \delta^\alpha + h_E^{-\alpha}\vert \DOF(\b v_h) \vert^\alpha \big)^{\frac{r-2}\alpha}  \DOF (\b v_h) \cdot \DOF (\w_h)\,,
\end{aligned}
\end{equation}
where $\overline{\mu}_E = \Pi^{0,E}_0 \mu$ and $\DOF \colon \VDL \to {\mathbb R}^{N_E}$, with $N_E$ denoting the dimension of $\VDL$, is the function that associates to each $\b v_h \in \VDL$ the vector of the local degrees of freedom in {\bf (P2)}. 
We remark that, according to the strong monotonicity and H\"older continuity of the Carreau--Yasuda law, any choice of $\alpha\in (0,\infty)$ in \eqref{eq:dofi} give rise to an equivalent stabilization, in the sense that
\begin{equation}\label{eq:equi.stab}
\overline{\mu}_E \, \big( \delta^r + h_E^{-r}\vert \DOF(\b v_h) \vert^r \big)^{\frac{r-2}r}  \vert\DOF (\b v_h) \vert^2
\lesssim S^E(\b v_h,\b v_h) \lesssim
\overline{\mu}_E \, \big( \delta^r + h_E^{-r}\vert \DOF(\b v_h) \vert^r \big)^{\frac{r-2}r}  \vert\DOF (\b v_h) \vert^2.
\end{equation}
\begin{comment}
\red{
\textbf{Andre: suggestion of the generalization of the above stabilization function:}
\begin{equation}
\label{eq:dofi2}
\begin{aligned}
S^E(\b v_h,\w_h) & = \sigma_E(\cdot,\DOF (\b v_h)) \cdot \DOF (\w_h)\,,
\end{aligned}
\end{equation}
where $\sigma_E : E \times {\mathbb R}^{N_E} \to {\mathbb R}^{N_E}$ satisfies for all $\b v \in \mathbb R^{N_E}$, a.e. in $E$
\begin{equation}
    \sigma_E(\cdot,\b{v}) \simeq \big(\delta + h_E^{-1}\vert \b{v} \vert\big)^{r-2}\b{v},
\end{equation}
and where the hidden constant depends only on $\sigma_{\rm c},\sigma_{\rm m}$.
}
\end{comment}
%
Thus, we define the global form $a: \VDG \times \VDG \to \R$ such that
\begin{equation}
\label{eq:forma ah}
a_h(\b v_h, \b w_h) = 
\int_\Omega \b \sigma(\cdot, \b \Pi^0_{k-1} \b \epsilon(\b v_h)) : 
\b \Pi^0_{k-1} \b \epsilon(\b \w_h) + 
S((I - \Pi^0_k) \b v_h, (I - \Pi^0_k) \b w_h)
\quad \forall\b v_h,\b w_h \in \VDG.
\end{equation}
%\[
%a_h^E(\b v_h,  \b w_h) = 
%\int_E \b\sigma(\cdot, \PP0 \b \epsilon(\b v_h)) : \PP0 \b\epsilon(\b w_h)
%+ S^E((I - \P0 ) \b v_h, \, (I - \P0 ) \b w_h) \,,
%\]
We observe that, owing to property \textbf{(P3)}, all projection operators appearing above are computable explicitly in terms of the velocity DOFs.

The virtual element discretization of Problem \eqref{eq:stokes.weak} is given by:
Find $(\b u_h, p_h) \in \VDG \times \QDG$ such that
\begin{equation}\label{eq:stokes.vem}
  \begin{aligned}
     a_h(\b u_h,\b v_h)+b(\b v_h, p_h) &= \int_\Omega \b f_h \cdot \b v_h + \int_{\Gamma_N} \b g \cdot \b v_h &\qquad \forall \b v_h \in \VDG , \\
     b(\b u_h, q_h) &= 0 &\qquad \forall q_h \in \QDG \,. 
  \end{aligned}
\end{equation}
Recalling the definition of the discrete kernel $\ZDG$ in \eqref{eq:z-glo}, the previous problem can also be written in the kernel formulation:
Find $\b u_h \in \b \ZDG$ such that
\begin{equation}\label{eq:stokes.vem.Z}
     a_h(\b u_h, \b v_h) = \int_\Omega \b f_h \cdot \b v_h + \int_{\Gamma_N} \b g \cdot \b v_h \qquad \forall \b v_h \in \b \ZDG \,.
\end{equation}

\subsection{Well-posedness}
This section aims to establish the well-posedness of the discrete problem \eqref{eq:stokes.vem}. To do so, first we prove the inf-sup stability of the bilinear form $b(\cdot,\cdot)$ and then we investigate the continuity and monotonicity properties of $a_h(\cdot,\cdot)$.

We define, for all $\b v \in \b W^{1,r}(\Omega) \cup \VDG$ with $r \in [2,\infty)$, the discrete quantity:
\begin{equation}
\label{eq:discrete_norm}
\tri \b v \tri_{r}^r \coloneq 
\|\b\Pi^0_{k-1}\b \epsilon(\b v)\|_{\mathbb{L}^r(\Omega)}^r
+\sum_{E \in \Omega_h} h_E^{2-r} | \DOF((I-\Pi^0_{k})\b v) |^r.
%+S((I-\Pi^0_{k})\b v,(I-\Pi^0_{k})\b v).
\end{equation}
%where the real $\delta \ge 0$ refers to the parameter $\delta$ of the model that influences the stabilization part. We also make use of the shorthand notation $\tri \cdot \tri_{r} = \tri \cdot \tri_{0,r}$. 
Note that, owing to Lemma \ref{Lem:vemstab-2} below, $\tri \cdot \tri_{r}$ defines a norm on $\VDG$. In Section~\ref{sec:error_analysis}, we will measure the discretization error with respect to the quantity
\begin{equation}
\label{eq:error_measure}
\tri \b v \tri_{\delta,r}^r \coloneq 
\|\b\Pi^0_{k-1}\b \epsilon(\b v)\|_{\mathbb{L}^r(\Omega)}^r
+S((I-\Pi^0_{k})\b v,(I-\Pi^0_{k})\b v),
\end{equation}
which is not a norm since absolute homogeneity does not hold due to the dependence on $\delta$ of the stabilization term. However, for all $\delta\ge0$ and $r\ge 2$, the error measure in \eqref{eq:error_measure} controls the discrete norm. This is established in the next Lemma.
\begin{lemma}
Given $\delta\ge 0$, the discrete norm defined in \eqref{eq:discrete_norm}, and the error measure as in \eqref{eq:error_measure}, the following inequalities hold:
\begin{equation}\label{eq:comparison_delta_r_norm}
\tri \b v \tri_{r} \le \tri \b v \tri_{\delta,r} \lesssim (\delta+\tri \b v \tri_{r})^\frac{r-2}{r}\tri \b v \tri_{r}^\frac{2}{r}.
\end{equation} 
\end{lemma}

%\textbf{\color{red}Andre: we have the above right-hand side essentially with the $(\frac{r}{r-2},\frac{r}{2})$-Holder inequality.}
\begin{proof}
The first inequality in \eqref{eq:comparison_delta_r_norm} is a direct consequence of the definition of the stabilization function \eqref{eq:dofi} and the fact that $r\ge2$. To prove the second inequality, we set ${\b v}^\perp =(I - \Pi^0_k) \b v$, use the equivalence property \eqref{eq:equi.stab}, recall that $h_E^2 \simeq |E|$, and apply the discrete $(\frac{r}{r-2},\frac{r}{2})$-H\"older inequality, to infer
\[
\begin{aligned}
    S({\b v}^\perp,{\b v}^\perp) &\lesssim \sum_{E\in\Omega_h}\big( \delta^r + h_E^{-r}\vert \DOF({\b v}^\perp) \vert^r \big)^{\frac{r-2}{r}}  |\DOF ({\b v}^\perp)|^2 \\
    &\lesssim \sum_{E\in\Omega_h}\big( |E|\delta^r + h_E^{2-r}\vert \DOF({\b v}^\perp) \vert^r \big)^{\frac{r-2}{r}} (h_E^{2-r}|\DOF ({\b v}^\perp)|^r)^\frac{2}{r}\\
    &\le \big(|\Omega|\delta^r + \sum_{E\in\Omega_h}h_E^{2-r}\vert \DOF({\b v}^\perp) \vert^r \big)^{\frac{r-2}{r}} (\sum_{E\in\Omega_h}h_E^{2-r}|\DOF ({\b v}^\perp)|^r)^\frac{2}{r}
    \lesssim  (\delta+\tri \b v \tri_{r})^{r-2}\tri \b v \tri_{r}^{2}.
\end{aligned}
\]
Moreover, owing to $\delta\ge0$ and $r\ge2$, one also has
$\|\b\Pi^0_{k-1}\b \epsilon(\b v)\|_{\mathbb{L}^r(\Omega)}^r 
%=\|\b\Pi^0_{k-1}\b \epsilon(\b v)\|_{\mathbb{L}^r(\Omega)}^{r-2}\|\b\Pi^0_{k-1}\b \epsilon(\b v)\|_{\mathbb{L}^r(\Omega)}^2 
\le (\delta+\tri \b v \tri_{r})^{r-2}\tri \b v \tri_{r}^2$.
\end{proof}

%----------------------------------------------------------------------
\subsubsection{Discrete inf-sup condition}\label{sec:infsup}
\def\ww{{\bf w}}
We here prove a discrete inf-sup condition analogous to the continuous one. The difference with respect to the analogous condition proved in \cite{Antonietti_et_al_2024} is that here we make use of the discrete norm in \eqref{eq:discrete_norm}.

We recall the following Lemma established in \cite{MBM_23} (see also \cite[Lemmas 8, 9]{Antonietti_et_al_2024}).
%\begin{lemma}
%Let the mesh regularity assumptions stated in Assumption~\ref{ass:mesh} hold. For any $E \in \mesh$ and for all $\b v_h \in \VDL$ we have
%$$
%S^E(\b v_h,\b v_h) \simeq 
%h_E^{2-r} | \DOF(\b v_h) |^r \simeq 
%h_E^{2-r} \max_{1 \le i \le \NE} | \DOFi (\b v_h) |^r \,.
%$$
%\end{lemma}
%\begin{proof}
%    See \cite{Antonietti_et_al_2024}.
%\end{proof}
\begin{lemma}\label{Lem:vemstab-2}
%\label{Lem:vemstab-2}
Let the mesh regularity in Assumption~\ref{ass:mesh} hold.
For any $E \in \mesh$ we have
$$ 
|\b v_h|_{\b W^{1,2}(E)} 
\lesssim
| \DOF(\b v_h) | 
\lesssim 
|\b v_h|_{\b W^{1,2}(E)} 
\qquad \text{for all $\b v_h \in \VDL$ s.t. $\ \Proj\b v_h=0$.}
$$
\end{lemma}

\begin{lemma}[Discrete inf-sup]\label{inf-sup:vem}
Let the mesh regularity assumptions stated in Assumption~\ref{ass:mesh} hold. Then, for any $r\in [2,\infty)$ it exists a constant $\overline{\beta}(r)$, such that
$$
\inf_{q_h\in \QDG}\;\sup_{\b w_h\in \VDG}\; 
\frac{b(\b w_h,q_h)}{\norm[L^{r'}(\Omega)]{q_h} \tri {\b w_h} \tri_r} 
\ge \overline{\beta}(r) > 0.
$$
\end{lemma}
\begin{proof}
The proof is a modification of that of Lemma 16 in \cite{Antonietti_et_al_2024}; to shorten the exposition, we refer here directly to the notation and derivations in that article. We consider the same Fortin operator $\Pi^{\cal F}$ introduced in the lemma above. The only difference in the present proof is that when showing the continuity of $\Pi^{\cal F}$ from $W^{1,r}(\Omega)$ into $\VDG$, the latter space is equipped with the $\tri\cdot\tri_r$ norm instead of the $W^{1,r}$ norm.

%We start by recalling that, for all $E \in \Omega_h$ and ${\bf v}_h \in \VDL$ with $\P0 {\bf v}_h = 0$, it holds (see \cite{MBM_23})
%\begin{equation}\label{New:equiv-dofs} 
%|{\bf v}_h|_{H^1(E)} \lesssim |\DOF({\bf v}_h)| \lesssim |{\bf v}_h|_{H^1(E)} \, ,
%\end{equation}
%with hidden constants independent of the element $E \in \{ \Omega_h \}_h$.
%%
Let now $\b w \in W^{1,r}(\Omega)$. First, by definition of the stabilization form and some trivial algebra, then combining a standard inverse estimate for polynomials with the continuity of the $\Pi^0_{k-1}$ operator and recalling  Lemma \ref{Lem:vemstab-2}, we obtain
$$
\begin{aligned}
\tri \Pi^{\cal F} \b w \tri_r^r & =  
\sum_{E \in \Omega_h} \Big( \|\b\Pi^0_{k-1}\b \epsilon(\Pi^{\cal F} \b w) \|_{\mathbb{L}^r(E)}^r
+   h_E^{2-r} |\DOF (I-\Pi^0_{k})\Pi^{\cal F} \b w|^r \Big)
\\
& \lesssim \sum_{E \in \Omega_h} h_E^{2-r} \Big(
\| \b \epsilon(\Pi^{\cal F} \b w) \|_{\mathbb{L}^2(E)}^r
+ |(I-\Pi^0_{k})\Pi^{\cal F}\b w|_{W^{1,2}(E)}^r \Big) \, .
\end{aligned}
$$
Using the continuity of $\Pi^0_{k}$, the ``local'' continuity of $\Pi^{\cal F}$ in $W^{1,2}$ (see for instance Lemma 16 in \cite{Antonietti_et_al_2024}) and a H\"older inequality we get
$$
\tri \Pi^{\cal F} \b w \tri_r^r 
\lesssim \sum_{E \in \Omega_h} h_E^{2-r} | \Pi^{\cal F} \b w |_{W^{1,2}(E)}^r \lesssim \sum_{E \in \Omega_h} h_E^{2-r} | \b w |_{W^{1,2}(\omega_E)}^r
\lesssim \sum_{E \in \Omega_h} | \b w |_{W^{1,r}(\omega_E)}^r \, ,
$$
where $\omega_E$ is the union of all elements sharing at least a vertex with $E$. The above bound shows the required continuity for the $\Pi^{\cal F}$ operator.
\end{proof}

\subsubsection{Properties of the stabilization}\label{sec:stab}
We establish the continuity and monotonicity of the local stabilization form.
\begin{lemma}[H\"older continuity and strong monotonicity of $S^E(\cdot,\cdot)$]
\label{Lemma:S_holder_monotonicity}
    Let the mesh regularity assumptions stated in Assumption~\ref{ass:mesh} hold.
    Let $\b u_h$, $\b w_h \in \VDL$  and set $\b e_h=\b u_h - \b w_h$. Then, for all $\b v_h \in \VDL$, there holds
    \begin{equation}
    \label{eq:continuity:local_stab1}
        \vert S^E(\b u_h,\b v_h)-S^E(\w_h,\b v_h)\vert
        \lesssim \left(\delta^r + h_E^{-r}\vert \DOF(\b u_h) \vert^{r} 
      + h_E^{-r}\vert \DOF(\b w_h) \vert^{r}\right)^{\frac{r-2}{r}} \vert\DOF (\b e_h) \vert \,
      \vert\DOF(\b v_h)\vert.
    \end{equation}
Moreover, there holds
\begin{equation}\label{eq:monotonicity:local_stab}
    S^E(\b u_h,\b e_h)-S^E(\b w_h,\b e_h) \gtrsim S^E(\b e_h,\b e_h) \gtrsim h_E^{2-r}\vert \DOF(\b e_h) \vert^{r}.
\end{equation}
    
\end{lemma}
\begin{proof}
    {\em (i) H\"older continuity}. Recalling the definition of $S^E(\cdot,\cdot)$ in \eqref{eq:dofi} and employing \eqref{eq:hypo.continuity}, bound \eqref{eq:continuity:local_stab1} is derived as follows
    \begin{equation*}
    \begin{aligned}
    \vert S^E(\b u_h,\b v_h)&-S^E(\b w_h,\b v_h)\vert\\
    &\lesssim \left\vert (\delta^\alpha + h_E^{-\alpha}\vert \DOF(\b u_h) \vert^{\alpha})^{\frac{r-2}\alpha} \DOF(\b u_h) 
      - (\delta^\alpha + h_E^{-\alpha}\vert \DOF(\b w_h) \vert^{\alpha})^{\frac{r-2}\alpha} \DOF (\b w_h)\right\vert 
      \vert \DOF(\b v_h) \vert
      \\
     &\lesssim
      \left(\delta^r + h_E^{-r}\vert \DOF(\b u_h) \vert^{r} 
      + h_E^{-r}\vert \DOF(\b w_h) \vert^{r}\right)^{\frac{r-2}{r}} \vert\DOF (\b e_h) \vert \,
      \vert\DOF(\b v_h)\vert  \,.
    \end{aligned}
    \end{equation*}

     {\em (ii) Strong monotonicity.}
    We recall the strong monotonicity bound stating that, for all $\b x,\b y \in\mathbb{R}^n$, it holds
   \begin{equation}\label{eq:bound_hyp_model_alpha}
       \vert \b x- \b y\vert^2 (\delta^r+ \vert \b x\vert^r+\vert \b y \vert^r)^{\frac{r-2}r} \lesssim
       \left\{ 
       (\delta^\alpha + \vert \b x\vert^\alpha)^{\frac{r-2}{\alpha}}\b x - 
       (\delta^\alpha + \vert \b y\vert^\alpha)^{\frac{r-2}{\alpha}}\b y
       \right\} \cdot (\b x -\b y).
   \end{equation}
  Inserting $\delta^\alpha \ge 0$, and using a triangle inequality together with the fact that $(x+y)^\frac\alpha r \lesssim x^\frac\alpha r+y^\frac\alpha r$ for all $x,y \in [0,\infty)$, and employing \eqref{eq:bound_hyp_model_alpha} with 
$\b x= \DOF(\b u_h)$,  $\b y= \DOF(\b w_h) $, and recalling the definition of $S^E(\cdot, \cdot)$ in \eqref{eq:dofi} we infer
   \begin{equation}
   \label{eq:strong1}
   \begin{aligned}
   h_E^{2-r}\vert \DOF(\b e_h)\vert^r 
&\lesssim
       \vert \DOF(\b e_h)\vert^2 (\delta^\alpha+h_E^{-\alpha}\vert \DOF(\b e_h)\vert^\alpha)^{\frac{r-2}{\alpha}} \simeq S^E(\b e_h,\b e_h) \\   
       &\lesssim
       \vert \DOF(\b e_h)\vert^2 (\delta^r+h_E^{-r}\vert \DOF(\b u_h)\vert^r+h_E^{-r}\vert \DOF(\b w_h) \vert^r)^{\frac{r-2}r} \\
       &\lesssim 
       \left(
       (\delta^\alpha+h_E^{-\alpha}\vert \DOF(\b u_h)\vert^\alpha)^{\frac{r-2}{\alpha}} \DOF(\b u_h) - 
       (\delta^\alpha+h_E^{-\alpha}\vert \DOF(\b w_h)\vert^\alpha)^{\frac{r-2}{\alpha}}\DOF(\b w_h)
       \right) \cdot \DOF(\b e_h)
       \\
       &\lesssim 
       S^E(\b u_h, \b e_h) - S^E(\b w_h, \b e_h)  \,.
  \end{aligned}     
  \end{equation}

\end{proof}

Next, we prove the following result that be useful in the sequel.
\begin{proposition}
For all $\b u_h$, $\b w_h \in \VDG$,
   \begin{equation}
\label{eq:S_ineq_1}
\begin{aligned}
    |S( {\b u_h}, {\b w_h})| \lesssim \big(\delta^r+S( {\b u}_h,  {\b u}_h)\big)^\frac{r-2}{2r}S( {\b u}_h,  {\b u}_h)^\frac{1}{2}S( {\b w}_h,  {\b w}_h)^\frac{1}{r}\,.
    \end{aligned}
    \end{equation}  
\end{proposition}

\begin{proof}
Employing the discrete 3-terms $(\frac{2r}{r-2},2,r)$-H\"older inequality together with $h_E^2 \simeq |E|$, also recalling \eqref{eq:monotonicity:local_stab}, we obtain
    \begin{equation}
\label{eq:pressT4:1}
\begin{aligned}
    &|S({\b u_h}, {\b w_h})|
    \\&\lesssim \sum_{E\in \Omega_h}  \big(\delta^r +h_E^{-r}|\DOF({\b u}_h)|^r\big)^\frac{r-2}{r} \vert\DOF ({\b u}_h) \vert \, \vert \DOF ({\b w}_h)\vert
    \\
    &= \sum_{E\in \Omega_h}  \big(h_E^2\delta^r+h_E^{2-r}|\DOF({\b u}_h)|^r\big)^\frac{r-2}{2r}\big(\big(\delta^r+h_E^{-r}|\DOF({\b u}_h)|^r\big)^\frac{r-2}{r} \vert\DOF ({\b u}_h) \vert^2)^\frac{1}{2} \, (h_E^{2-r}\vert \DOF ({\b w}_h)\vert^r)^\frac{1}{r}
    \\
    &\lesssim \left(\sum_{E\in \Omega_h}  \big(|E|\delta^r +h_E^{2-r}|\DOF({\b u}_h)|^r\big)\right)^\frac{r-2}{2r}\left(\sum_{E\in \Omega_h}  \big(\delta^r +h_E^{-r}|\DOF({\b u}_h)|^r\big)^\frac{r-2}{r} \vert\DOF ({\b u}_h) \vert^2\right)^\frac{1}{2} \, \left(\sum_{E\in \Omega_h}h_E^{2-r}\vert \DOF ({\b w}_h)\vert^r\right)^\frac{1}{r}
    \\
    &\lesssim \big(\delta^r+S( {\b u}_h,  {\b u}_h)\big)^\frac{r-2}{2r}S( {\b u}_h,  {\b u}_h)^\frac{1}{2}S( {\b w}_h,  {\b w}_h)^\frac{1}{r}\,.
    \end{aligned}
    \end{equation}
\end{proof}

\subsubsection{Properties of the discrete viscous term}\label{sec:viscous}

Hinging on the results of the previous section, we establish the continuity and monotonicity properties of the discrete viscous function $a_h(\cdot, \cdot)$.

\begin{proposition}[H\"older continuity and strong monotonicity of $a_h(\cdot,\cdot)$] 
\label{lem:holder_monotonicity_ah}
Let $\b u_h$, $\b w_h \in \VDG$ and set $\b e_h=\b u_h - \b w_h \in \VDG$. Then for any $\b v_h \in \VDG$ there holds
\begin{equation}\label{eq:ah_ineq_1}
\begin{aligned}
    |a_h(\b u_h, \b v_h) - a_h(\b w_h, \b v_h)| & \lesssim
    \left(\delta^r 
+ \tri \b u_h \tri_{r}^r + \tri \b w_h \tri_{r}^r  \right)^\frac{r-2}{r} 
\tri \b e_h \tri_{r} \tri \b v_h \tri_{r}.
\end{aligned}
\end{equation}
Moreover, it holds
\begin{equation}\label{eq:ah_ineq_2}
\begin{aligned}
    a_h(\b u_h, \b e_h) - a_h(\b w_h, \b e_h) &  \gtrsim \tri \b e_h \tri_{\delta,r}^r \gtrsim
    \tri \b e_h \tri_{r}^r.
\end{aligned}
\end{equation}
\end{proposition}
\begin{proof}
For the sake of conciseness, for all $\b v_h\in\b U_h$, we let $ {\b v}_h^\perp\coloneq(I - \Pi^0_k) \b v_h$.\\
{\em (i) H\"older continuity.}   Recalling the definition of $a_h(\cdot, \cdot)$ in \eqref{eq:forma ah} we can write
\begin{equation}
    \label{eq:lemmaholdh-0}
    |a_h(\b u_h, \b v_h) - a_h(\b w_h, \b v_h)| \le |T_1| + |T_2| 
\end{equation}
where
$$
    \begin{aligned}
    T_1 &: =
    \int_\Omega \big(\b\sigma(\cdot, \b\Pi^0_{k-1} \b \epsilon(\b u_h)) -  \b\sigma(\cdot, \b\Pi^0_{k-1} \b \epsilon(\b w_h))\big): \b\Pi^0_{k-1} \b\epsilon(\b v_h)   \,,
\\
T_2 &= S({\b u}_h^\perp, {\b v}_h^\perp) -
    S({\b w}_h^\perp,  {\b v}_h^\perp) \,.
    \end{aligned}
$$    
Following the lines of \eqref{eq:for-holder} and recalling the definition of the discrete norm in \eqref{eq:discrete_norm} it is inferred that
\begin{equation}
\label{eq:holder.ahT1}
\begin{aligned}
|T_1| &\le
\int_\Omega \left|\big(\b\sigma(\cdot, \b\Pi^0_{k-1} \b \epsilon(\b u_h)) -  \b\sigma(\cdot, \b\Pi^0_{k-1} \b \epsilon(\b w_h))\big): \b\Pi^0_{k-1} \b\epsilon(\b v_h)  \right| \\
&\lesssim \left(\int_\Omega \delta^r + |\b\Pi^0_{k-1} \b \epsilon(\b u_h))|^r + |\b\Pi^0_{k-1} \b \epsilon(\b w_h))|^r \right)^\frac{r-2}{r} 
\norm[\mathbb{L}^r(\Omega)]{\b\Pi^0_{k-1} \b \epsilon(\b e_h))} \norm[\mathbb{L}^r(\Omega)]{\b\Pi^0_{k-1} \b \epsilon(\b v_h))}\\
&\lesssim \left(|\Omega| \delta^r 
+ \tri \b u_h \tri_{r}^r + \tri \b w_h \tri_{r}^r  \right)^\frac{r-2}{r} 
\tri \b e_h \tri_{r} \tri \b v_h \tri_{r}.
\end{aligned}
\end{equation}
% 
%so that $S({\b v}_h^\perp, {\b v}_h^\perp)\le \tri \b v_h \tri_{r}^r$.
%%
%%
Then, applying \eqref{eq:continuity:local_stab1} together with a discrete H\"older inequality with exponents $(\frac{r}{r-2},r, r)$, we infer
\begin{equation}
    \label{eq:lemmaholdh-2}
\begin{aligned}
    |T_2| &\lesssim \sum_{E\in\Omega_h} 
    \vert S^E({\b u}_h^\perp, {\b v}_h^\perp) -
    S^E({\b w}_h^\perp,  {\b v}_h^\perp) \vert
    \\ &\lesssim \left(\sum_{E\in\Omega_h} \left(h_E^{2}
    \delta^r+h_E^{2-r}\vert \DOF({\b u}_h^\perp) \vert^{r} 
    +h_E^{2-r}\vert\DOF({\b w}_h^\perp)\vert^{r}\right)
    \right)^{\frac{r-2}{r}} 
    \left(\sum_{E\in\Omega_h} h_E^{2-r}\vert\DOF ({\b e}_h^\perp) \vert^r\right)^{\frac1r} 
    \left(\sum_{E\in\Omega_h} h_E^{2-r} \vert\DOF({\b v}_h^\perp)\vert^r\right)^{\frac1r}
    %\\ &\lesssim
    %\left(\left(\sum_{E\in\Omega_h} 
    %h_E^2\right)\delta^r+S(\b u_h^\perp, \b u_h^\perp) 
    %+S(\b w_h^\perp, \b w_h^\perp)\right)^{\frac{r-2}{r}} 
    %S(\b e_h^\perp, \b e_h^\perp )^{\frac1r} 
    %S(\b v_h^\perp, \b v_h^\perp)^{\frac1r}
    \\ &\lesssim
    \left(|\Omega|\delta^r+\tri \b u_h \tri_{r}^r + \tri \b w_h \tri_{r}^r\right)^{\frac{r-2}{r}} 
    \tri \b e_h \tri_{r} \tri \b v_h \tri_{r},
\end{aligned}
\end{equation}
where we used the fact that $h_E^2 \simeq |E|$. The proof follows inserting \eqref{eq:holder.ahT1} and \eqref{eq:lemmaholdh-2} in \eqref{eq:lemmaholdh-0}.

\noindent  {\em (ii) Strong monotonicity.} Recalling \eqref{eq:hypo.monotonicity} and using the inequality in \eqref{eq:pre_holdercont}, we get
$$
\begin{aligned}
\|\b\Pi^0_{k-1}\b \epsilon(\b e_h)\|_{\mathbb{L}^r(\Omega)}^r
&\le
\int_\Omega 
(\delta+\vert\b\Pi^0_{k-1} \b\epsilon(\b u_h-\b w_h) \vert)^{r-2}\vert\b\Pi^0_{k-1} \b\epsilon(\b u_h-\b w_h) \vert^2 
\\
&\lesssim
\int_\Omega 
(\b\sigma(\cdot,\b\Pi^0_{k-1}\b\epsilon(\b u_h)) - \b\sigma(\cdot,\b\Pi^0_{k-1}\b\epsilon(\b w_h)) ) : \b\Pi^0_{k-1} \b\epsilon(\b e_h). 
%\\&\lesssim
%a_h(\b u_h, \b e_h) - a_h(\b w_h, \b e_h).
\end{aligned}
$$
Moreover, from the first inequality in \eqref{eq:monotonicity:local_stab} it is readily inferred that
$$
\begin{aligned}
S({\b e}_h^\perp,{\b e}_h^\perp)
&\lesssim 
S({\b u}_h^\perp,{\b e}_h^\perp) -
    S({\b w}_h^\perp,{\b e}_h^\perp)
\end{aligned}
$$
The assertion follows by summing the previous bounds, recalling the definition of the  $\tri \cdot \tri_{\delta,r}$-norm and using \eqref{eq:comparison_delta_r_norm}. 
\end{proof}

\subsubsection{Main results}\label{sec:discrete_wellpose}

We are now ready to prove the well-posedness of the discrete Virtual Element problem \eqref{eq:stokes.vem}.

%%%%%%%
\begin{theorem}[Existence and uniqueness]
    For any $r\in[2,\infty)$, there exists a unique solution $\b u_h \in \ZDG$ to the discrete problem \eqref{eq:stokes.vem.Z}.
\end{theorem}
\begin{proof}
\textit{(i) Existence.} 
Let the mesh $\Omega_h$ be fixed. We equip the space $\b \ZDG$ with the $(\cdot,\cdot)_{\b W^{1,2}}$-inner product and induced norm $\norm[\b W^{1,2}(\Omega)]{\cdot}$. 
Owing to the equivalence of norms in finite-dimensional spaces, we have 
\begin{equation}\label{eq:inv_wellp}
    \norm[\b W^{1,2}(\Omega)]{\b v_h} \lesssim
    C_h\tri\b v_h\tri_{r},
\end{equation}
with the positive constant $C_h$ depending on the mesh size $h$.
We also define the nonlinear function $\b \Phi_h:\b \ZDG\to\b \ZDG$ such that
$$
(\b \Phi_h(\b v_h), \b w_h)_{\b W^{1,2}(\Omega)} \coloneq a_h(\b v_h, \b w_h), 
\qquad\forall\ \b v_h, \b w_h\in \b \ZDG.
$$
The strong monotonicity of $a_h(\cdot, \cdot)$ established in Lemma \ref{lem:holder_monotonicity_ah} together with \eqref{eq:inv_wellp} leads to, for any $\b v_h \in \b \ZDG$, 
$$
\begin{aligned}
\lim_{\norm[\b W^{1,2}(\Omega)]{\b v_h}\to\ \infty} 
\frac{(\b \Phi_h(\b v_h), \b v_h)_{\b W^{1,2}(\Omega)}}{\norm[\b W^{1,2}(\Omega)]{\b v_h}}
&\gtrsim
\lim_{\norm[\b W^{1,2}(\Omega)]{\b v_h}\to\ \infty} 
\frac{\tri \b v_h\tri_{r}^r}{\norm[\b W^{1,2}(\Omega)]{\b v_h}}
\gtrsim C_h^{-1}
\lim_{\norm[\b W^{1,2}(\Omega)]{\b v_h}\to\ \infty} 
\tri \b v_h\tri_{r}^{r-1} \to \infty.
\end{aligned}
$$
By applying \cite[Theorem 3.3]{Deimling:85}, the previous result shows that the operator $\b \Phi_h$ is onto. As a result, there exists $\b u_h\in \b \ZDG$ such that $\b \Phi_h(\b u_h) = \b z_h$, with
$\b z_h\in \b \ZDG$ defined such that 
$$(\b z_h, \b w_h)_{\b W^{1,2}(\Omega)} = \int_\Omega \b f_h \cdot \b w_h + \int_{\Gamma_N} \b g\cdot \b w_h \qquad \forall \b w_h \in \ZDG \, .
$$
Thanks to the definition of $\b \Phi_h$, this implies that $\b u_h$ is a solution to the discrete problem \eqref{eq:stokes.vem.Z}.

\medskip\noindent 
\textit{(ii) Uniqueness.}
Let $\b u_{h,1}, \b u_{h,2}\in \b \ZDG$ solve \eqref{eq:stokes.vem.Z}.
Subtracting \eqref{eq:stokes.vem.Z} for $\b u_{h,2}$ from \eqref{eq:stokes.vem.Z} for $\b u_{h,1}$ and then taking $\b v_h = \b u_{h,1} - \b u_{h,2}$ as test function, we obtain 
$$
a_h(\b u_{h,1},\b u_{h,1} - \b u_{h,2}) - a_h(\b u_{h,2},\b u_{h,1} - \b u_{h,2}) = 0.
$$
Hence, using again the strong monotonicity of $a_h(\cdot, \cdot)$ with $\b e_h = \b u_{h,1} - \b u_{h,2}$, we get $\tri \b u_{h,1} - \b u_{h,2}\tri_{r}^r = 0$, that implies $\b u_{h,1}=\b u_{h,2}$.
\end{proof}
The next result is derived by using the discrete inf-sup condition established in Lemma \ref{inf-sup:vem} and the equivalence of the discrete problems \eqref{eq:stokes.vem} and \eqref{eq:stokes.vem.Z}.
\begin{corollary}[Well-posedness of \eqref{eq:stokes.vem}]
For any $r\in[2,\infty)$, there exists a unique solution $(\b u_h, p_h) \in \b \VDG \times \QDG$ to the discrete problem \eqref{eq:stokes.vem}. 
\end{corollary}
\begin{remark}[Stability estimates] 
\emph{A priori} estimates for the unique discrete velocity and pressure fields $(\b u_h, p_h)$ solving problem \eqref{eq:stokes.vem} can be obtained by reasoning as in the proof of Proposition \ref{prop:well-posed_cont}. The estimate 
$$
\tri{\b u_h}\tri_{r}\lesssim
\left(\| \b f_h \|_{\b L^{r'}(\Omega)} +
\| \b g \|_{\b L^{r'}(\Gamma_N)}\right)^{\frac1{r-1}}
$$
hinges on the monotonicity property in Proposition \ref{lem:holder_monotonicity_ah}; whereas the estimate for the discrete pressure field follows from the inf-sup condition in Lemma \ref{inf-sup:vem} and again Proposition \ref{lem:holder_monotonicity_ah} and reads 
$$
\norm[L^{r'}(\Omega)]{p_h} \lesssim  
\left(\| \b f_h \|_{\b L^{r'}(\Omega)} + \| \b g \|_{\b L^{r'}(\Gamma_N)}\right) + \delta^{r-2}
\left(\| \b f_h \|_{\b L^{r'}(\Omega)} +
\| \b g \|_{\b L^{r'}(\Gamma_N)}\right)^{\frac1{r-1}}.
$$

\end{remark}

%\subsubsection{A note on the case \texorpdfstring{$r<2$}{r<2}}\label{sec:rless2}

\section{\emph{A priori} error analysis}
\label{sec:error_analysis}
This section is devoted to the \emph{a priori} error analysis.
\subsection{Additional properties of the stress-strain law}

We recall some important results regarding the stress-strain relation that are instrumental for the \emph{a priori} analysis of the scheme.
We mainly follow \cite[Section 3]{Berselli.Diening.ea:10} and \cite[Section 2]{Hirn:13}. 
For $r\geq 2$, we introduce, with $a\ge0$, the shifted functions $\varphi_a(t) = \int_0^t (a+s)^{r-2}s\, {\rm d} s$. The following Lemma provides important properties of the shifted functions $\varphi_a$. We refer the reader to \cite[Lemmata 28--32]{Diening.Ettwein:08} and \cite[Corollary 26]{Kreuzer-Diening:2008} for the detailed proof.

\begin{lemma}[Young type inequalities]
\label{lem:young_shifted}
Let $r\geq 2$. For all $\varepsilon>0$ there exists $C_\varepsilon >0$ only depending on $r$ such that for all $s,t,a\ge 0$ and all $\b\tau,\b\eta\in\mathbb{R}^{d\times d}$ there holds
\begin{subequations}
\begin{align}
\label{eq:Young1}
s\varphi_a'(t) +t\varphi_a'(s) &\le 
\varepsilon\varphi_a(s) + C_\varepsilon  \varphi_a(t),
\\
%\label{eq:Young2}
\varphi_{a+|\b\tau|}(t) &\le
\varepsilon \varphi_{a+|\b\eta|}(|\b\tau-\b\eta|)
+C_\varepsilon  \varphi_{a+|\b\eta|}(t).
\end{align}
\end{subequations}
\end{lemma}

\begin{proof} In the following we only sketch the proof that $C_\varepsilon $ does not depend on $a$, thanks to \cite[Lemmata 28--32]{Diening.Ettwein:08} and \cite[Corollary 26]{Kreuzer-Diening:2008}.
We define
\begin{subequations}
\begin{alignat}{2}
\Delta(\varphi_a) &\coloneq  \min\{\alpha > 0 : \varphi_a(2t) \le \alpha \varphi_a(t) \quad \forall t \in \mathbb{R}_0^+\},\\
\Delta(\{\varphi_a,\varphi_a^*\}) &\coloneq \max(\Delta(\varphi_a),\Delta(\varphi_a^*)),
 \end{alignat}
\end{subequations}
where $\varphi_a^*$ is the Fenchel conjugate of $\varphi_a$, i.e., $(\varphi_a^*)'(\varphi_a'(t))=t$ for all $t \in \mathbb R_0^+$. 
Now, we only need to show that $\Delta(\{\varphi_a,\varphi_a^*\})$ is bounded independently of $a$.
We have for all $ t \in \mathbb{R}_0^+$,
\begin{equation}
    \varphi_a(2t) = \int_0^{2t}(a+s)^{r-2}s \,\text{d}s = 4\int_0^{t}(a+2s)^{r-2}s \,\text{d}s \leq 2^r\int_0^{t}(a+s)^{r-2}s \,\text{d}s =  2^r\varphi_a(t)  
\end{equation}
so $\Delta(\varphi_a) \le 2^r$.
Moreover, using the fact that $
    \varphi_a'(2t) = (a+2t)^{r-2}2t \ge 2(a+t)^{r-2}t  =  2\varphi_a'(t)  
$, 
we obtain, ${\varphi_a}'(2{\varphi_a^{*}}'(2t)) \ge 2{\varphi_a}'({\varphi_a^{*}}'(t))=2t$, so $2{\varphi_a^{*}}'(t) \ge {\varphi_a^{*}}'(2t)$, thus $2{\varphi_a^{*}}(t) \ge \frac{1}{2}{\varphi_a^{*}}(2t)$ by integrating, hence, ${\varphi_a^{*}}(2t) \le 4{\varphi_a^{*}}(t)$ and we obtain $\Delta(\varphi_a^*) \le 4$.
Therefore, 
\begin{equation}
    \Delta(\{\varphi_a,\varphi_a^*\}) \le  \max(2^r,4) = 2^r. 
\end{equation}
\end{proof}

The next result showing the equivalence of several quantities is strictly related to the continuity and monotonicity assumptions given in Assumption~\ref{ass:hypo}. 
%%
% First, we introduce a function $\b{\mathcal{F}}:\mathbb{R}^{d\times d}_{\rm s}\to \mathbb{R}^{d\times d}_{\rm s}$ which is closely related to the stress tensor in the Carreau--Yasuda model \eqref{eq:Carreau}:
% \begin{equation}\label{eq:addfun}
% \b{\mathcal{F}}(\b\eta) = (\delta + |\b\eta|)^{\frac{r-2}2} \b\eta.
% \end{equation}
%%
The proofs of the next lemma can be found in \cite[Section 2.3]{Hirn:13}. 
 The lemma here below applies to any (scalar or tensor valued) function $\b{\sigma}$ which satisfies Assumption~\ref{ass:hypo}. In the following, with a slight abuse of notation, we will apply such lemma both to the constitutive law $\b{\sigma}$ but also to the auxiliary scalar function $\sigma(\tau)= (\delta+|\tau|)^{r-2} \tau$.
\begin{lemma}\label{lm:Hirn}
Let $\b\sigma$ satisfy \eqref{eq:hypo} for $r\in [2,\infty)$ and $\delta \geq0$. 
% and let $\b{\mathcal{F}}$ be defined by \eqref{eq:addfun}. 
Then, uniformly for all $\b\tau,\b\eta\in \mathbb{R}^{d\times d}_{\rm s}$ and all $\b{v},\b{w}\in\b{U}$ there hold 
\begin{subequations}
%\label{eq:extrass}
  \begin{alignat}{2} 
  \label{eq:extrass.continuity}
  |\b{\sigma}(\cdot,\b{\tau})-\b{\sigma}(\cdot,\b{\eta})| & \simeq 
  \left(\delta+|\b\tau|+|\b\eta| \right)^{r-2}
  |\b\tau-\b\eta| \simeq \varphi_{\delta+|\b\tau|}'(|\b\tau - \b\eta|),
  \\
  \label{eq:extrass.monotonicity}
  \left(\b{\sigma}(\cdot,\b\tau)-\b{\sigma}(\cdot, \b\eta)\right):\left(\b\tau-\b\eta\right) &\simeq 
  \left(\delta+|\b\tau|+|\b\eta|\right)^{r-2} |\b\tau-\b\eta|^{2} \simeq 
  \varphi_{\delta+|\b\tau|}(|\b\tau - \b\eta|) ,
  % \simeq |\b{\mathcal{F}}(\b\tau)-\b{\mathcal{F}}(\b\eta)|^2,
  % \\
  % \label{eq:extrass.add1}
  % \norm[\mathbb{L}^r(\Omega)]{\b\epsilon(\b v -\b w)}^2 &\lesssim
  % \norm[\mathbb{L}^2(\Omega)]{\b{\mathcal{F}}(\b\epsilon(\b v)) -\b{\mathcal{F}}(\b\epsilon(\b w))}^2
  % \norm[\mathbb{L}^r(\Omega)]{\delta+|\b\epsilon(\b v)|+|\b\epsilon(\b w)|}^{2-r},
  % \\
  % \label{eq:extrass.add2}
  % \norm[\mathbb{L}^2(\Omega)]{\b{\mathcal{F}}(\b\epsilon(\b v)) -\b{\mathcal{F}}(\b\epsilon(\b w))}^2
  % &\lesssim \norm[\mathbb{L}^r(\Omega)]{\b\epsilon(\b v -\b w)}^r,
  % \\
  % \label{eq:extrass.add3}
  % \norm[\mathbb{L}^{r'}(\Omega)]{\b\sigma(\cdot,\b\epsilon(\b v)) -\b\sigma(\cdot,\b\epsilon(\b w))}^2
  % &\lesssim 
  % \norm[\mathbb{L}^2(\Omega)]{\b{\mathcal{F}}(\b\epsilon(\b v)) -\b{\mathcal{F}}(\b\epsilon(\b w))}^{\frac2{r'}},
  \end{alignat}
\end{subequations}
where the hidden constants only depend on  $\sigma_{\rm{c}}, \sigma_{\rm{m}}$ in Assumption~\ref{ass:hypo} and on $r$.
\end{lemma}

\subsection{\emph{A priori} error estimate: velocity}
{ We start by a simple lemma; the proof is perhaps different from expected since $\VDL \subset W^{1,2}(E)$ but $\VDL \not\subset W^{1,r}(E)$ for non-convex elements $E$ and large $r$, see Remark~\ref{rem:nonincl}.
\begin{lemma}\label{lem:nonconv:approx}
Let the mesh regularity in Assumption~\ref{ass:mesh} hold.
Let $E \in \Omega_h$ and $r \in [2,\infty]$. Let $\b v \in \b W^{s,{r}}(E)$, $1<s \le k+1$, and $\b v_I \in \VDG$ denote the interpolant of $\b v$ previously introduced. Then it holds
\begin{eqnarray}
\label{eq:lemmanonconv:1}
&& \| \b{\epsilon}(\b v) - \Pi^{0, E}_{k-1} \b{\epsilon}(\b v_I) \|_{\b L^{r}(E)} \lesssim h_E^{s-1} |\b v |_{\b W^{s,r}(E)} \, , \\
&& |\DOF((I - \Pi^0_k) \b v_I)| \lesssim h_E^{s-2/r} |\b v 
|_{\b W^{s,{r}(E)}} \, .
\label{eq:lemmanonconv:2}
\end{eqnarray}
\end{lemma} 
\begin{proof}
We start by some trivial manipulation and afterwards apply the first bound in Lemma \ref{lem:approx-interp} together with the polynomial inverse estimate \eqref{eq:inverse}, obtaining
$$
\begin{aligned}
\| \b{\epsilon}(\b v) - \Pi^{0, E}_{k-1} \b{\epsilon}(\b v_I) \|_{\b L^{{r}}(E)}
& \le 
\| \b{\epsilon}(\b v) - \Pi^{0, E}_{k-1} \b{\epsilon}(\b v) \|_{\b L^{{r}}(E)}
+ \| \Pi^{0, E}_{k-1} \b{\epsilon}(\b v - \b v_I) \|_{\b L^{{r}}(E)} \\
& \lesssim h_E^{s-1} |\b v |_{\b W^{s,{r}}(E)} +
|E|^{1/{{r}} - 1/2} \| \Pi^{0, E}_{k-1} \b{\epsilon}(\b v - \b v_I) \|_{\b L^2(E)} \, .
\end{aligned}
$$
We conclude the proof of the first bound by the $L^2(E)$ continuity of $\Pi^{0, E}_{k-1}$, the interpolation estimates in Lemma \ref{lem:approx-interp} and finally a H\"older inequality on the element:
$$
\begin{aligned}
\| \b{\epsilon}(\b v) - \Pi^{0, E}_{k-1} \b{\epsilon}(\b v_I) \|_{\b L^{{r}}(E)}
& \lesssim h_E^{s-1} |\b v |_{\b W^{s,{r}}(E)} +
|E|^{1/r - 1/2} h_E^{s-1} |\b v |_{\b W^{s,2}(E)} \\
& \lesssim h_E^{s-1} |\b v |_{\b W^{s,{r}}(E)} \, .
\end{aligned}
$$
In order to deal with the second bound, we first apply Lemma \ref{Lem:vemstab-2}, then some obvious manipulations, finally Lemma \ref{lem:approx-interp} and an inverse estimate for polynomials. We obtain
$$
\begin{aligned}
|\DOF(I - \Pi^0_k) \b v_I)| & \lesssim |(I - \Pi^0_k) \b v_I|_{\b W^{1,2}(E)}
\le |\b v- \b v_I|_{\b W^{1,2}(E)} + |(I - \Pi^0_k) \b v|_{\b W^{1,2}(E)}
+ |\Pi^0_k(\b v - \b v_I)|_{\b W^{1,2}(E)} \\
& \lesssim h_E^{s-1} |\b v |_{\b W^{s,2}(E)} + 
h_E^{-1} \|\Pi^{0,E}_k(\b v - \b v_I)\|_{\b L^{2}(E)} \, .
\end{aligned}
$$
We now recall the $L^2(E)$ continuity of $\Pi^{0,E}_{k}$, again Lemma \ref{lem:approx-interp} and a H\"older inequality:
$$
|\DOF(I - \Pi^0_k) \b v_I)|\lesssim h_E^{s-1} |\b v |_{\b W^{s,2}(E)}
\lesssim h_E^{s-2/{r}} |\b v |_{\b W^{s,{r}}(E)} \, .
$$
\end{proof}
}

{
Moreover, we state the following  interpolation lemma.

\begin{lemma}\label{lem:interp:tribar}
Let the mesh regularity in Assumption~\ref{ass:mesh} hold.
Let $\b v \in \b W^{s,r}(\Omega_h)$, with $1<s \le k+1$ and $r \ge 2$. Let $\b v_I \in \VDG$ denote the interpolant of $\b v$ previously introduced. Then it holds
%  \begin{equation}
%    \corr{\tri \b u - \b u_I \tri_{\delta,r}^r  \lesssim \delta^{(r-2)} h^{2(s-1)} | \b u|_{W^{s,2}(\Omega_h)}^{2} 
% + h^{r(s-1)} | \b u|_{W^{s,r}(\Omega_h)}^r \, .}{}{}
% \end{equation}
% \bigbreak
\begin{equation}\label{eq:main_delta_nonzero}
\tri \b u - \b u_I \tri_{\delta,r}^r  
\lesssim(\delta+h^{s-1}|\b u|_{W^{s,r}(\Omega_h)})^{r-2} h^{2(s-1)}|\b u|_{W^{s,r}(\Omega_h)}^2.
\end{equation}
\end{lemma}
\begin{proof}
The proof will be presented briefly, since it essentially makes use of techniques already developed in previous results of this contribution.
Let now $\b e_I = \b u - \b u_I$ and, as usual, ${\b e}_I^\perp : = (I-\Pi^{0}_{k})\b e_I$.
By definition we have
$$
\tri \b e_I \tri_{\delta,r}^r = \| \Pi^{0}_{k-1} \b e_I \|_{\b L^r(\Omega_h)}^r 
+ S(e_I^\perp,e_I^\perp) =: T_1 + T_2 \, .
$$
The first term on the right-hand side is easily bounded by the triangle inequality, polynomial approximation estimates and Lemma \ref{lem:nonconv:approx}, obtaining
$$
T_1  \lesssim h^{r(s-1)} | \b u|_{W^{s,r}(\Omega_h)}^r \, .
$$
% For the second term we note that, by definition of interpolant, $\DOF(\b u - \b u_I)=0$;
By definition of the stabilization form, first by a trivial manipulation, then by a discrete $(\frac{r}{2},\frac{r}{r-2})$-H\"older inequality, we write (recall $|E|\simeq h_E^2$)
%(we set $\alpha=1$, c.f. \eqref{eq:dofi}, which is equivalent)
\begin{equation}\label{eq:koala}
\begin{aligned}
T_2 &\simeq \sum_{E \in \Omega_h} (\delta + h_E^{-1}|\DOF({\b e}_I^\perp)|)^{r-2} |\DOF({\b e}_I^\perp)|^2 \\
& = \sum_{E \in \Omega_h} (h_E^{2/r}\delta + h_E^{2/r - 1}|\DOF({\b e}_I^\perp)|)^{r-2} 
(h_E^{(2-r)/r} |\DOF({\b e}_I^\perp)|)^2 \\
& \lesssim 
\big( \sum_{E\in\Omega_h} h_E^2\delta^r + h_E^{2-r}\vert \DOF({\b e}_I^\perp) \vert^r \big)^{\frac{r-2}{r}} (\sum_{E\in\Omega_h}h_E^{2-r}|\DOF ({\b e}_I^\perp)|^r)^\frac{2}{r} \\
&\lesssim \big(\delta^r + \sum_{E\in\Omega_h}h_E^{2-r}\vert \DOF({\b e}_I^\perp) \vert^r \big)^{\frac{r-2}{r}} (\sum_{E\in\Omega_h}h_E^{2-r}|\DOF ({\b e}_I^\perp)|^r)^\frac{2}{r} \, .
\end{aligned}
\end{equation}
We now apply Lemma  \ref{Lem:vemstab-2} and approximation properties for polynomials
$$
\begin{aligned}
T_2 &\lesssim \big(\delta^r + \sum_{E\in\Omega_h}h_E^{2-r}|{\b e}_I^\perp|_{W^{1,2}(E)}^r \big)^{\frac{r-2}{r}} (\sum_{E\in\Omega_h}h_E^{2-r}|{\b e}_I^\perp|_{W^{1,2}(E)}^r)^\frac{2}{r}
\\
& \lesssim \big(\delta^r + \sum_{E\in\Omega_h}h_E^{2-r+r(s-1)} | \b u|_{W^{s,2}(E)}^r \big)^{\frac{r-2}{r}} (\sum_{E\in\Omega_h}h_E^{2-r+r(s-1)} | \b u|_{W^{s,2}(E)}^r)^\frac{2}{r}
\\
&\le \big(\delta^r + \sum_{E\in\Omega_h}h_E^{r(s-1)} | \b u|_{W^{s,r}(E)}^r \big)^{\frac{r-2}{r}} (\sum_{E\in\Omega_h}h_E^{r(s-1)} | \b u|_{W^{s,r}(E)}^r)^\frac{2}{r}
\\
&\lesssim(\delta+h^{s-1}|\b u|_{W^{s,r}(\Omega_h)})^{r-2} (h^{s-1}|\b u|_{W^{s,r}(\Omega_h)})^2 \, .
\end{aligned}
$$
%%
\begin{comment}
\corr{
where in the second step we applied Lemma \ref{Lem:vemstab-2}. Approximation estimates and standard arguments yield 
}{}{}
\corr{
$|{\b e}_I^\perp|_{W^{1,2}(E)} \lesssim h_E^{s-1} | \b u|_{W^{s,2}(E)} $, which in turn implies
}{}{}
$$
\corr{
\sum_{E \in \Omega_h} \delta^{r-2} |{\b e}_I^\perp|_{W^{1,2}(E)}^2  
\lesssim \delta^{r-2} h^{2(s-1)} | \b u|_{W^{s,2}(\Omega_h)}^2,
}{}{}
$$
and 
$$
\corr{
\sum_{E \in \Omega_h} h_E^{2-r} |{\b e}_I^\perp|_{W^{1,2}(E)}^r \lesssim
\sum_{E \in \Omega_h} h_E^{2-r} h_E^{r(s-1)} | \b u|_{W^{s,2}(E)}^r \, .
}{}{}
$$
\corr{
By applying a H\"older inequality from $W^{s,2}(E)$ into $W^{s,r}(E)$ and some basic calculations the above term can be bounded by
}{}{}
$$
\corr{
\sum_{E \in \Omega_h} h_E^{2-r} h_E^{r(s-1)} h_E^{r-2} | \b u|_{W^{s,r}(E)}^r
\le h^{r(s-1)} | \b u|_{W^{s,r}(\Omega_h)}^r \, .
}{}{}
$$
\end{comment}
%%
The proof is concluded trivially by combining the bounds above.
\end{proof}
}

We now present the main result of this section (see also the important Remark~\ref{rem:orders}).

\begin{theorem}\label{theo:main} 
Let $\b u$ be the solution of problem \eqref{eq:stokes.weak.Z} and let $\b u_h$ be the solution of problem \eqref{eq:stokes.vem.Z}. 
Assume that $\b u \in \b W^{k_1+1,r}(\Omega_h)$, 
$\b\sigma(\cdot, \b \epsilon(\b u)) \in \mathbb{W}^{k_2,r'}(\Omega_h)$, $\b f \in \b W^{k_3+1,r'}(\Omega_h)$ for some positive integers $k_1$, $k_2$, $k_3 \le k$. 
Let the mesh regularity assumptions stated in Assumption~\ref{ass:mesh} hold. Then, we have
\begin{equation}\label{falcao}
    \tri \b u -\b u_h \tri_{\delta,r} \lesssim
 (\delta+h^{k_1}R_1+R_4)^\frac{r-2}{r}h^{\frac{2k_1}{r}}R_1^\frac{2}{r} +
    h^{\frac{k_2}{r-1}}R_2^\frac{1}{r-1} +h^{\frac{{ k_3+2}}{r-1}} R_3^\frac{1}{r-1}\, ,
\end{equation}
where the regularity terms are
\begin{equation}\label{eq:reg:terms}
\begin{aligned}
& R_1 = | \b u |_{\b W^{k_1+1,r}(\Omega_h)} \, , \qquad
R_2 = |\b\sigma(\cdot, \b \epsilon(\b u))|_{\mathbb{W}^{k_2,r'}(\Omega_h)} \,, \\
& R_3 = |\b f |_{\b W^{k_3+1,r'}(\Omega_h)} \, , \qquad
R_4 : =  \| \b \epsilon(\b u) \|_{L^r(\Omega_h)} 
%+\| \b u \|_{W^{1,\infty}(\Omega_h)}  
\, .
\end{aligned}
\end{equation}
\end{theorem}
\begin{proof}
We set $\b \xi_h=\b u_h -\b u_I$, and as usual, $\b v^\perp=(I - \Pi^0_k) \b v$ for any $\b v \in \b L^2(\Omega)$.
First, by a triangle inequality and Lemma \ref{lem:interp:tribar} (with $s=k_1+1$), we have
\begin{equation}\label{eq:bound_u_uh}
% \begin{aligned}
\tri\b u -\b u_h \tri_{\delta,r} \le \tri\b u -\b u_I \tri_{\delta,r} +\tri\b \xi_h \tri_{\delta,r} 
\lesssim 
(\delta+h^{k_1}R_1)^\frac{r-2}{r}h^{\frac{2k_1}{r}}R_1^\frac{2}{r} 
+ \tri\b \xi_h \tri_{\delta,r} \,.
\end{equation}
Since $\b \xi_h \in \ZDG$, manipulating \eqref{eq:stokes.vem.Z} and \eqref{eq:stokes.continuous} and recalling \eqref{eq:forma ah} we have
\begin{equation}
\label{eq:conv0}
\begin{aligned}
    &a_h(\b u_h, \b \xi_h) - a_h(\b u_I, \b \xi_h)
    =  
    \int_\Omega (-\DIV\b{\sigma}(\cdot,\b{\epsilon}(\b{u}))  
    - \b{f}) \cdot \b \xi_h  - a_h(\b u_I, \b \xi_h) +
    (\b f_h, \b \xi_h )  
    \\
    &= 
    \int_\Omega\left( \b\sigma(\cdot, \b \epsilon(\b u)) : \b\epsilon( \b \xi_h) - \b\sigma(\cdot, \b \Pi^0_{k-1}\b \epsilon(\b u_I)) : \b \Pi^0_{k-1} \b\epsilon( \b \xi_h)  \right) - S({\b u}_I^\perp, {\b \xi}_h^\perp) 
    + (\b f_h -\b f , \b \xi_h )
    \\
    &=: T_1+T_2+T_3 \,.
\end{aligned}    
\end{equation}

We next estimate each term on the right-hand side above.
In the following $C$  will denote a generic positive constant independent of $h$ that may change at each occurrence, whereas the positive parameter $\theta$ adopted in \eqref{eq:T1A} and \eqref{eq:T3} will be specified later.

\noindent 
$\bullet$ \, Estimate of $T_1$. Employing the definition of $L^2$-projection \eqref{eq:P0_k^E} we have
\begin{equation}
\label{eq:T10}
\begin{aligned}
T_1 &= \int_\Omega \bigl(\b\sigma(\cdot, \b \epsilon(\b u)) - \b \Pi^0_{k-1}\b\sigma(\cdot, \b \Pi^0_{k-1} \b \epsilon(\b u_I)) \bigr) : \b\epsilon(\b \xi_h)
\\
&=
\int_\Omega \bigl((I - \b \Pi^0_{k-1})\b\sigma(\cdot, \b \epsilon(\b u))  \bigr) : ((I - \b \Pi^{0}_{k-1})\b\epsilon(\b \xi_h))  \\
&\qquad +
\int_\Omega \bigl(\b\sigma(\cdot, \b \epsilon(\b u)) - \b\sigma(\cdot, \b \Pi^0_{k-1} \b \epsilon(\b u_I)) \bigr) : \b \Pi^0_{k-1} \b\epsilon(\b \xi_h)\\
&=:T_1^A+ T_1^B \,.
\end{aligned}
\end{equation}
We now recall the standard polynomial interpolation result
\begin{equation}\label{eq:pol-temp}
\| (I - \b \Pi^{0,E}_{k-1})\b\sigma(\cdot, \b \epsilon(\b u)) \|_{\b L^2(E)} 
\lesssim h_E^{k_2 -1} |\b\sigma(\cdot, \b \epsilon(\b u))|_{\mathbb{W}^{k_2,1}(E)} \, .
\end{equation}
 Furthermore, combining Lemma \ref{Lem:vemstab-2} with the first line of equation \eqref{eq:strong1}, it is easy to check
\begin{equation}\label{newkidinblock}
\|\b\epsilon({\b \xi}_h^\perp) \|_{\b L^2(E)} 
\lesssim |E|^{\frac{1}{2}-\frac{1}{r}}S^E({\b \xi}_h^\perp,{\b \xi}_h^\perp)^\frac{1}{r} \, .
\end{equation} 
The term $T_1^A$ can be bounded as follows   
\begin{equation}
\label{eq:T1A}
\begin{aligned}
    T_1^A &= \sum_{E \in \Omega_h}
    \int_E \bigl((I - \b \Pi^{0,E}_{k-1})\b\sigma(\cdot, \b \epsilon(\b u))  \bigr) : ((I - \b \Pi^{0,E}_{k-1})\b\epsilon(\b \xi_h))
    \\
   & \le \sum_{E\in \Omega_h} { \| (I - \b \Pi^{0,E}_{k-1})\b\sigma(\cdot, \b \epsilon(\b u)) \|_{\b L^2(E)}}  \|(I - \b \Pi^{0,E}_{k-1})\b\epsilon(\b \xi_h) \|_{\b L^2(E)} 
    & \quad & \text{(Cauchy--Schwarz ineq.)}
        \\
   & \le C \sum_{E\in \Omega_h} h_E^{k_2 -1} |\b\sigma(\cdot, \b \epsilon(\b u))|_{\mathbb{W}^{k_2,1}(E)} \|\b\epsilon({\b \xi}_h^\perp) \|_{\b L^2(E)}
    & \quad & { \text{(Cont of $\Pi^{0,E}_{k-1}$ \& \eqref{eq:pol-temp})}}
    \\
       & \le C \sum_{E\in \Omega_h} h_E^{k_2 -1} |E|^\frac{1}{r}|\b\sigma(\cdot, \b \epsilon(\b u))|_{\mathbb{W}^{k_2,r'}(E)}  |E|^{\frac{1}{2}-\frac{1}{r}}S^E({\b \xi}_h^\perp,{\b \xi}_h^\perp)^\frac{1}{r} 
    & \quad & \text{($(r',r)$-H\"older ineq. \& \eqref{newkidinblock})}
    \\
   &
   \le \frac{C}{r'\theta^{r'}}  h^{k_2r'} |\b\sigma(\cdot, \b \epsilon(\b u))|_{\mathbb{W}^{k_2,r'}(\Omega_h)}^{r'} + 
   \frac{\theta^r}{2r}S({\b \xi}_h^\perp,{\b \xi}_h^\perp),
    & \quad & \text{($(r',r)$-Young ineq.)}
\end{aligned}
\end{equation}
where we used the fact that $|E|^\frac{1}{2} \simeq h_E$.
%
% Then, the term $T_1^B$ can be bounded as follows
% \begin{equation}
% \label{eq:T1A}
% \begin{aligned}
%    T_1^B &= \sum_{E \in \Omega_h}
%    \int_E \bigl((I - \b \Pi^{0,E}_{k-1})\b\sigma(\cdot, \b \epsilon(\b u))  \bigr) : \b \Pi^{0,E}_{k-1}\b\epsilon(\b \xi_h) 
%    \\
%   & \le C \sum_{E\in \Omega_h} h_E^{k_2} |\b\sigma(\cdot, \b \epsilon(\b u))|_{\mathbb{W}^{k_2,r'}(E)} \|\b \Pi^{0,E}_{k-1}\b\epsilon(\b \xi_h) \|_{\b L^r(E)} 
%    & \quad & \text{($(r',r)$-H\"older ineq.)}
%        \\
%   &
%   \le \frac{C}{r'\theta^{r'}}  h^{k_2 r'} |\b\sigma(\cdot, \b \epsilon(\b u))|_{\mathbb{W}^{k_2,r'}(\Omega_h)}^{r'} + 
%   \frac{\theta^r}{2r} \|\b \Pi^{0}_{k-1}\b\epsilon(\b \xi_h) \|_{\b L^r(\Omega_h)}^r
%\end{aligned}
% \end{equation}
%
Employing \eqref{eq:extrass.continuity}, we obtain 
   \[
   \begin{aligned}
    T_1^B &\le
    \sum_{E\in\Omega_h}\int_E \vert \b\sigma(\cdot, \b \epsilon(\b u)) - \b\sigma(\cdot, \PP0 \b \epsilon(\b u_I))\vert \, \vert \PP0 \b\epsilon(\b \xi_h)\vert 
    \\
    &\lesssim \sum_{E\in\Omega_h}\int_E
    \varphi_{\delta+|\PP0 \b \epsilon(\b u_I)|}'(|\PP0 \b \epsilon(\b u_I) - \b \epsilon(\b u)|) \,\vert \PP0 \b\epsilon(\b \xi_h)\vert \,.
    \end{aligned}
    \]
Employing \eqref{eq:Young1} we get that for every $\varepsilon>0$, there exists a positive constant $C_\varepsilon $ such that 
\[
\begin{aligned}
     T_1^B &\le
    \varepsilon \sum_{E\in\Omega_h}\int_E
    \varphi_{\delta+|\PP0 \b \epsilon(\b u_I)|}(|\PP0 \b \epsilon(\b \xi_h)|) + 
    C_\varepsilon  \sum_{E\in\Omega_h}\int_E\varphi_{\delta+|\PP0 \b \epsilon(\b u_I)|}(|\PP0 \b \epsilon(\b u_I) - \b \epsilon(\b u)|) .
\end{aligned}
\]     
Using \eqref{eq:extrass.monotonicity}  we obtain 
(with $\gamma$ denoting the associated uniform hidden constant)
\[
\begin{aligned}
T_1^B &\le
    \gamma \varepsilon
    (\b\sigma(\cdot, \b\Pi_{k-1}^0 \b \epsilon(\b u_h))- \b\sigma(\cdot, \b\Pi_{k-1}^0 \b \epsilon(\b u_I)), \b\Pi_{k-1}^0 \b\epsilon(\b \xi_h)) \\
    &\qquad +C_\varepsilon  
    \sum_{E\in\Omega_h} \int_E (\delta+|\PP0 \epsilon(\b u_I)|+|\b\epsilon(\b u)|)^{r-2}  
    |\b\epsilon(\b u) - \PP0 \epsilon(\b u_I)|^2  \,.   
\end{aligned}
\]
Notice that the constant $C_\varepsilon$ depends only on $\sigma_{\rm c}, \sigma_{\rm m}, r$, and $\varepsilon$. With respect to $\varepsilon$, it may depend on the degree $k$ and the domain $\Omega$. However, given our mesh assumptions, it is independent of the particular mesh or mesh element within the family $\{ \Omega_h \}_h$.
Using an $(\frac{r}{r-2},\frac{r}{2})$-H\"older inequality since $r > 2$, from the last equation we get
\[
\begin{aligned}
T_1^B&\le
    \gamma \varepsilon
    (\b\sigma(\cdot, \b\Pi_{k-1}^0 \b \epsilon(\b u_h))- \b\sigma(\cdot, \b\Pi_{k-1}^0 \b \epsilon(\b u_I)), \b\Pi_{k-1}^0 \b\epsilon(\b \xi_h))\\
    &\qquad +C_\varepsilon  \sum_{E\in\Omega_h}
    ( |E| \delta^r+\|\b\Pi_{k-1}^0\b \epsilon(\b u_I)\|^r_{\mathbb{L}^r(E)}+\| \b\epsilon(\b u)\|^r_{\mathbb{L}^r(E)})^\frac{r-2}{r}\| \b\epsilon(\b u) - \b\Pi_{k-1}^0 \epsilon(\b u_I)\|^2_{\mathbb{L}^r(E)} \, ,
    \end{aligned}
\]
which, making use of Lemma \ref{lem:nonconv:approx} and taking $\varepsilon= \frac{1}{2 \gamma}$ becomes (using a discrete $(\frac{r}{r-2},\frac{r}{2})$-H\"older inequality)
\begin{equation}
\label{eq:T1B}
\begin{aligned}
T_1^B&\le
    \frac{1}{2}
    (\b\sigma(\cdot, \b\Pi_{k-1}^0 \b \epsilon(\b u_h))- \b\sigma(\cdot, \b\Pi_{k-1}^0 \b \epsilon(\b u_I)), \b\Pi_{k-1}^0 \b\epsilon(\b \xi_h))\\
    &\qquad + C h^{2k_1}( |\Omega| \delta^r+\| \b\epsilon(\b u)\|^r_{\mathbb{L}^r(\Omega_h)})^\frac{r-2}{r} | \b u |_{W^{k_1+1,r}(\Omega_h)}^2 \, .
\end{aligned}
\end{equation}
Combining \eqref{eq:T1A} and \eqref{eq:T1B} in \eqref{eq:T10} we infer
\begin{equation}
\label{eq:T1}
\begin{aligned}
T_1 &\le
    \frac{1}{2}
    (\b\sigma(\cdot, \b\Pi_{k-1}^0 \b \epsilon(\b u_h))- \b\sigma(\cdot, \b\Pi_{k-1}^0 \b \epsilon(\b u_I)), \b\Pi_{k-1}^0 \b\epsilon(\b \xi_h))\\
    &\qquad +
    C h^{2k_1} (\delta^r+R_4^r)^\frac{r-2}{r} R_1^2 
      +\frac{C}{r'\theta^{r'}}   h^{k_2 r'} R_2^{r'}+ 
    \frac{\theta^r}{2r} \tri\b \xi_h \tri_{\delta,r}^r \, .
\end{aligned}    
\end{equation}
\noindent 
$\bullet$ \, Estimate of $T_2$. 
Recalling definitions \eqref{eq:Sglobal} and \eqref{eq:dofi}, with some trivial algebra we obtain
\[
    T_2 = - \sum_{E \in \Omega_h}S^E({\b u}_I^\perp, {\b \xi}_h^\perp)
    \le \sum_{E\in \Omega_h} C h_E^{2-r} (\delta h_E+|\DOF({\b u}_I^\perp)|)^{r-2} \vert\DOF ({\b u}_I^\perp) \vert \, \vert \DOF ({\b \xi}_h^\perp)\vert \,.
\]
Employing \eqref{eq:extrass.continuity} in Lemma \ref{lm:Hirn} to the scalar function   
$\sigma(\tau)= (\delta h_E+|\tau|)^{r-2} \tau$ (hence $\eta =0$) we have
\[
T_2 \le C\sum_{E \in \Omega_h}  h_E^{2-r}\varphi'_{\delta h_E+|\DOF ({\b u}_I^\perp)|}(|\DOF ({\b u}_I^\perp)|) \, \vert \DOF ({\b \xi}_h^\perp)\vert \,.
\]
Employing \eqref{eq:Young1} we get that for every $\varepsilon>0$ there exists a positive constant $C_\varepsilon $ such that 
\[
\begin{aligned} 
    T_2 & \le \varepsilon \sum_{E\in \Omega_h} h_E^{2-r} 
      \varphi_{\delta h_E+|\DOF ({\b u}_I^\perp)|}(|\DOF({\b \xi}_h^\perp)|)
    + C_\varepsilon  \sum_{E\in \Omega_h} h_E^{2-r}  \varphi_{\delta h_E+|\DOF ({\b u}_I^\perp)|}(|\DOF ({\b u}_I^\perp)|) \,.
\end{aligned}
\]
We now use \eqref{eq:extrass.monotonicity} and, denoting with $\gamma$ the hidden constant, we infer
\begin{equation}
\label{eq:T20}
\begin{aligned} 
T_2 & \le \varepsilon \gamma \sum_{E\in \Omega_h} h_E^{2-r} 
      (\delta h_E+|\DOF ({\b u}_I^\perp)| + |\DOF ({\b u}_h^\perp)|)^{r-2} |\DOF({\b \xi}_h^\perp)|^2 \\
      &\qquad + C_\varepsilon  \sum_{E\in \Omega_h} h_E^{2-r}(\delta h_E+|\DOF ({\b u}_I^\perp)|)^{r-2}|\DOF ({\b u}_I^\perp)|^2 =: T_2^A + T^2_B \,.
\end{aligned}
\end{equation}
Employing \eqref{eq:strong1} (cf. proof of Lemma \ref{Lemma:S_holder_monotonicity}) with  $\b u_h = {\b u}_h^\perp$ and  $\b w_h = {\b u}_I^\perp$ and recalling definition \eqref{eq:Sglobal} we obtain
\begin{equation}
\label{eq:T2A}
T_2^A \le \varepsilon C  \gamma \sum_{E \in \Omega_h}
(S^E({\b u}_h^\perp, {\b \xi}_h^\perp) - 
S^E({\b u}_I^\perp, {\b \xi}_h^\perp) ) 
= \varepsilon C  \gamma
(S({\b u}_h^\perp, {\b \xi}_h^\perp) - 
S({\b u}_I^\perp, {\b \xi}_h^\perp) ) \,.
\end{equation}
Reasoning as in \eqref{eq:koala} and employing Lemma \ref{lem:nonconv:approx} we can easily derive 
\begin{equation}
\label{eq:T2B}
     T_2^B \lesssim  h^{2k_1} (\delta+h^{k_1}R_1)^{r-2}R_1^2 \, .
\end{equation}
Combining \eqref{eq:T2A} and \eqref{eq:T2B} in \eqref{eq:T20} and taking $\epsilon = \frac{1}{2 C \gamma}$ we infer
\begin{equation}
\label{eq:T2}
T_2 \le 
\frac{1}{2}
(S({\b u}_h^\perp, {\b \xi}_h^\perp) - 
S({\b u}_I^\perp, {\b \xi}_h^\perp) ) + 
C h^{2k_1}(\delta+h^{k_1}R_1)^{r-2}R_1^2 \, .
\end{equation}

\noindent 
$\bullet$ \, Estimate of $T_3$. With analogous arguments to those in \eqref{eq:T1A} we infer 
\begin{equation}
\label{eq:T3}
\begin{aligned}
    T_3 &= \sum_{E \in \Omega_h}\int_E (\Pi^{0,E}_k \b f - \b f) \cdot {\b \xi}_h^\perp 
    & \quad & \text{(def. \eqref{eq:forma fh} \& def. \eqref{eq:P0_k^E})}
    \\
   & %\le { \sum_{E\in \Omega_h} \|(\Pi^{0,E}_k \b f - \b f) \|_{\b L^2(E)}} \|{\b \xi}_h^\perp\|_{\b L^2(E)} 
    % & \quad & \text{(Cauchy--Schwarz ineq.)} \\ & 
   \le C \sum_{E\in \Omega_h} h_E^{k_3} |\b f|_{\mathbb{W}^{{ k_3+1},1}(E)} h_E|\DOF({\b \xi}_h^\perp)|
  & & \text{(same reasoning as in \eqref{eq:T1A})}
    \\
       & \le C \sum_{E\in \Omega_h} h_E^{k_3} |E|^\frac{1}{r}|\b f|_{\mathbb{W}^{{ k_3+1},r'}(E)}  h_E^{1+\frac{r-2}{r}}S^E({\b \xi}_h^\perp,{\b \xi}_h^\perp)^\frac{1}{r} 
    & \quad & \text{($(r',r)$-H\"older ineq. \& Lemma \ref{Lemma:S_holder_monotonicity})}
    \\
   &
   \le \frac{C}{r'\theta^{r'}}   h^{{ (k_3+2)}r'} |\b f|_{\mathbb{W}^{{ k_3+1},r'}(\Omega_h)}^{r'} + 
  \frac{\theta^r}{2r}S({\b \xi}_h^\perp,{\b \xi}_h^\perp)
    & \quad & \text{($|E|^\frac{1}{2} \simeq h_E$ \& $(r',r)$-Young ineq.)}
     \\
   &
   \le \frac{C}{r'\theta^{r'}}  h^{{ (k_3+2)}r'} R_3^{r'} + 
   \frac{\theta^r}{2r}\tri\b \xi_h \tri_{\delta,r}^r
   &
\end{aligned}
\end{equation}
Plugging the estimates in \eqref{eq:T1}, \eqref{eq:T2} and \eqref{eq:T3} in \eqref{eq:conv0} and recalling definition \eqref{eq:forma ah} we obtain
\begin{equation}
\label{eq:upper}
\begin{aligned}
\frac{1}{2} (a_h(\b u_h, \b \xi_h) - a_h(\b u_I, \b \xi_h))    
   & \le  
   Ch^{2k_1}(\delta+{h^{k_1}R_1} + R_4)^{r-2}R_1^2 +
    \frac{C}{r'\theta^{r'}} h^{{k_2}{r'}} R_2^{r'}  \\
    &\qquad +
    \frac{C}{r'\theta^{r'}}   h^{{ (k_3+2)}{r'}} R_3^{r'} + 
   \frac{\theta^r}{r} \tri\b \xi_h \tri_{\delta,r}^r
   \, .
\end{aligned}
\end{equation}
We now write with \eqref{eq:ah_ineq_2} that
\begin{equation}
\label{eq:lower}
\begin{aligned}
\frac{1}{2} (a_h(\b u_h, \b \xi_h) - a_h(\b u_I, \b \xi_h)) & \ge
{\widetilde C} {\tri\b \xi_h \tri_{\delta,r}^r}.
\end{aligned}
\end{equation}
Combining  \eqref{eq:upper} and \eqref{eq:lower} and taking $\theta= (\frac{{\widetilde C}r}{2})^{1/r}$ we obtain:
\begin{equation}
\label{eq:pressF}
\begin{aligned}
\tri\b \xi_h \tri_{\delta,r}^r &\lesssim
 h^{2k_1}(\delta+{h^{k_1}R_1}+ R_4)^{r-2}R_1^2+
    h^{k_2r'} R_2^{r'} +h^{{(k_3+2)}r'} R_3^{r'}
\,.
\end{aligned}
\end{equation}
The proof follows by {\eqref{eq:bound_u_uh} together with} the previous bound. 
\end{proof}

According to Remark~\ref{rem:nonincl}, if in addition to  Assumption~\ref{ass:mesh} we assume that all the mesh elements are convex, we can state the error estimate with respect to the standard $\b W^{1,r}$-norm. Indeed, in this case, due to the elliptic regularity results established in \cite{Kellogg.Osborn:76} together with Sobolev embeddings, we have $\VDG(E) \subset \b W^{2,2}(E) \subset \b W^{1,r}(E)$ for all $E \in \Omega_h$ and $r \in [2,\infty)$, which combined with the global continuity of $\VDG$ implies $\VDG \subset \b W^{1,r}(\Omega)$. 
Furthermore, in such case our error estimate in Theorem \ref{theo:main}, which is expressed in the $\tri \cdot \tri_{\delta,r}$ norm, directly translates into equivalent estimates in the classical $W^{1,r}(\Omega)$ norm. 
\begin{corollary}\label{cor:W1r.estimate}
Let $\b u$ be the solution of problem \eqref{eq:stokes.weak.Z} and let $\b u_h$ be the solution of problem \eqref{eq:stokes.vem.Z}. Let all the mesh elements be convex. Under the regularity assumptions of Theorem~\ref{theo:main} and with the regularity terms defined as in \eqref{eq:reg:terms}, it holds
\[
    \| \b u -\b u_h \|_{\b W^{1,r}(\Omega)} \lesssim
 (\delta+h^{k_1}R_1+R_4)^\frac{r-2}{r}h^{\frac{2k_1}{r}}R_1^\frac{2}{r} +
    h^{\frac{k_2}{r-1}}R_2^\frac{1}{r-1} +h^{\frac{{ k_3+2}}{r-1}} R_3^\frac{1}{r-1}\, .
\]
\end{corollary}
\begin{proof}
We show the proof only briefly. 
The key step is deriving that 
\begin{equation}\label{eq:add:X}
\| \b v_h \|_{\b W^{1,r}(\Omega)} \lesssim \tri \b v_h \tri_{r} \qquad \forall \b v_h \in \VDG \, ,
\end{equation}
which can be shown easily, first by the Korn and triangle inequalities,
$$
\begin{aligned}
\| \b v_h \|_{\b W^{1,r}(\Omega)} & \lesssim \| \b\Pi^0_{k-1}\b \epsilon(\b v_h) \|_{\b L^{r}(\Omega_h)} 
+ \| ({\mathbf I} - \b\Pi^0_{k-1}) \b \epsilon(\b v_h) \|_{\b L^{r}(\Omega_h)} \\
& \lesssim \| \b\Pi^0_{k-1}\b \epsilon(\b v_h) \|_{\b L^{r}(\Omega_h)} + | (I - \Pi_k^0)\b v_h |_{\b W^{1,r}(\Omega_h)} 
\end{aligned}
$$
and then by definition of $\tri \cdot \tri_{r}$ norm and the stabilizing form, using Lemma \ref{Lem:vemstab-2} and an inverse estimate on virtual functions.  
Afterwards, by the triangle inequality and interpolation estimates combined with \eqref{eq:add:X}, we can write
$$
\| \b u -\b u_h \|_{\b W^{1,r}(\Omega)} \lesssim
\| \b u -\b u_I \|_{\b W^{1,r}(\Omega)}  + 
\| \b u_h -\b u_I \|_{\b W^{1,r}(\Omega)} \lesssim
C h^k | \b u |_{\b W^{k+1,r}(\Omega)} + \tri \b u_h -\b u_I \tri_{r} \, .
$$
The proof then follows from the above bound and \eqref{eq:comparison_delta_r_norm}, combined with Theorem \ref{theo:main}.
\end{proof}

We have also the following Corollary, stating a better bound for the $\b\sigma$ approximation term in Theorem~\ref{theo:main}, valid in the $\delta>0$ case.

\begin{corollary}\label{cor:main} 
Under the same assumptions of Theorem~\ref{theo:main}, if in addition $\delta>0$, $\b\sigma(\cdot, \b \epsilon(\b u)) \in \mathbb{W}^{k_2,2}(\Omega_h)$ and $\b f \in \b{W}^{k_3+1,2}(\Omega_h)$, it holds
\begin{equation}\label{eq:est_vel_delta_nonzero}
    \tri \b u -\b u_h \tri_{\delta,r} \lesssim
 h^{\frac{2k_1}{r}}(\delta +{h^{k_1}R_1}+ R_4)^\frac{r-2}{r}R_1^\frac{2}{r}+
    h^{\frac{2k_2}{r}}{\delta^{\frac{2-r}{r}}} \widehat{R}_2^{\frac{2}{r}} + h^{\frac{2(k_3+2)}{r}} {\delta^{\frac{2-r}{r}}}\widehat{R}_3^\frac{2}{r}\, ,
\end{equation}
with new regularity terms
\begin{equation}\label{eq:R_2-3-hat}
\widehat{R}_2 = |\b\sigma(\cdot, \b \epsilon(\b u))|_{\mathbb{W}^{k_2,2}(\Omega_h)}  \, ,
\qquad 
\widehat{R}_3 : = |\b f |_{\b W^{k_3+1,2}(\Omega_h)}
\, .
\end{equation}
\end{corollary}
\begin{proof}
The proof follows the same steps as that of Theorem~\ref{theo:main}, the only modification being the bounds for $T_1^A$ and for $T_3$. 
We start by noting that from definition \eqref{eq:dofi} and Lemma \ref{Lem:vemstab-2} it follows
\begin{equation}\label{bound:alternative}
S^E({\b \xi}_h^\perp,{\b \xi}_h^\perp) \gtrsim \delta^{r-2} | \DOF({\b \xi}_h^\perp) |^2
\gtrsim \delta^{r-2} | {\b \xi}_h^\perp |_{W^{1,2}(E)}^2 \, .
\end{equation}
Furthermore, standard polynomial interpolation results yield
\begin{equation}\label{eq:pol-temp-bis}
\| (I - \b \Pi^{0,E}_{k-1})\b\sigma(\cdot, \b \epsilon(\b u)) \|_{\b L^2(E)} 
\le C h_E^{k_2} |\b\sigma(\cdot, \b \epsilon(\b u))|_{\mathbb{W}^{k_2,2}(E)} \, .
\end{equation}
We now take the steps from \eqref{eq:T1A} but modify the bound for $\|\b\epsilon({\b \xi}_h^\perp) \|_{\b L^2(E)}$ using \eqref{bound:alternative} instead of \eqref{newkidinblock}, and apply \eqref{eq:pol-temp-bis} instead of \eqref{eq:pol-temp}. 
We obtain, also using a classical Young inequality,
\begin{equation}
\label{eq:T1A-bis}
\begin{aligned}
    T_1^A & \le {C\delta^{(2-r)/2}} \sum_{E\in \Omega_h} h_E^{k_2} |\b\sigma(\cdot, \b \epsilon(\b u))|_{\mathbb{W}^{k_2,2}(E)} 
    S^E({\b \xi}_h^\perp,{\b \xi}_h^\perp)^{1/2}
    \\
   &
   \le {\delta^{2-r}} \frac{C}{2\theta^{r}}  h^{2 k_2} |\b\sigma(\cdot, \b \epsilon(\b u))|_{\mathbb{W}^{k_2,2}(\Omega_h)}^{2} + 
   \frac{\theta^r}{2r}S({\b \xi}_h^\perp,{\b \xi}_h^\perp) \, .
\end{aligned}
\end{equation}
The loading term $T_3$ can be handled by analogous modifications.
The rest of the proof then follows as in Theorem~\ref{theo:main}.
\end{proof}

% ----------------------------------------------------
\subsection{\emph{A priori} error estimate: pressure}

\begin{theorem}\label{theo:main:2} 
Let $(\b u, p)$ be the solution of problem \eqref{eq:stokes.weak} and $(\b u_h, p_h)$ the solution of problem \eqref{eq:stokes.vem}. 
Assume that $\b u \in \b W^{k_1+1,r}(\Omega_h)$, 
$\b\sigma(\cdot, \b \epsilon(\b u)) \in \mathbb{W}^{k_2,r'}(\Omega_h)$, 
$\b f \in \b W^{k_3+1,r'}(\Omega_h)$, and 
$p \in  W^{k_4,r'}(\Omega_h)$, for some $k_1$, $k_2$, $k_3$, $k_4 \le k$. Let the mesh regularity  in Assumption~\ref{ass:mesh} hold. Then, we have
\begin{equation}
\begin{aligned}
    \|p - p_h\|_{L^{r'}(\Omega_h)} &\lesssim 
     h^{k_2}R_2+ h^{k_3+2}R_3+(\delta+\tri  \b u-\b u_h\tri_{\delta,r}+R_4)^{r-2}(\tri  \b u-\b u_h\tri_{\delta,r}+h^{k_1} R_1)\\
&+ (\delta^r+ { h^{2k_1} (\delta+h^{k_1}R_1 + R_4)^{r-2}R_1^2} +
    h^{k_2{r'}} R_2^{r'} +
      h^{{ (k_3+2)}{r'}} R_3^{r'})^\frac{r-2}{2r}\\ 
      &\times ( { h^{2k_1} (\delta+h^{k_1}R_1 +R_4)^{r-2}R_1^2} +
    h^{k_2{r'}} R_2^{r'} +
      h^{{ (k_3+2)}{r'}} R_3^{r'})^\frac{1}{2}
  +
    h^{k_4} R_5  \,,
    \end{aligned}
\end{equation}
where the regularity term $R_1$, $R_2$, $R_3$, and $R_4$ are defined in \eqref{eq:reg:terms} and
$R_5 = |p |_{W^{k_4,r'}(\Omega_h)}$.
\end{theorem}

\begin{proof}
Let $p_I = \Pi^0_{k-1} p \in \Pk_{k-1}(\Omega_h)$  and let $\rho_h=p_h-p_I$. 
For the sake of brevity also in this proof we employ the notation ${\b v}^\perp=(I - \Pi^0_k) \b v$ for any $\b v \in \b L^2(\Omega)$.
Employing  \eqref{eq:stokes.weak} and  \eqref{eq:stokes.vem}, recalling the definition of the form $b(\cdot, \cdot)$ in \eqref{eq:a.b} and combining item $(ii)$ in \eqref{eq:v-loc} with the definition of $L^2$-projection, for all $\b w_h\in \b U_h$ we get
   \begin{equation}
   \label{eq:press0}
   \begin{aligned}
   b(\b w_h,\rho_h) &= - a_h(\b u_h, \b w_h) + (\b f_h,\b w_h)
   +a(\b u, \b w_h)  - (\b f,\b w_h) + b(\b w_h,p-p_I) & &
   \\
   & = -a_h(\b u_h, \b w_h)+  a(\b u, \b w_h) + (\b f_h -\b f,\b w_h)
   \\
   & = \int_\Omega \left((I - \b \Pi^0_{k-1})\b\sigma(\cdot, \b \epsilon(\b u)) \right): (I - \b \Pi^0_{k-1})\b\epsilon( \b \w_h) +   (\b f_h -\b f,\b w_h)  
       \\
      & \qquad + \int_\Omega \left(\b\sigma(\cdot, \b \Pi^0_{k-1}\b \epsilon(\b u)) -
        \b\sigma(\cdot, \b \epsilon(\b u_h))
        \right): \b \Pi^0_{k-1} \b\epsilon( \b \w_h) -
       S({\b u}_h^\perp,{\b w}_h^\perp)  
    \\
    & =: T_1 + T_2 + T_3 + T_4 \,.
   \end{aligned}
   \end{equation}
In the first identity above, we have taken $\b v = \b w_h$ as a test function in the continuous weak problem \eqref{eq:stokes.weak} even if, as observed in Remark~\ref{rem:nonincl}, $\b w_h\in \b W^{1,2}(\Omega)$ but $\b w_h\notin \b W^{1,r}(\Omega)$. This is possible due to the additional regularity assumptions on the forcing term $\b f \in \b W^{1,r'}(\Omega_h)\subset \b L ^2(\Omega)$ and exact stress field $\b\sigma(\cdot, \b \epsilon(\b u)) \in \mathbb W^{1,r'}(\Omega_h)\subset \mathbb L ^2(\Omega)$.
We estimate separately each term in \eqref{eq:press0}.
Employing the H\"older inequality with exponents $(r',r)$ and polynomial approximation properties  we infer
   \begin{equation}
   \label{eq:pressT1}
   \begin{aligned}
       T_1  & \le \sum_{E\in \Omega_h} { \| (I - \b \Pi^{0,E}_{k-1})\b\sigma(\cdot, \b \epsilon(\b u)) \|_{\b L^2(E)}}  \|(I - \b \Pi^{0,E}_{k-1})\b\epsilon(\b w_h) \|_{\b L^2(E)} 
    & \quad & \text{(Cauchy--Schwarz ineq.)}
        \\
   & \lesssim \sum_{E\in \Omega_h} h_E^{k_2 -1} |\b\sigma(\cdot, \b \epsilon(\b u))|_{\mathbb{W}^{k_2,1}(E)} \|\b\epsilon({\b w}_h^\perp) \|_{\b L^2(E)}
    & \quad & { \text{(Cont of $\Pi^{0,E}_{k-1}$ \& \eqref{eq:pol-temp})}}
    \\
       & \lesssim \sum_{E\in \Omega_h} h_E^{k_2-1} |E|^\frac{1}{r}|\b\sigma(\cdot, \b \epsilon(\b u))|_{\mathbb{W}^{k_2,r'}(E)}  |E|^{\frac{1}{2}-\frac{1}{r}}S^E({\b w}_h^\perp,{\b w}_h^\perp)^\frac{1}{r} 
    & \quad & \text{($(r',r)$-H\"older ineq. \& \eqref{newkidinblock})}
    \\
   &
   \lesssim h^{k_2} |\b\sigma(\cdot, \b \epsilon(\b u))|_{\mathbb{W}^{k_2,r'}(\Omega_h)} S({\b w}_h^\perp,{\b w}_h^\perp)^\frac{1}{r}
    & \quad & \text{($|E|^\frac{1}{2} \simeq h_E$ \& $(r',r)$-Young ineq.)}\\
           &\lesssim  h^{k_2} R_2 \,   (\delta+\tri \b w_h \tri_{r})^\frac{r-2}{r}\tri \b w_h \tri_{r}^\frac{2}{r} \,. & \quad &\text{(apply \eqref{eq:comparison_delta_r_norm})}
          \end{aligned}
   \end{equation}   
 On the other hand,
   \begin{equation}
   \label{eq:pressT2}
   \begin{aligned}
   T_2 &= \sum_{E \in \Omega_h}\int_E (\Pi^{0,E}_k \b f - \b f) \cdot {\b w}_h^\perp 
    & \quad & \text{(def. \eqref{eq:forma fh} \& def. \eqref{eq:P0_k^E})}
    \\
   & \le { \sum_{E\in \Omega_h} \|(\Pi^{0,E}_k \b f - \b f) \|_{\b L^2(E)}} \|{\b w}_h^\perp\|_{\b L^2(E)} 
    % & \quad & \text{(Cauchy--Schwarz ineq.)} \\ & 
   \le C \sum_{E\in \Omega_h} h_E^{k_3} |\b f|_{\mathbb{W}^{{ k_3+1},1}(E)} h_E|\DOF({\b w}_h^\perp)|
  & & \text{(same as in \eqref{eq:T3})}
    \\
       & \lesssim \sum_{E\in \Omega_h} h_E^{k_3} |E|^\frac{1}{r}|\b f|_{\mathbb{W}^{{ k_3+1},r'}(E)}  h_E^{1+\frac{r-2}{r}}S^E({\b w}_h^\perp,{\b w}_h^\perp)^\frac{1}{r} 
    & \quad & \text{($(r',r)$-H\"older ineq. \& Lemma \ref{Lemma:S_holder_monotonicity})}
    \\
   &
   \lesssim h^{k_3+2} |\b f|_{\mathbb{W}^{k_3+1,r'}(\Omega_h)} 
  S({\b w}_h^\perp,{\b w}_h^\perp)^\frac{1}{r}
    & \quad & \text{($|E|^\frac{1}{2} \simeq h_E$ \& $(r',r)$-Young ineq.)}
     \\
   &
   \lesssim h^{k_3+2}R_3 (\delta+\tri \b w_h \tri_{r})^\frac{r-2}{r}\tri \b w_h \tri_{r}^\frac{2}{r}\,.& \quad &\text{(apply \eqref{eq:comparison_delta_r_norm})}
   \end{aligned}
   \end{equation}

   Employing \eqref{eq:hypo.continuity},  the 3-terms $(\frac{r}{r-2},r,r)$-H\"older inequality, the triangle inequality together with the $L^r$-stability of $\Pi^0_{k-1}$ we have
   \begin{equation}
   \label{eq:pressT3}
   \begin{aligned}
       T_3 &\lesssim
       \int_\Omega (\delta+|\b \Pi^0_{k-1} \b \epsilon(\b u_h)|+|\b \epsilon(\b u)|)^{r-2} |\b \Pi^0_{k-1} \b \epsilon(\b u_h)-\b \epsilon(\b u))| \, |\b \Pi^0_{k-1} \b\epsilon( \b \w_h)| & \quad &\\
       &\lesssim
       (\delta^r+\|\b \Pi^0_{k-1} \b \epsilon(\b u_h)\|_{\mathbb L^{r}(\Omega)}^r+\|\b \epsilon(\b u)\|_{\mathbb L^{r}(\Omega)}^r)^\frac{r-2}{r}\|\b \Pi^0_{k-1} \b \epsilon(\b u_h)-\b \epsilon(\b u)\|_{\mathbb L^{r}(\Omega)} \,\|\b \Pi^0_{k-1} \b\epsilon( \b \w_h)\|_{\mathbb L^{r}(\Omega)} & \quad & 
      \\
      &\lesssim
       (\delta^r+\tri  \b u-\b u_h\tri_{\delta,r}^r +\|\b \epsilon(\b u)\|_{\mathbb L^{r}(\Omega)}^r)^\frac{r-2}{r}(\tri  \b u-\b u_h\tri_{\delta,r}+\|(I-\b \Pi^0_{k-1})\b \epsilon(\b u)\|_{\mathbb L^{r}(\Omega)}) \, \tri\b w_h\tri_{r}& \quad &\\
             &\lesssim
       (\delta^r+\tri  \b u-\b u_h\tri_{\delta,r}^r+R_4^r)^\frac{r-2}{r}(\tri  \b u-\b u_h\tri_{\delta,r}+h^{k_1} R_1) \, (\delta+\tri \b w_h \tri_{r})^\frac{r-2}{r}\tri \b w_h \tri_{r}^\frac{2}{r} \\
       &\eqcolon  R_{1,h}(\delta+\tri \b w_h \tri_{r})^\frac{r-2}{r}\tri \b w_h \tri_{r}^\frac{2}{r} \,.& \quad &\\
   \end{aligned}    
   \end{equation}
Using \eqref{eq:S_ineq_1} 
%Lemma \ref{Lemma:S_holderG}, Corollary~\ref{Cor:vemstab},  the $W^{1,r}$-stability of $\Pi^0_{k-1}$ and finally
and Lemma \ref{Lemma:B},
we obtain 
\begin{equation}
\begin{aligned}
    T_4 
    &\lesssim \big(\delta^r+S({\b u}_h^\perp, {\b u}_h^\perp)\big)^\frac{r-2}{2r}S({\b u}_h^\perp, {\b u}_h^\perp)^\frac{1}{2}S({\b w}_h^\perp, {\b w}_h^\perp)^\frac{1}{r}
    \\
    &\lesssim  \big(\delta^r+S({\b u}_h^\perp, {\b u}_h^\perp)\big)^\frac{r-2}{2r}S({\b u}_h^\perp, {\b u}_h^\perp)^\frac{1}{2}
    \tri \b \w_h\tri_{\delta,r}
    \\
    & \lesssim
    \big(\delta^r+S({\b u}_h^\perp, {\b \xi}_h^\perp) - 
    S({\b u}_I^\perp, {\b \xi}_h^\perp) + 
      R_{h}\big)^\frac{r-2}{2r}\left( S({\b u}_h^\perp, {\b \xi}_h^\perp) - 
    S({\b u}_I^\perp, {\b \xi}_h^\perp) + 
      R_{h} \right)^\frac{1}{2}
     \tri \b \w_h\tri_{\delta,r} \\
    & \lesssim
    \big(\delta^r+{a_h({\b u}_h, {\b \xi}_h)} - 
    a_h({\b u}_I, {\b \xi}_h) + 
      R_{h}\big)^\frac{r-2}{2r}\left( {a_h({\b u}_h, {\b \xi}_h)} - 
    a_h({\b u}_I, {\b \xi}_h) + 
      R_{h} \right)^\frac{1}{2}
     \tri \b \w_h\tri_{\delta,r} \, ,
\end{aligned}
\end{equation}
which, using the bound for 
$(a_h({\b u}_h, {\b \xi}_h) -a_h({\b u}_I, {\b \xi}_h))$ encased in the proof of Theorem~\ref{theo:main}, check equation \eqref{eq:conv0}, becomes
\begin{equation}     
\label{eq:pressT4}
\begin{aligned}
      T_4 & \lesssim
     (\delta^r+R_{h}' +
    h^{{k_2}{r'}} R_2^{r'}+  h^{{ (k_3+2)}{r'}} R_3^{r'})^\frac{r-2}{2r}(R_{h} +
    h^{{k_2}{r'}} R_2^{r'}+  h^{{ (k_3+2)}{r'}} R_3^{r'})^\frac{1}{2}
     {(\delta+\tri \b w_h \tri_{r})^\frac{r-2}{r}\tri \b w_h \tri_{r}^\frac{2}{r}} \\
     &\eqcolon R_{2,h} (\delta+\tri \b w_h \tri_{r})^\frac{r-2}{r}\tri \b w_h \tri_{r}^\frac{2}{r}
     \,,
\end{aligned}    
\end{equation}
where $R_{h}' = h^{2k_1} (\delta+h^{k_1}R_1 + R_4)^{r-2}R_1^2$.
Employing the discrete inf-sup of Lemma \ref{inf-sup:vem}, and collecting \eqref{eq:pressT1}, \eqref{eq:pressT2}, \eqref{eq:pressT3}, and \eqref{eq:pressT4} in \eqref{eq:press0}, we obtain
\begin{equation}\label{eq:proof_press_b}
   \begin{aligned}
   \|\rho_h\|_{L^{r'}(\Omega_h)}
    &   \le \frac{1}{\bar{\beta}(r)}
       \sup_{\b w_h \in \VDG, \tri {\b w_h} \tri_{r} = 1}
       b(\b w_h,\rho_h)
       \\
    &  \lesssim 
       \sup_{\b w_h \in \VDG, \tri {\b w_h} \tri_{r} = 1}
       \left[\big( h^{k_2}R_2+ h^{k_3+2}R_3+R_{1,h}+R_{2,h}
\big) 
      (\delta+\tri \b w_h \tri_{r})^\frac{r-2}{r}\tri \b w_h \tri_{r}^\frac{2}{r}\right]
       \\
   &     = 
       \big( h^{k_2}R_2+ h^{k_3+2}R_3+R_{1,h}+R_{2,h})   (\delta+1)^\frac{r-2}{r}   \,.
   \end{aligned}    
   \end{equation}
Now the thesis follows from triangular inequality and standard polynomial approximation properties
\begin{equation}
    \label{eq:bound_p_ph}
    \Vert p - p_h\Vert_{L^{r'}(\Omega_h)} \leq \Vert p - p_I\Vert_{L^{r'}(\Omega_h)}  + \Vert \rho_h\Vert_{L^{r'}(\Omega_h)} \leq  h^{k_4} R_5 + \Vert \rho_h\Vert_{L^{r'}(\Omega_h)} \,.
\end{equation}
\end{proof}

\begin{corollary}\label{cor:main:2} 
Under the same assumptions of Theorem~\ref{theo:main}, if in addition $\delta>0$, $\b\sigma(\cdot, \b \epsilon(\b u)) \in \mathbb{W}^{k_2,2}(\Omega_h)$ and $f \in \mathbb{W}^{k_3+1,2}(\Omega_h)$, it holds
\begin{equation}\label{eq:est_press_delta_nonzero}
\begin{aligned}
    \|p - p_h\|_{L^{r'}(\Omega_h)} &\lesssim 
     h^{k_2} \delta^{\frac{2-r}{2}}\widehat{R}_2 + h^{k_3+2} \delta^{\frac{2-r}{2}}\widehat{R}_3+(\delta+\tri  \b u-\b u_h\tri_{\delta,r}+R_4)^{r-2}(\tri  \b u-\b u_h\tri_{\delta,r}+h^{k_1} R_1)\\
&+ (\delta^r+ h^{2k_1} (\delta+h^{k_1}R_1 + R_4)^{r-2}R_1^2 +
    h^{2k_2} \delta^{\frac{2-r}{2}}\widehat{R}_2^{2} + h^{2(k_3+2)} \delta^{\frac{2-r}{2}}\widehat{R}_3^2)^\frac{r-2}{2r}\\ 
      &\times ( h^{2k_1} (\delta+h^{k_1}R_1 + R_4)^{r-2}R_1^2 +
    h^{2k_2} \delta^{\frac{2-r}{2}}\widehat{R}_2^{2} + h^{2(k_3+2)} \delta^{\frac{2-r}{2}}\widehat{R}_3^2)^\frac{1}{2}
 +
    h^{k_4} R_5  \,,
    \end{aligned}
\end{equation}
where the regularity terms $R_1$ and $R_4$ are from \eqref{eq:reg:terms}, $\widehat{R}_2$ and $\widehat{R}_3$ from \eqref{eq:R_2-3-hat}, and
$R_5 = |p |_{W^{k_4,r'}(\Omega_h)}$.
\end{corollary}

\begin{proof}
    The proof follows the same steps as that of Theorem~\ref{theo:main:2} but applying Corollary~\ref{cor:main} instead of Theorem~\ref{theo:main}.
\end{proof}

\begin{lemma}\label{Lemma:B}
Under the same assumptions of Theorem~\ref{theo:main}, let $\b u_I \in \VDG$ be the interpolant of $\b u$ (cf. Lemma \ref{lem:approx-interp}) and $\b \xi_h =\b u_h - \b u_I$, then the following holds
\begin{equation}
    S({\b u}_h^\perp, {\b u}_h^\perp)\lesssim
    S({\b u}_h^\perp, {\b \xi}_h^\perp) - 
    S({\b u}_I^\perp, {\b \xi}_h^\perp) + 
       h^{2k_1} (\delta+h^{k_1}R_1)^{r-2}R_1^2 \,,
\end{equation}
where the regularity term $R_1$ is from \eqref{eq:reg:terms}.
\end{lemma}

\begin{proof}
Let $\widehat{\b \sigma}(\b x)= (h_E\delta+|\b x|)^{r-2} \b x$ for any $\b x \in \R^{N_E}$. 
Then simple computations yield
\begin{equation}
    \label{eq:lemmaB0}
    \begin{aligned}
S({\b u}_h^\perp, {\b u}_h^\perp) &\simeq 
\sum_{E \in \Omega_h} h_E^{2-r}(h_E\delta+ |\DOF({\b u}_h^\perp)|)^{r-2}|\DOF({\b u}_h^\perp)|^2 =
\sum_{E \in \Omega_h} h_E^{2-r}\widehat{\b \sigma}(\DOF({\b u}_h^\perp)) \cdot \DOF({\b u}_h^\perp)
\\
&=
\sum_{E \in \Omega_h}   h_E^{2-r}\bigl(\widehat{ \b \sigma}(\DOF({\b u}_h^\perp)) -
\widehat{\b \sigma}(\DOF({\b u}_I^\perp)) \bigr)\cdot \DOF({\b u}_h^\perp) +
\sum_{E \in \Omega_h}  h_E^{2-r}\widehat{\b \sigma} (\DOF({\b u}_I^\perp)) \cdot \DOF({\b u}_h^\perp)
\\
& =: T_1 + T_2 \,.
    \end{aligned}
\end{equation}
In the following $C$ will denote a generic
positive constant independent of $h$ that may change at each occurrence, whereas the parameter $\varepsilon$ adopted in \eqref{eq:lemmaBT1} and \eqref{eq:lemmaBT2} will be specified later.
Using Lemma \ref{lem:young_shifted} and Lemma \ref{lm:Hirn} we infer
\begin{equation}
    \begin{aligned}
T_1 & \le 
    C \sum_{E \in \Omega_h} h_E^{2-r}\varphi'_{h_E\delta+|\DOF({\b u}_h^\perp)|}(|\DOF({\b \xi}_h^\perp)|) \,  |\DOF({\b u}_h^\perp)| 
    & \quad & \text{(by \eqref{eq:extrass.continuity})}
    \\
    & \le
    \varepsilon 
    \sum_{E \in \Omega_h} 
    h_E^{2-r}\varphi_{h_E\delta+|\DOF({\b u}_h^\perp)|}(|\DOF({\b u}_h^\perp)|)  +
    C_\varepsilon  \sum_{E \in \Omega_h} 
    h_E^{2-r}\varphi_{h_E\delta+|\DOF({\b u}_h^\perp)|}(|\DOF({\b \xi}_h^\perp)|)
    & \quad & \text{(by \eqref{eq:Young1})}
    \\
    &\le
    \gamma \varepsilon 
    \sum_{E \in \Omega_h} h_E^{2-r}(h_E\delta+|\DOF({\b u}_h^\perp)|)^{r-2}|\DOF({\b u}_h^\perp)|^2 + 
    C_\varepsilon  \sum_{E \in \Omega_h} 
    h_E^{2-r}(h_E\delta+|\DOF({\b u}_h^\perp)| + |\DOF({\b u}_I^\perp)|)^{r-2} \, |\DOF({\b \xi}_h^\perp)|^2 ,
    & \quad & \text{(by \eqref{eq:extrass.monotonicity})}
\end{aligned}
\end{equation}
where in the last line $\gamma$ denotes the uniform hidden positive constant from Lemma \ref{lm:Hirn}.
Applying a consequent bound of \eqref{eq:strong1} (from the second row to the last row) yields
\begin{equation}
    \label{eq:lemmaBT1}
    \begin{aligned}   
    T_1 &\le 
    \gamma \varepsilon S({\b u}_h^\perp, {\b u}_h^\perp)  + 
    C_\varepsilon  \sum_{E \in \Omega_h} h_E^{2-r}
    (h_E\delta+|\DOF({\b u}_h^\perp)| + |\DOF({\b u}_I^\perp)|)^{r-2} \, |\DOF({\b \xi}_h^\perp)|^2 
    \\
  &\le 
    \gamma \varepsilon S({\b u}_h^\perp, {\b u}_h^\perp)  + 
    C_\varepsilon  (S({\b u}_h^\perp, {\b \xi}_h^\perp) -
  S({\b u}_I^\perp, {\b \xi}_h^\perp)) 
    \, .
    \end{aligned}
\end{equation}
Using analogous arguments we have
\begin{equation}
    \label{eq:lemmaBT2}
    \begin{aligned}
T_2 & \le 
    C \sum_{E \in \Omega_h} h_E^{2-r} \varphi'_{h_E\delta+|\DOF({\b u}_I^\perp)|}(|\DOF({\b u}_I^\perp)|) \,  |\DOF({\b u}_h^\perp)| 
    & \quad & \text{(by \eqref{eq:extrass.continuity})}
    \\
    & \le
    \varepsilon 
    \sum_{E \in \Omega_h} h_E^{2-r}
    \varphi_{h_E\delta}(|\DOF({\b u}_h^\perp)|)  +
    C_\varepsilon  \sum_{E \in \Omega_h} h_E^{2-r}
    \varphi_{h_E\delta}(|\DOF({\b u}_I^\perp)|)
    & \quad & \text{(by \eqref{eq:Young1})}
    \\
    &\le
    \gamma \varepsilon 
    \sum_{E \in \Omega_h} h_E^{2-r} (h_E\delta+|\DOF({\b u}_h^\perp)|)^{r-2}|\DOF({\b u}_h^\perp)|^2 + 
    C_\varepsilon  \sum_{E \in \Omega_h} h_E^{2-r} (h_E\delta+|\DOF({\b u}_I^\perp)|)^{r-2}|\DOF({\b u}_I^\perp)|^2 
    & \quad & \text{(by \eqref{eq:extrass.monotonicity})}
    \\
    &\le 
    \gamma \varepsilon S({\b u}_h^\perp, {\b u}_h^\perp)  + 
    C_\varepsilon \sum_{E \in \Omega_h} h_E^{2-r} (h_E\delta+|\DOF({\b u}_I^\perp)|)^{r-2}|\DOF({\b u}_I^\perp)|^2  \,,    
    \end{aligned}
\end{equation}
Taking in \eqref{eq:lemmaBT1} and \eqref{eq:lemmaBT2} $\epsilon = \frac{1}{4 \gamma}$, from \eqref{eq:lemmaB0} we obtain
\[
 \begin{aligned}
S({\b u}_h^\perp, {\b u}_h^\perp) &\lesssim
S({\b u}_h^\perp, {\b \xi}_h^\perp) -
  S({\b u}_I^\perp, {\b \xi}_h^\perp)  +
 \sum_{E \in \Omega_h} h_E^{2-r} (h_E\delta+|\DOF({\b u}_I^\perp)|)^{r-2}|\DOF({\b u}_I^\perp)|^2 \\
 &\lesssim
  S({\b u}_h^\perp, {\b \xi}_h^\perp) -
  S({\b u}_I^\perp, {\b \xi}_h^\perp) +
   h^{2k_1} (\delta+h^{k_1}R_1)^{r-2}R_1^2\,.  
  \end{aligned}
\] 
where we used the bound \eqref{eq:T2B} of $T_2^B$.
\end{proof}

\begin{remark}[Orders of convergence]\label{rem:orders}
    We comment on the orders of convergence of velocity and pressure under enough regularity, i.e., if $k_1=k_2=k_3=k$.
    Employing Theorem~\ref{theo:main} and Corollary~\ref{cor:main},  for $h$ asymptotically small and  
    for any $r \ge 2$ and $\delta \ge 0$, the order of convergence of the velocity is:
    \begin{equation}\label{eq:vel:orders}
    \begin{aligned}
    & \mathcal{O}(u) = \min\left(\frac{k}{r-1},\frac{2k}{r}\right) = \frac{k}{r-1} \qquad \textrm{valid for } \delta \ge 0 \, , \\
    & \mathcal{O}(u) = \frac{2k}{r}
    \qquad \textrm{valid for } \delta > 0 \, .
    \end{aligned}
    \end{equation}
    In the first bound above we leave the minimum  expressed explicitly to shed light on the origin of the leading order $k/(r-1)$. Indeed, this order stems from the "$\sigma$ approximation" term $T^1_A$ (cf. proof of Theorem~\ref{theo:main} ) and therefore, in many situations, it is expected to dominate the estimate only asymptotically, but possibly not for practical mesh sizes (see also Section~\ref{sec:tests}). 
    Furthermore note that such ``$\sigma$-approximation term'' could be ameliorated by raising the order of the projection $\Pi^0_{k-1}$ appearing in the first addendum (consistency part) of \eqref{eq:forma ah}, that is using $\Pi^0_{\ell}$ with $\ell \ge k$. As a consequence, if $\sigma$ is sufficiently regular, a simple modification of bound \eqref{eq:T1A} would lead to the more favorable final bound 
    $$
    \mathcal{O}(u) = \min\left(\frac{\ell+1}{r-1},\frac{2k}{r}\right)
    \qquad \textrm{valid for } \delta \ge 0 \, .
    $$
    The above improvement can be achieved by suitably enhancing the virtual space, resulting in a more cumbersome computation of the local discrete forms (but not increasing the size of the global system).

For what concerns the pressure, from Theorem~\ref{theo:main:2} and Corollary~\ref{cor:main:2}, for $h$ asymptotically small and  
    for any $r \ge 2$ and $\delta \ge 0$, the order of convergence of the pressure is:
        \begin{equation}
    \mathcal{O}(p) = \mathcal{O}(u).
    \end{equation}
\end{remark}

\begin{remark}[The role of $\delta$]
    When $\delta$ is positive but small and comparable to $h$, the pre-asymptotic error reduction rate will be the outcome of the competition among the two bounds appearing in \eqref{eq:vel:orders}. In such cases, although the asymptotic convergence rate will clearly behave as $h^{{2k}/{r}}$, we may experience a slower pre-asymptotic error reduction rate, more similar to $h^{{k}/{r-1}}$.
    More precisely, a careful analysis reveals that when $\delta$ is small and comparable to $h$, the term $\delta^{\frac{2-r}{r}}$ in \eqref{eq:est_vel_delta_nonzero} (and \eqref{eq:est_press_delta_nonzero}) plays a role in the error reduction rate. In particular, % having in mind that the order of convergence for the velocity error is the maximum between the cases $\delta > 0$ and $\delta \ge 0$ in \eqref{eq:vel:orders}, 
    we obtain for $h<1$ that the velocity error is bounded by $Ch^\gamma$, with $C$ independent of $h, \delta$
 and     \begin{equation}\label{eq:OCV_delta_nonzero}
    \begin{aligned}
    & \gamma = \max\left(\frac{2k}{r}+\frac{2-r}{r}\frac{\ln(\min(\delta,1))}{\ln(h)},\frac{k}{r-1}\right) 
    \qquad \textrm{valid for } \delta > 0 \, ,
    \end{aligned}
    \end{equation}
    where we used $ \delta^\frac{2-r}{r} = e^{\frac{2-r}{r}\ln(\delta)} = e^{(\frac{2-r}{r}\frac{\ln(\delta)}{\ln(h)}) \ln(h)} = h^{\frac{2-r}{r}\frac{\ln(\delta)}{\ln(h)}}$ and we noted that if $\delta > 1$ then the other terms of \eqref{eq:est_vel_delta_nonzero} dominate the estimate, hence the minimum in \eqref{eq:OCV_delta_nonzero}. 
    % Therefore, we should observe a change of rate of convergence between the states $1 > \delta > h$ and $\delta < h < 1$ when $h$ decreases.
\end{remark}

\section{Numerical Results}
\label{sec:tests}
 In this section, we present three numerical tests to validate the theoretical results of Theorems~\ref{theo:main}  and~\ref{theo:main:2} (and the associated corollaries)
%and Corollaries~\ref{cor:W1r.estimate}, ~\ref{cor:main}, ~\ref{cor:main:2}  
for different values of the parameters $\delta$ and $r$, as well as of the Sobolev regularity indices $k_1$, $k_2$, $k_3$ and $k_4$.
To compute the VEM error between the exact solution $\b (u_{\rm ex}, p_{\rm ex})$ and the VEM solution $(\b u_h, p_h)$, we consider the computable error quantities 
\begin{equation}
    \label{eq:err_quant}
\begin{aligned}
\texttt{err}(\b u_h, \tri{\b \cdot}\tri_{r}) &=
\frac{\tri{\b u_{\rm ex} - \b u_h}\tri_r}{\tri{\b u_{\rm ex}}\tri_r}  \,,
\\
\texttt{err}(\b u_h, W^{1,r}) &=
\frac{\|\GRAD \b u_{\rm ex} - \b \Pi^{0}_{k-1} \GRAD \b u_h\|_{L^r(\Omega)}}{\|\GRAD \b u_{\rm ex}\|_{\mathbb{L}^r(\Omega)}}  \,,
\\
\texttt{err}(p_h, L^{r'}) &= \frac{\|p_{\rm ex} - p_h\|_{L^{r'}(\Omega)}}
{\|p_{\rm ex}\|_{L^{r'}(\Omega)}} \,,
\\
\texttt{err}(\b \sigma, L^{r'}) &=
\frac{\|\b\sigma(\cdot, \b \epsilon(\b u_{\rm ex})) - \b\sigma(\cdot, \b\Pi_{k-1}^0 \b \epsilon(\b u_h))\|_{L^{r'}(\Omega)}}
{\|\b\sigma(\cdot, \b \epsilon(\b u_{\rm ex}))\|_{\mathbb{L}^{r'}(\Omega)}}  \, .
\end{aligned}
\end{equation}
We make use of the $\tri \cdot \tri_r$ norm, which is bounded by the $\tri \cdot \tri_{\delta,r}$ norm, in order to have the same error measure in all tests.
Furthermore, note that we also include an error measure on the stress $\sigma$, although deriving a theoretical estimate for such a quantity is beyond the scope of the present contribution.\\

Given a sequence of $N+1$ meshes with mesh diameters $h_0 > \dots > h_N$, and denoting by $E_h$ any of the error quantities listed in \eqref{eq:err_quant}, we define the average experimental order of convergence \texttt{AEOC} as
\[
\texttt{AEOC}= \frac{1}{N} \sum_{n=1}^N 
\frac{\log(E_{h_{n-1}}/E_{h_{n}})}{\log(h_{n-1}/h_{n})} \,.
\]

As a model equation, we consider the Carreau-Yasuda model \eqref{eq:Carreau}, with $\alpha=2$ (i.e. corresponding to the Carreau model), $\mu=1$. In order to verify the apriori error estimates of Section~\ref{sec:error_analysis}, numerical tests are performed with the following values of $r$ and $\delta$:
\begin{equation}
\label{eq:r-delta}
r=
\texttt{2.00}, \,
\texttt{2.25}, \,
\texttt{2.50}, \,
\texttt{3.00}, 
\qquad 
\delta=
\texttt{1}, \,
\texttt{0} \,.
\end{equation}
In the forthcoming tests, we consider the scheme \eqref{eq:stokes.vem} with $k=2$. 
The nonlinear problem is solved by means of a two-step Picard-type iteration. First, we solve the
problem corresponding to $\bar r$, defined as the midpoint between $2$ and $r$, using as initial
guess the solution of the associated linear Stokes problem. Then, the solution obtained for $\bar r$ is employed as the initial iterate for a Picard iteration with exponent $r$.
An analogous strategy was adopted in \cite[Section 5.1]{Antonietti_et_al_2024} for the case $r \in (1, 2]$.
The domain $\Omega$ (specified in each test) is partitioned with the following sequences of polygonal meshes:
\texttt{QUADRILATERAL} distorted meshes,
\texttt{RANDOM} Voronoi meshes, and 
\texttt{CARTESIAN} meshes (see Fig.~\ref{fig:meshes}).
For the generation of the Voronoi meshes we used the code \texttt{Polymesher} \cite{polymesher}.
We emphasize that, for the families of meshes under consideration, all mesh elements are convex, therefore, according to Remark~\ref{rem:nonincl}, the discrete solution satisfies
$\b u_h \in \b W^{1,r}$.
\begin{figure}[!htbp]
\centering
\begin{overpic}[scale=0.23]{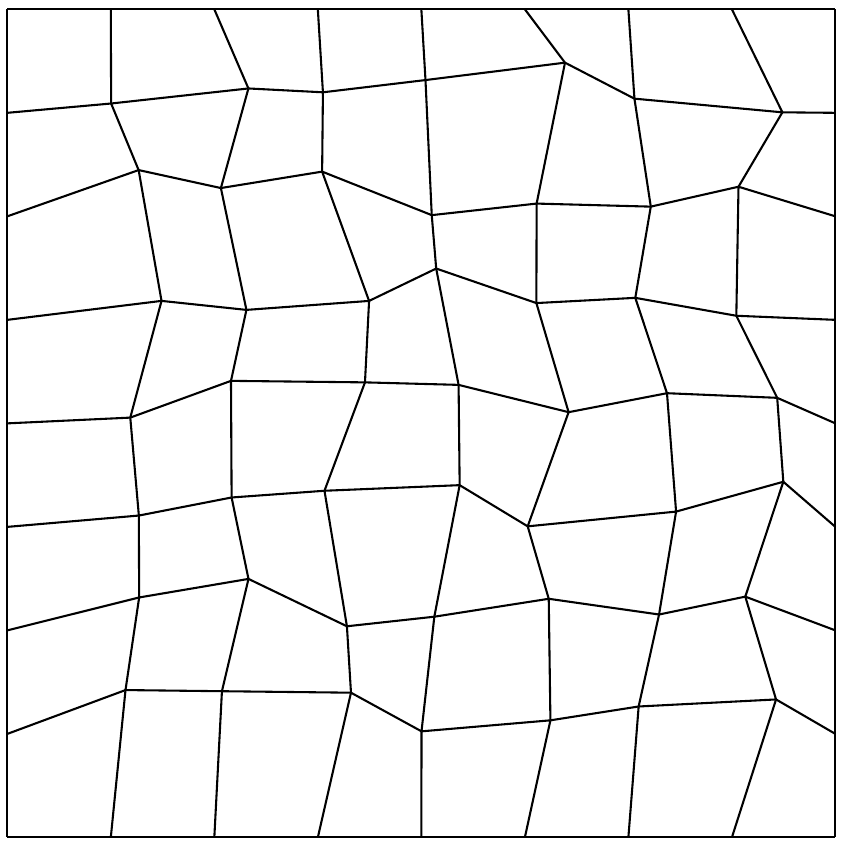}
\put(10,-15){{{\texttt{QUADRILATERAL}}}}
\end{overpic}
\qquad
\begin{overpic}[scale=0.23]{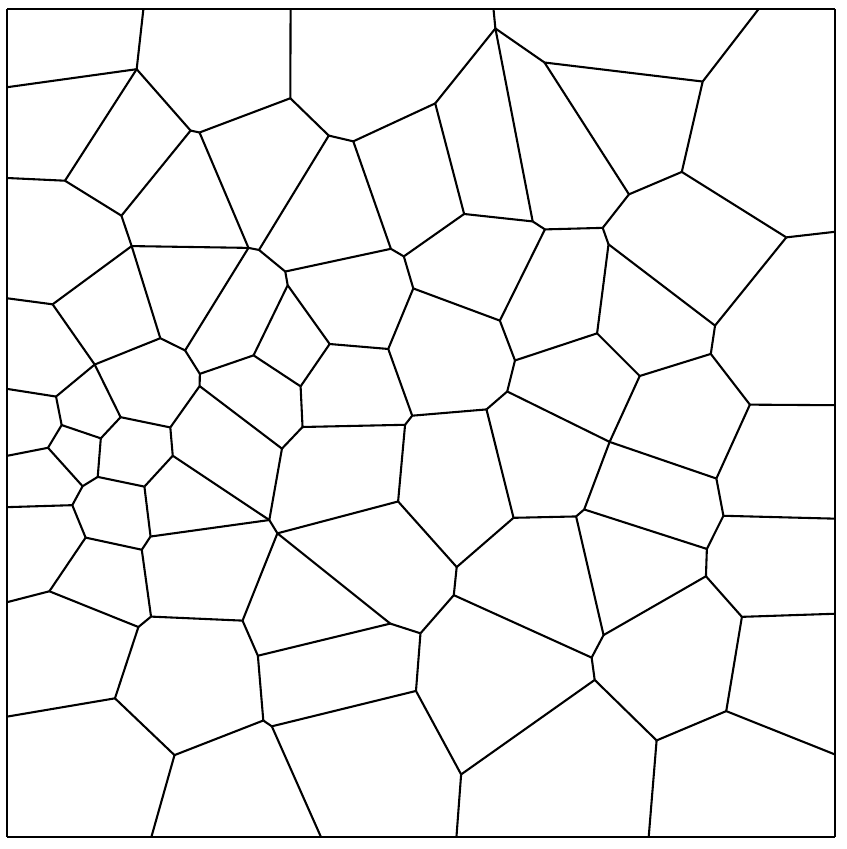}
\put(30,-15){{{\texttt{RANDOM}}}}
\end{overpic}
\qquad
\begin{overpic}[scale=0.20]{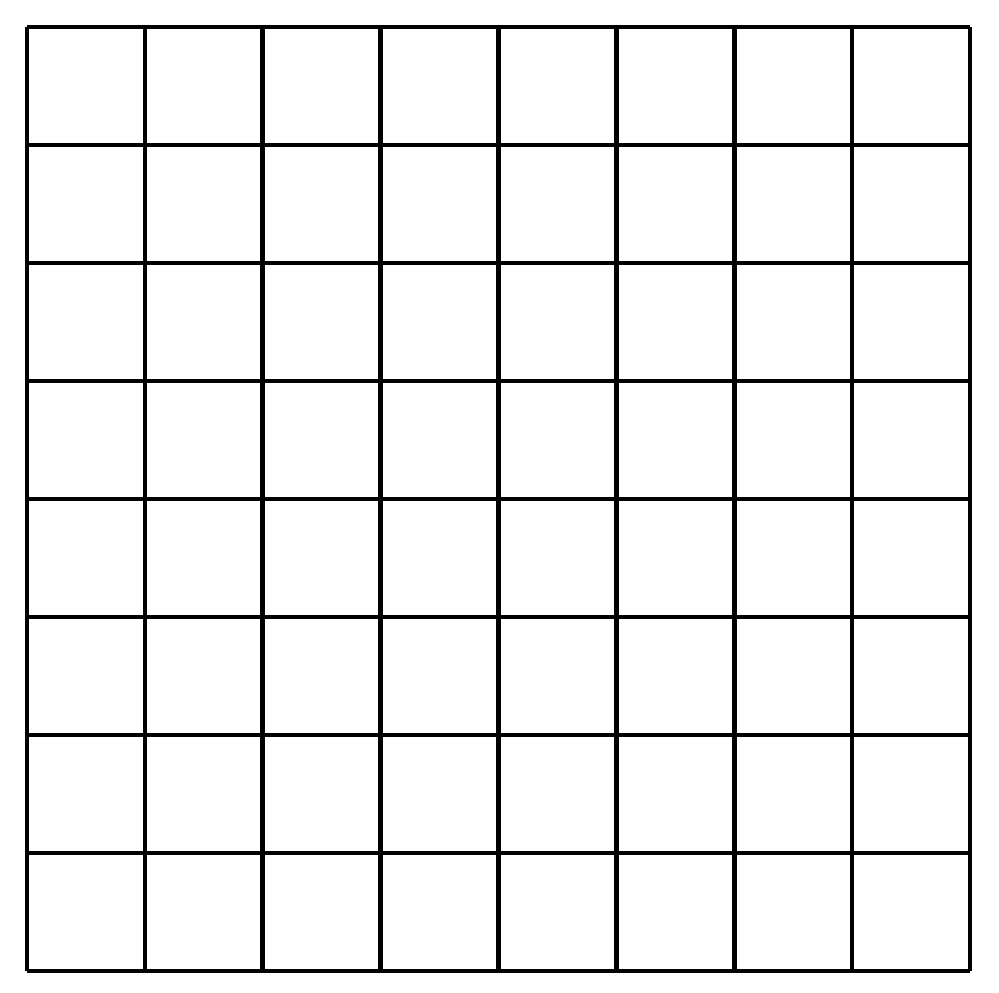}
\put(20,-15){{{\texttt{CARTESIAN}}}}
\end{overpic}
\vspace{0.5cm}
\caption{Example of the adopted polygonal meshes.}
\label{fig:meshes}
\end{figure}

\subsection{Test 1. Regular solution}
In the first test case, we consider Problem \eqref{eq:stokes.continuous} 
with full Dirichlet boundary conditions (i.e. $\Gamma_D = \partial \Omega$) on the unit square $\Omega = (0, 1)^2$. The load terms $\b f$ (depending on $r$ and $\delta$ in \eqref{eq:Carreau}) and the Dirichlet boundary conditions are chosen according to the analytical solution
\[
\b u_{\rm ex}(x_1,x_2) = 
\begin{bmatrix}
\sin\bigl(\frac{\pi}{2}x_1\bigr)\cos\bigl(\frac{\pi}{2}x_2\bigr)
\\
-\cos\bigl(\frac{\pi}{2}x_1\bigr)\sin\bigl(\frac{\pi}{2}x_2\bigr)
\end{bmatrix} \,,
\qquad
p_{\rm ex}(x_1,x_2) =  -\sin\biggl(\frac{\pi}{2}x_1\biggr)
\sin\biggl(\frac{\pi}{2}x_2\biggr) + \frac{4}{\pi^2} \,.
\]
The domain $\Omega$ is partitioned using the family of \texttt{QUADRILATERAL} distorted
meshes and the family of \texttt{RANDOM} meshes. For each mesh family, we consider a mesh sequence with diameters
$\texttt{h} = \texttt{1/4}, \, \texttt{1/8}, \, \texttt{1/16}, \, \texttt{1/32}, \, \texttt{1/64}$.
In Fig.~\ref{fig:test1-Q} and Fig.~\ref{fig:test1-R}, we plot the computed error quantities in \eqref{eq:err_quant} for the sequences of aforementioned meshes and parameters $r$ and $\delta$ as in \eqref{eq:r-delta}.
We observe that for $\delta = \texttt{1}$ (left panels of Fig.~\ref{fig:test1-Q} and Fig.~\ref{fig:test1-R}), a convergence rate of order \texttt{2} is observed, 
whereas for $\delta = \texttt{0}$ (left panels of Fig.~\ref{fig:test1-Q} and Fig.~\ref{fig:test1-R}), the plot shows the average experimental orders of convergence  \texttt{AEOC}.
%%%%%%%%% TEST 1 - QUAD %%%%%%%%%
\begin{figure}[htbp]
    \centering
    {\texttt{QUADRILATERAL MESHES}}
\\
    % Riga 1
    \begin{subfigure}{0.45\textwidth}
        \centering
        \includegraphics[width=\linewidth]{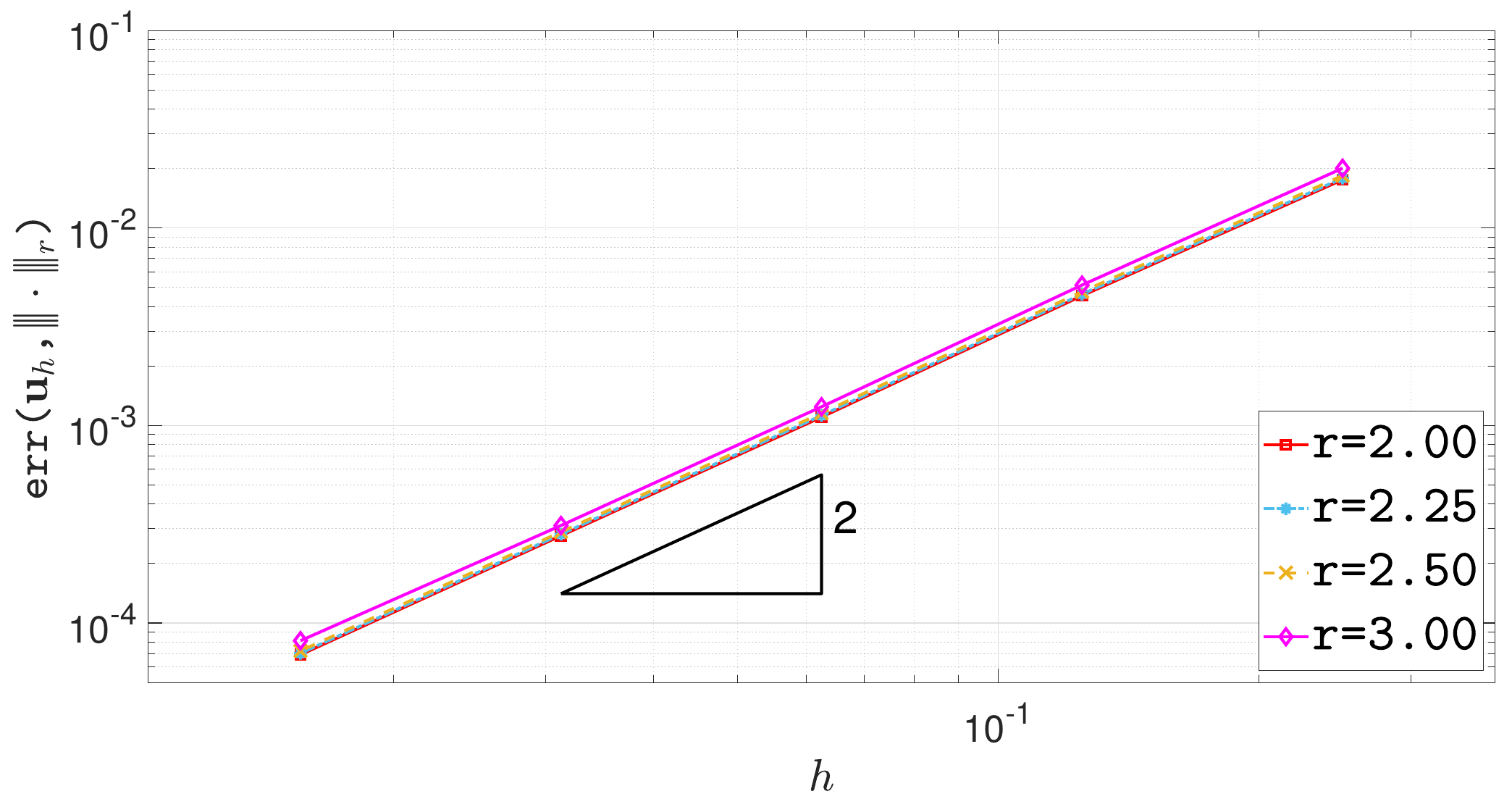}
        
    \end{subfigure}\qquad
    \begin{subfigure}{0.45\textwidth}
        \centering
        \includegraphics[width=\linewidth]{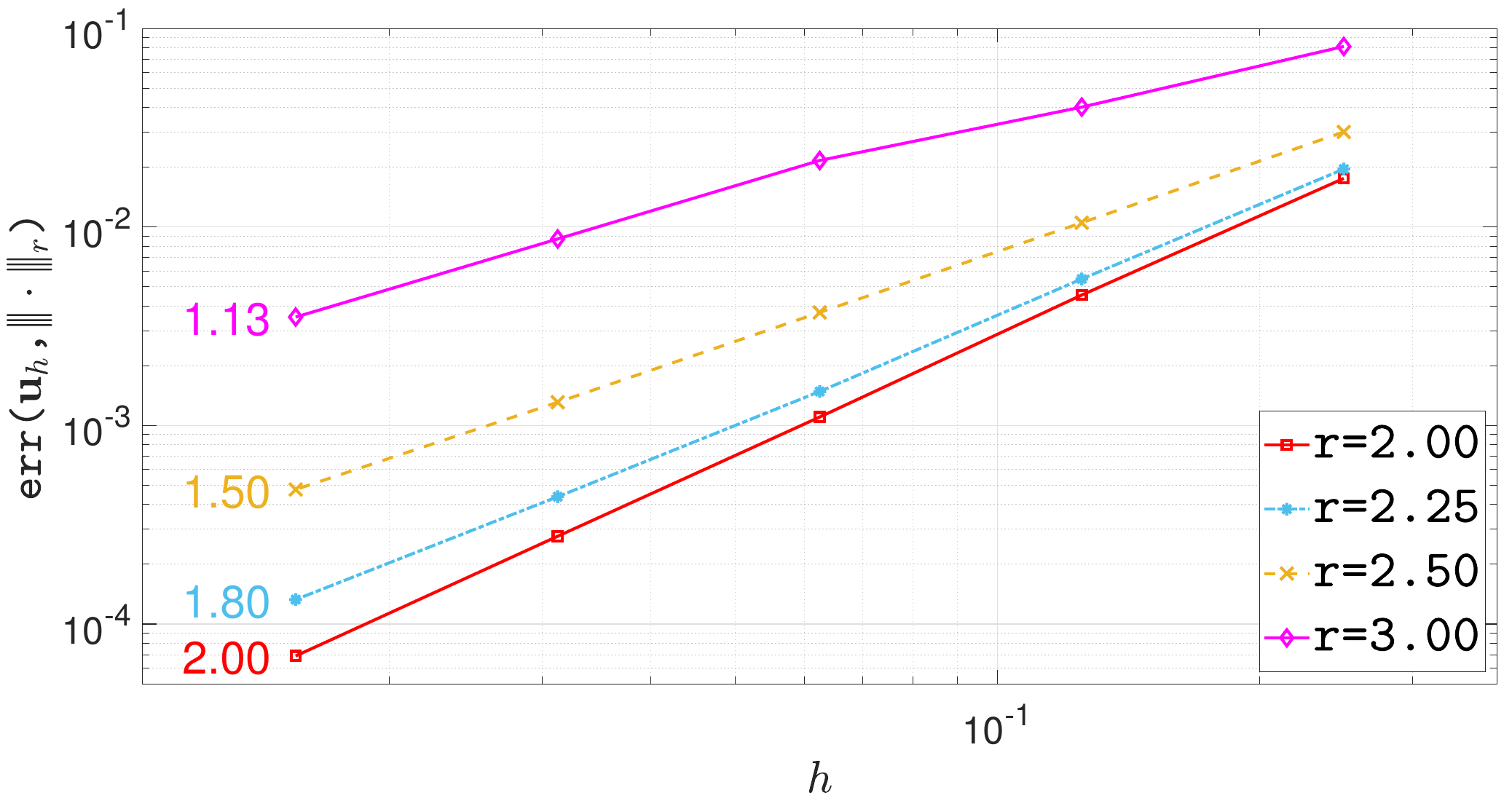}
    \end{subfigure}

    \vspace{0.5cm}

    % Riga 2
    \begin{subfigure}{0.45\textwidth}
        \centering
        \includegraphics[width=\linewidth]{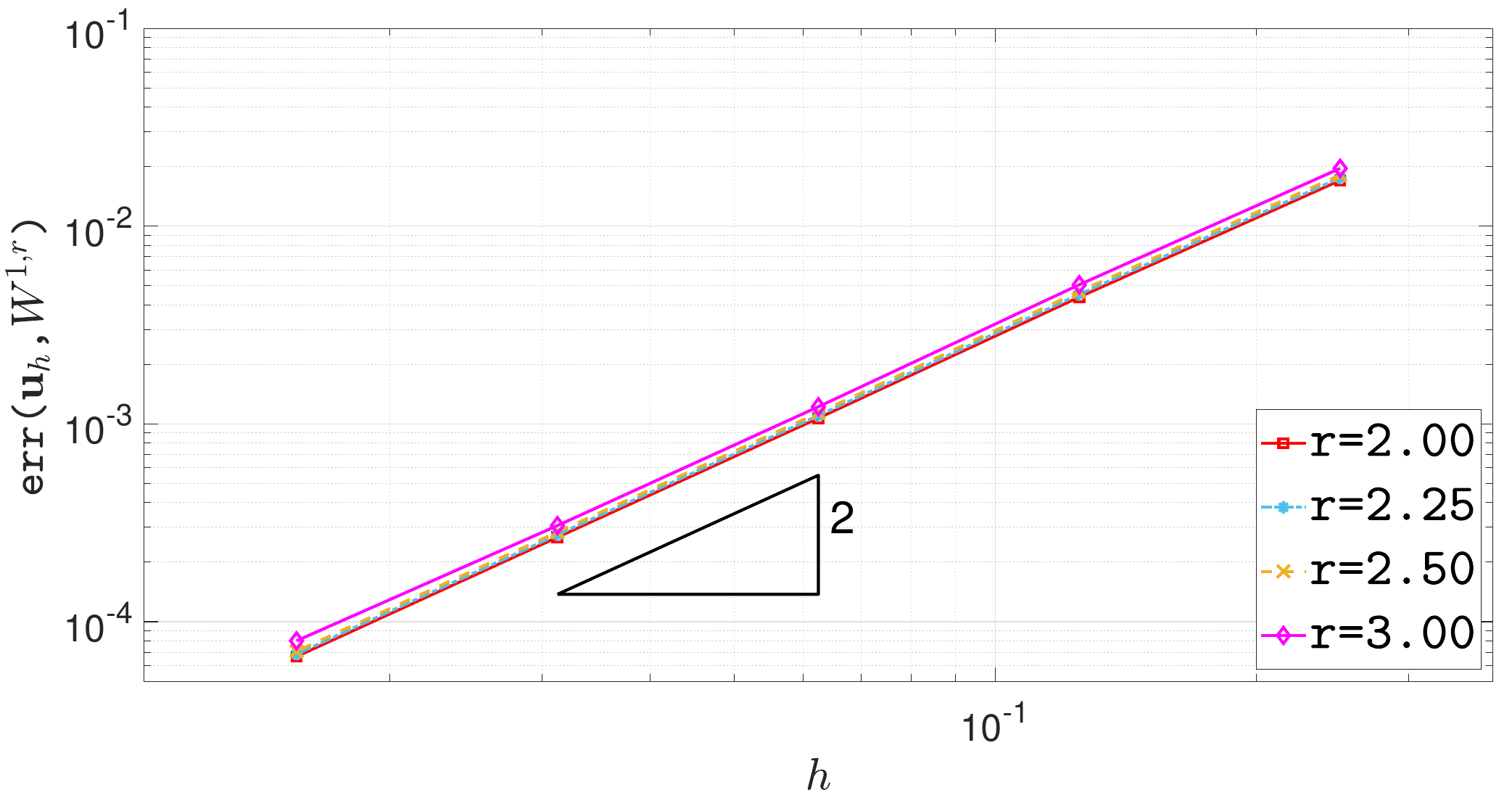}
    \end{subfigure}\qquad
    \begin{subfigure}{0.45\textwidth}
        \centering
        \includegraphics[width=\linewidth]{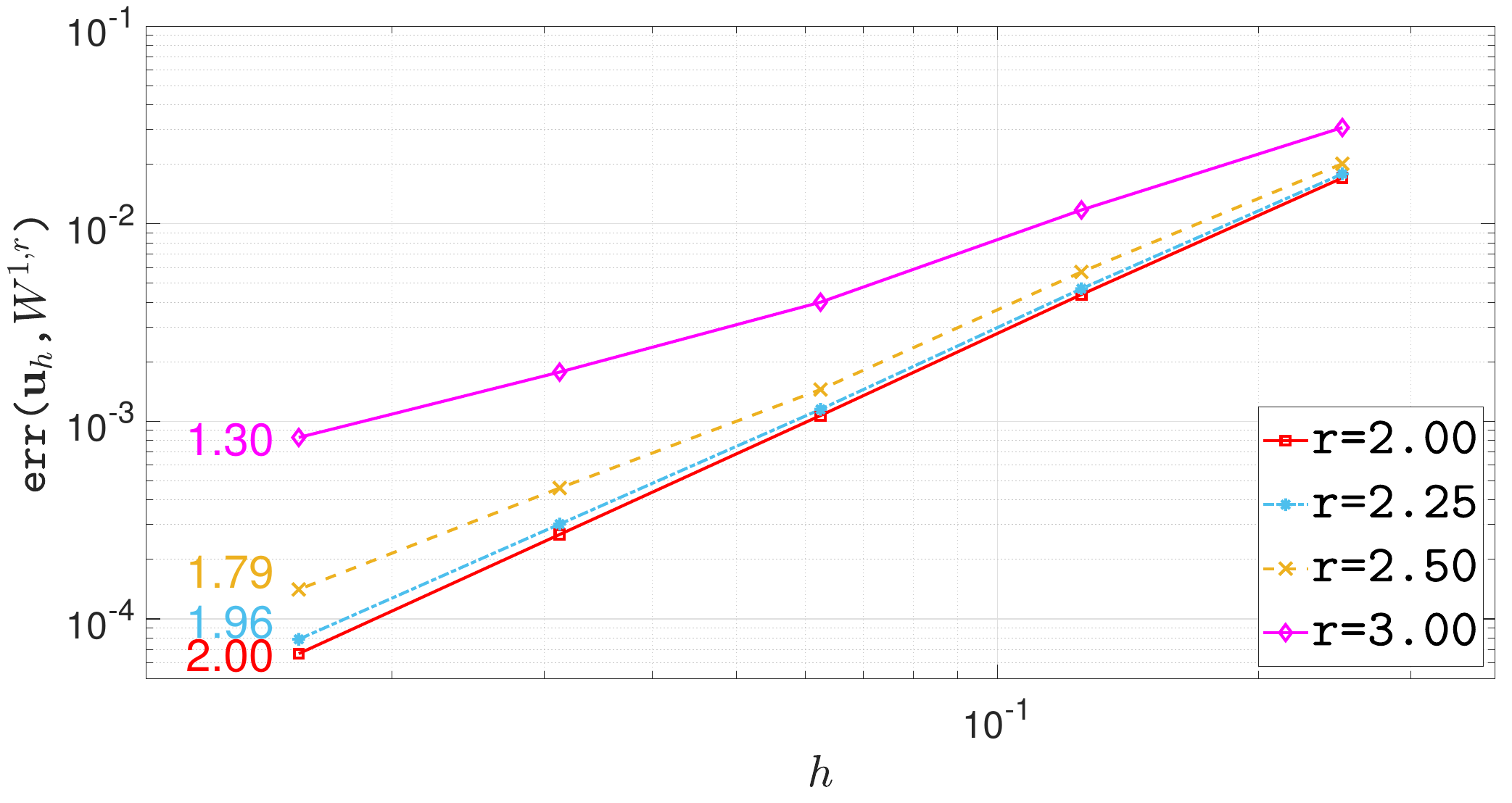}
    \end{subfigure}

    \vspace{0.5cm}

    % Riga 3
    \begin{subfigure}{0.45\textwidth}
        \centering
        \includegraphics[width=\linewidth]{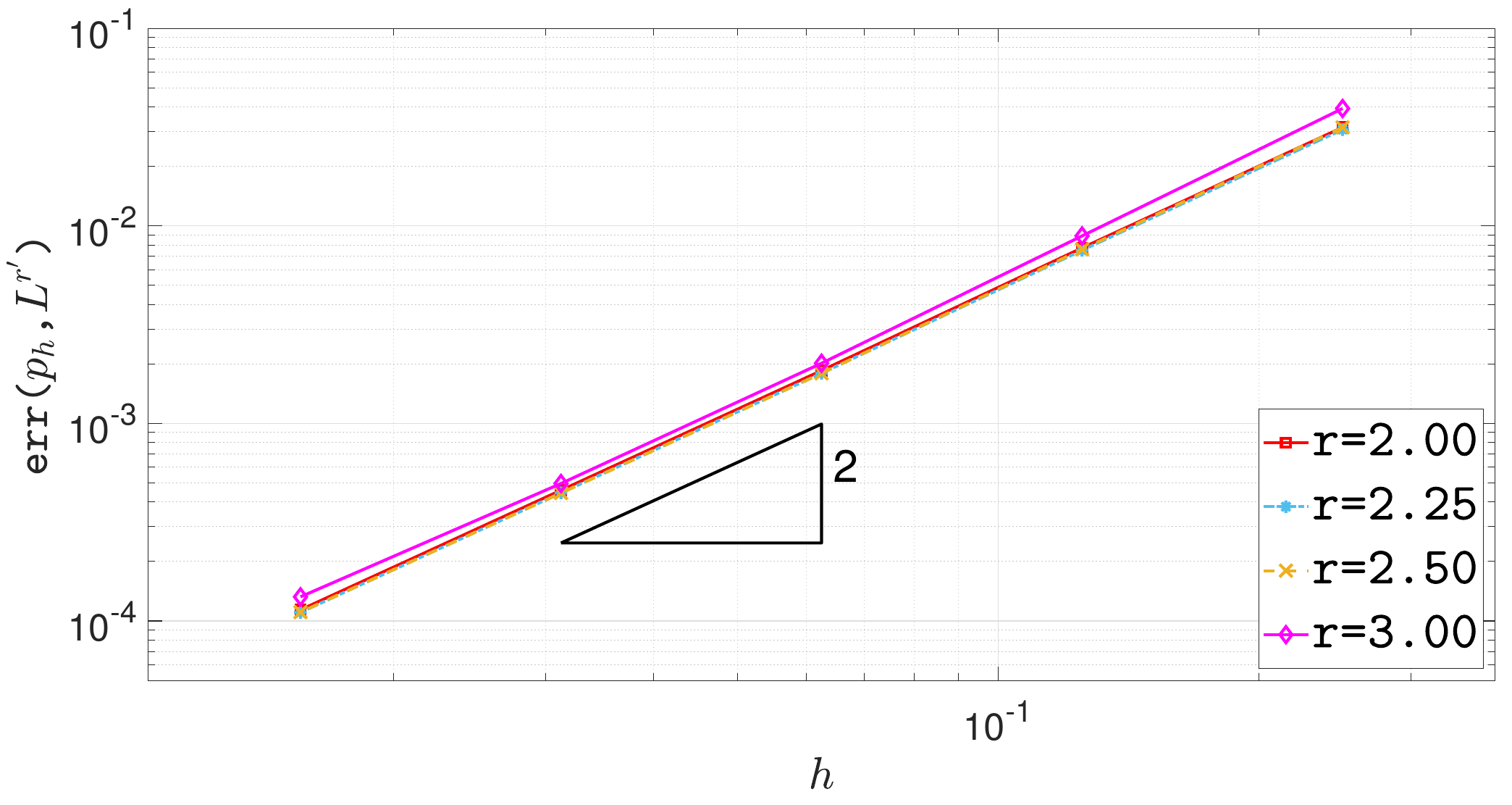}
    \end{subfigure}\qquad
    \begin{subfigure}{0.45\textwidth}
        \centering
        \includegraphics[width=\linewidth]{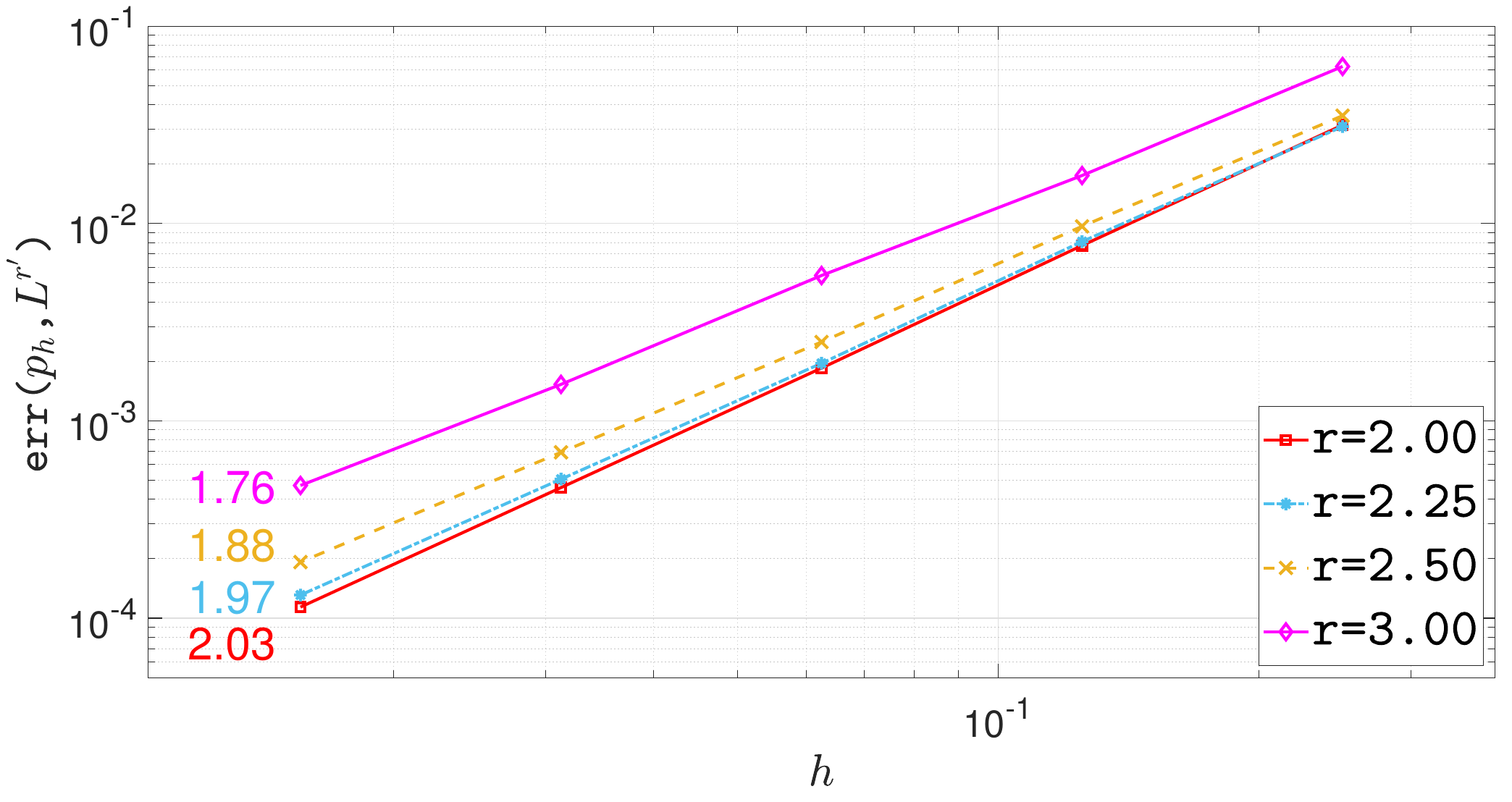}
    \end{subfigure}

    \vspace{0.5cm}
    
     % Riga 4
    \begin{subfigure}{0.45\textwidth}
        \centering
        \includegraphics[width=\linewidth]{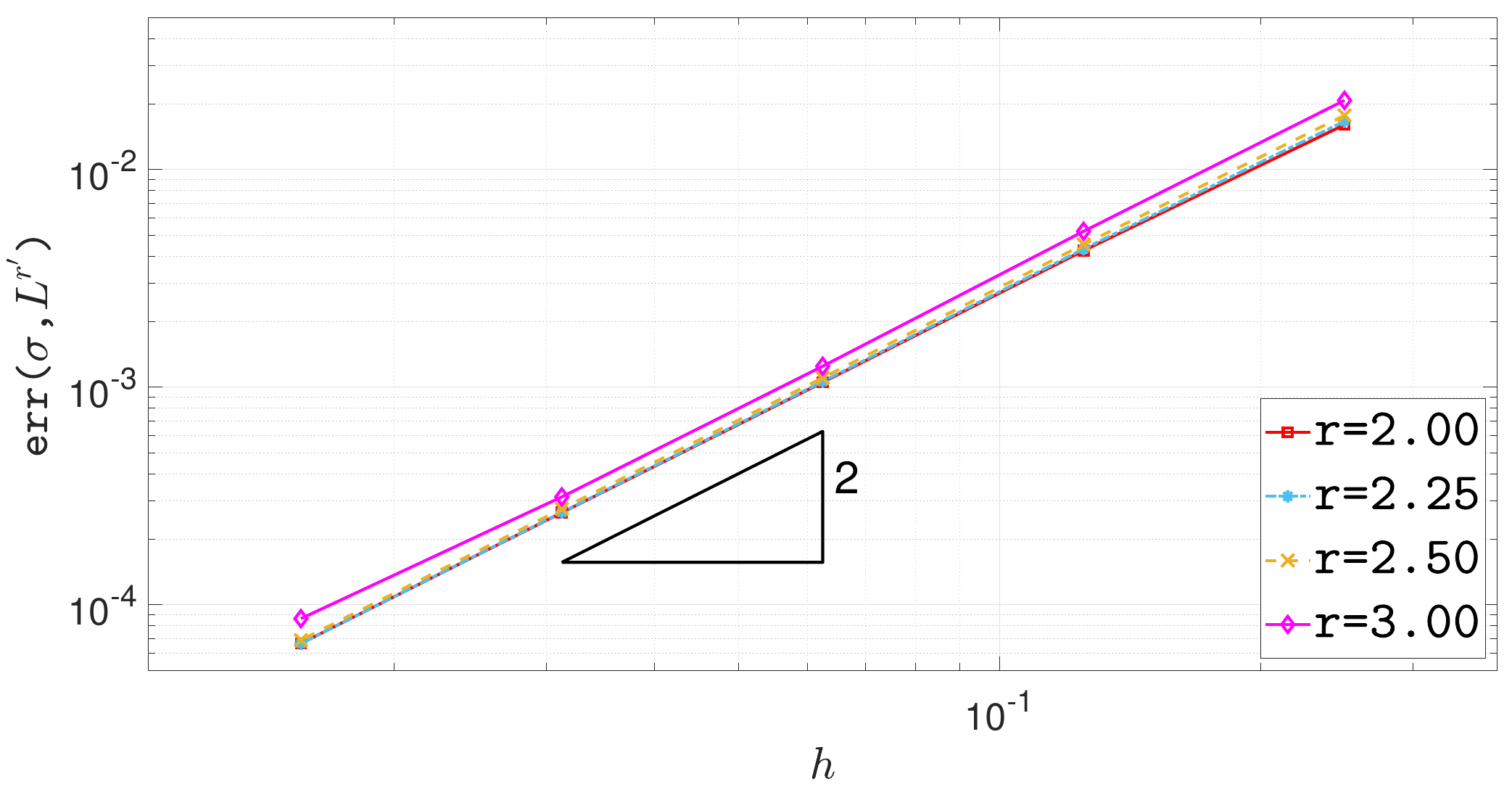}
    \end{subfigure}\qquad
    \begin{subfigure}{0.45\textwidth}
        \centering
        \includegraphics[width=\linewidth]{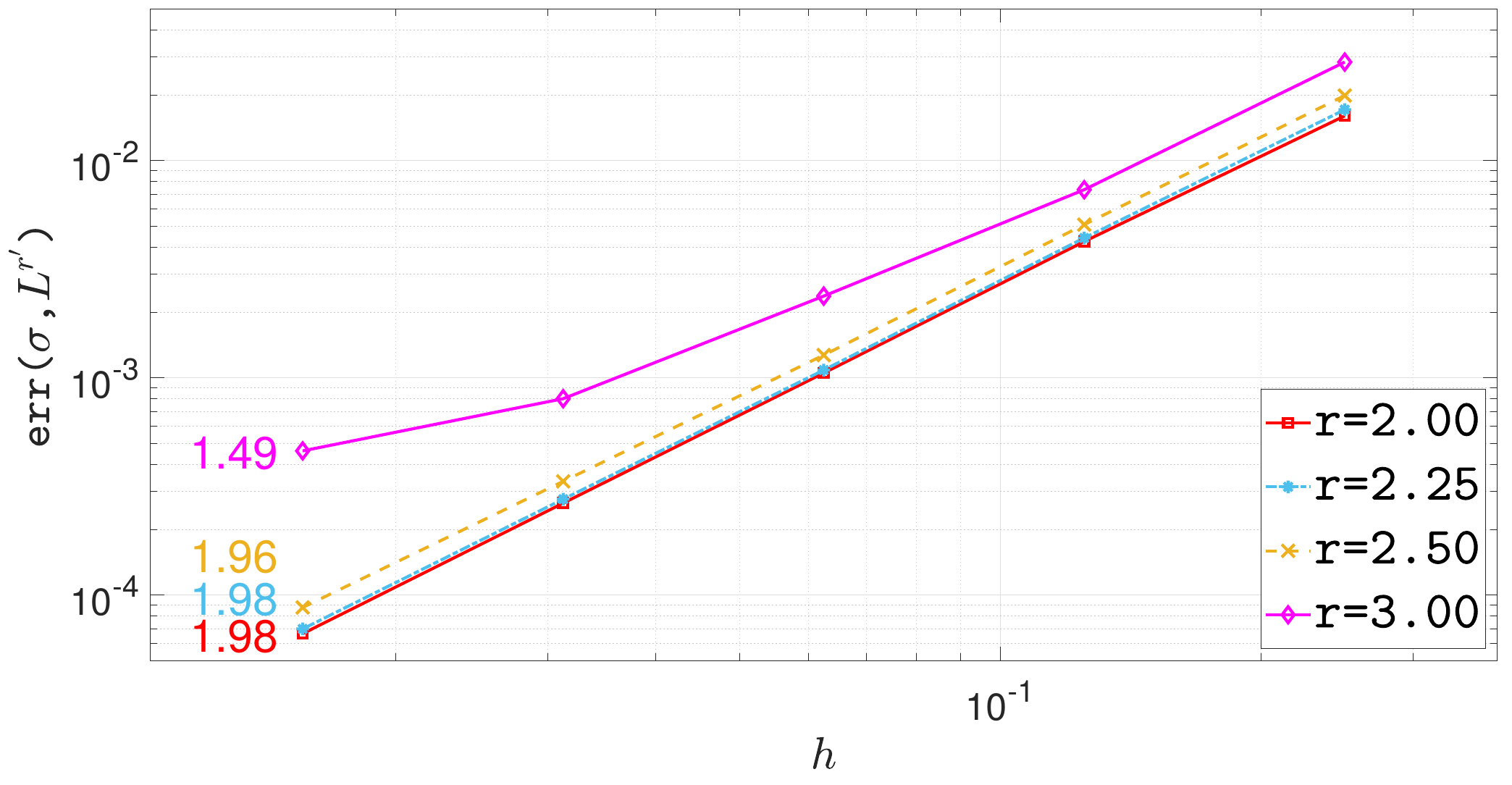}
    \end{subfigure}
   \caption{Test 1. Computed errors defined as in \eqref{eq:err_quant} as a function of the mesh size (loglog scale), for the mesh family \texttt{QUADRILATERAL}. Left panel: $\delta=1$, right panel: $\delta=0$.} 
    \label{fig:test1-Q}
\end{figure}
%%%%%%%%%%%%%%%%%%%%%%%%%%%%%%%%%%%%
%%%%%%%%%% TEST 1 - RANDOM %%%%%%%%%
\begin{figure}[htbp]
    \centering
    {\texttt{RANDOM MESHES}}
\\
    % Riga 1
    \begin{subfigure}{0.45\textwidth}
        \centering
        \includegraphics[width=\linewidth]{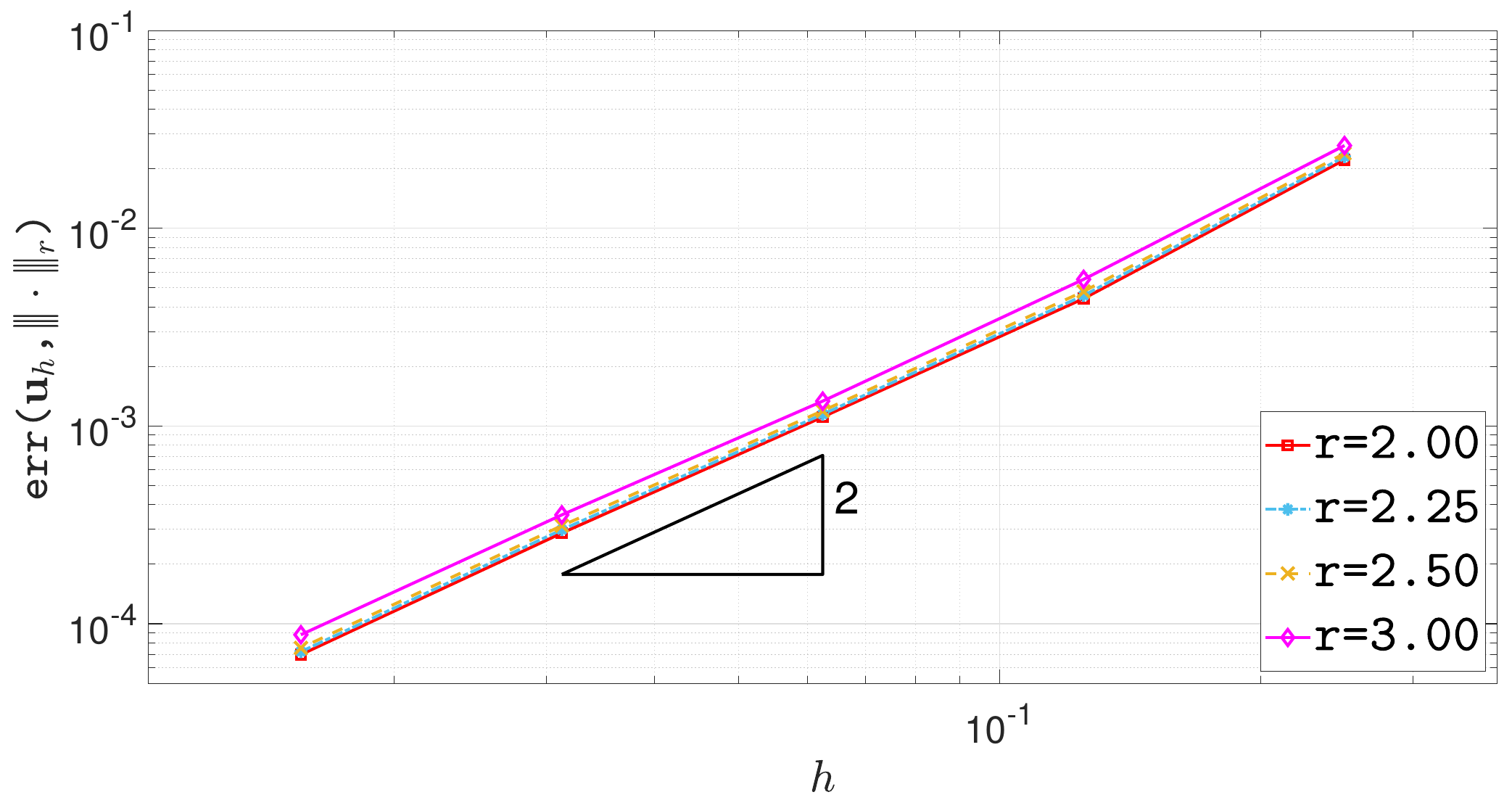}
        
    \end{subfigure}\qquad
    \begin{subfigure}{0.45\textwidth}
        \centering
        \includegraphics[width=\linewidth]{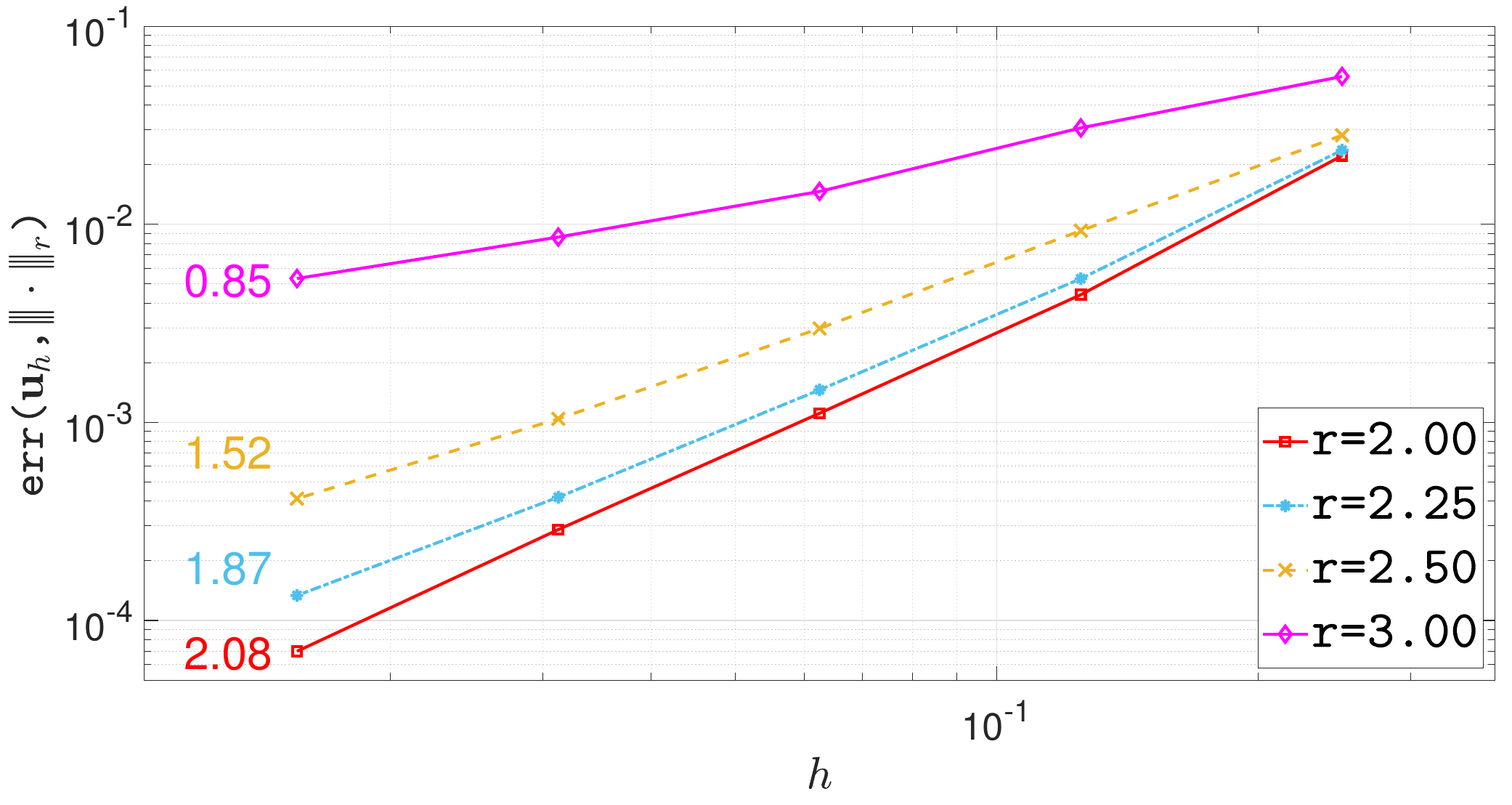}
    \end{subfigure}

    \vspace{0.5cm}

    % Riga 2
    \begin{subfigure}{0.45\textwidth}
        \centering
        \includegraphics[width=\linewidth]{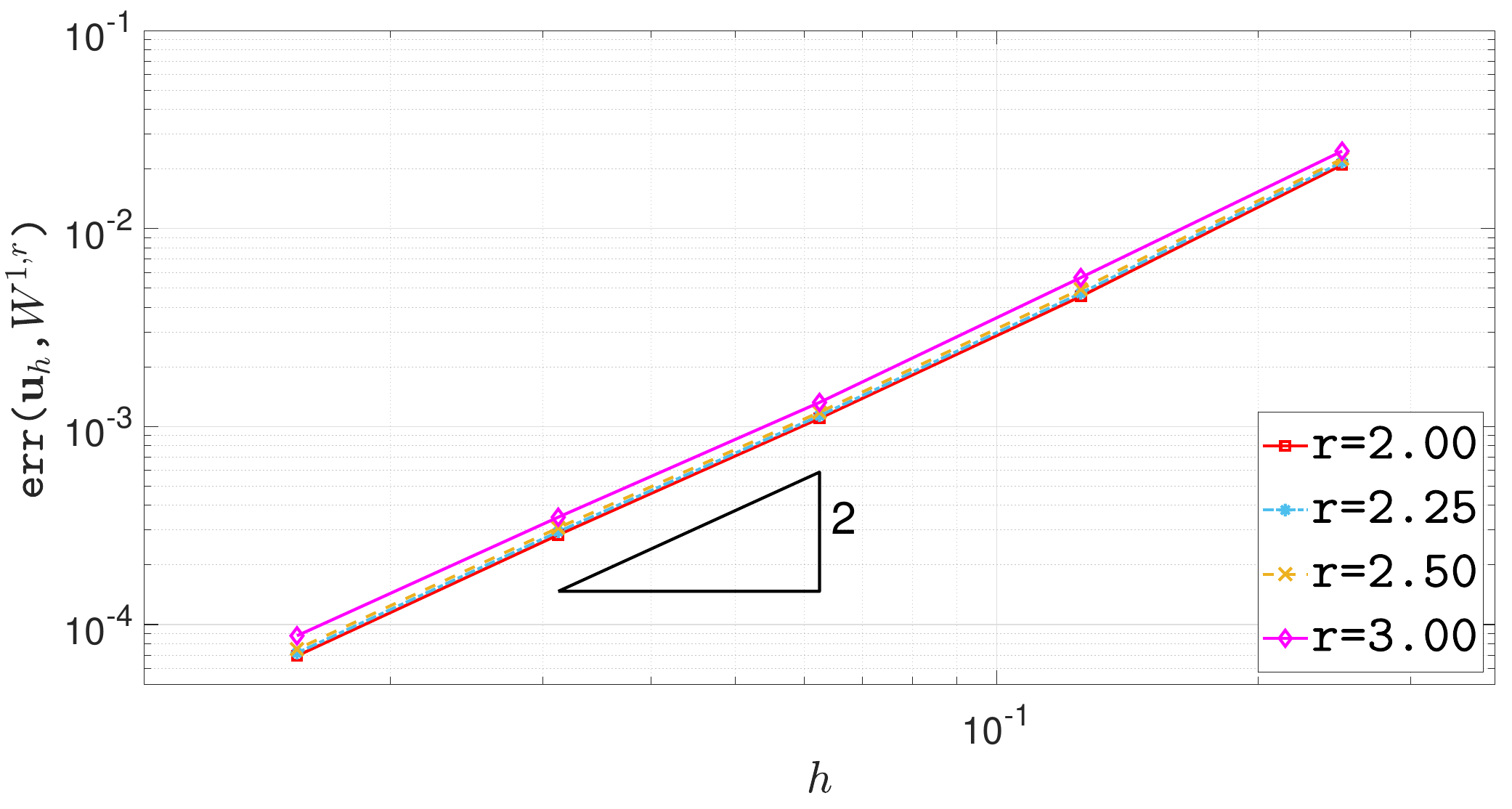}
    \end{subfigure}\qquad
    \begin{subfigure}{0.45\textwidth}
        \centering
        \includegraphics[width=\linewidth]{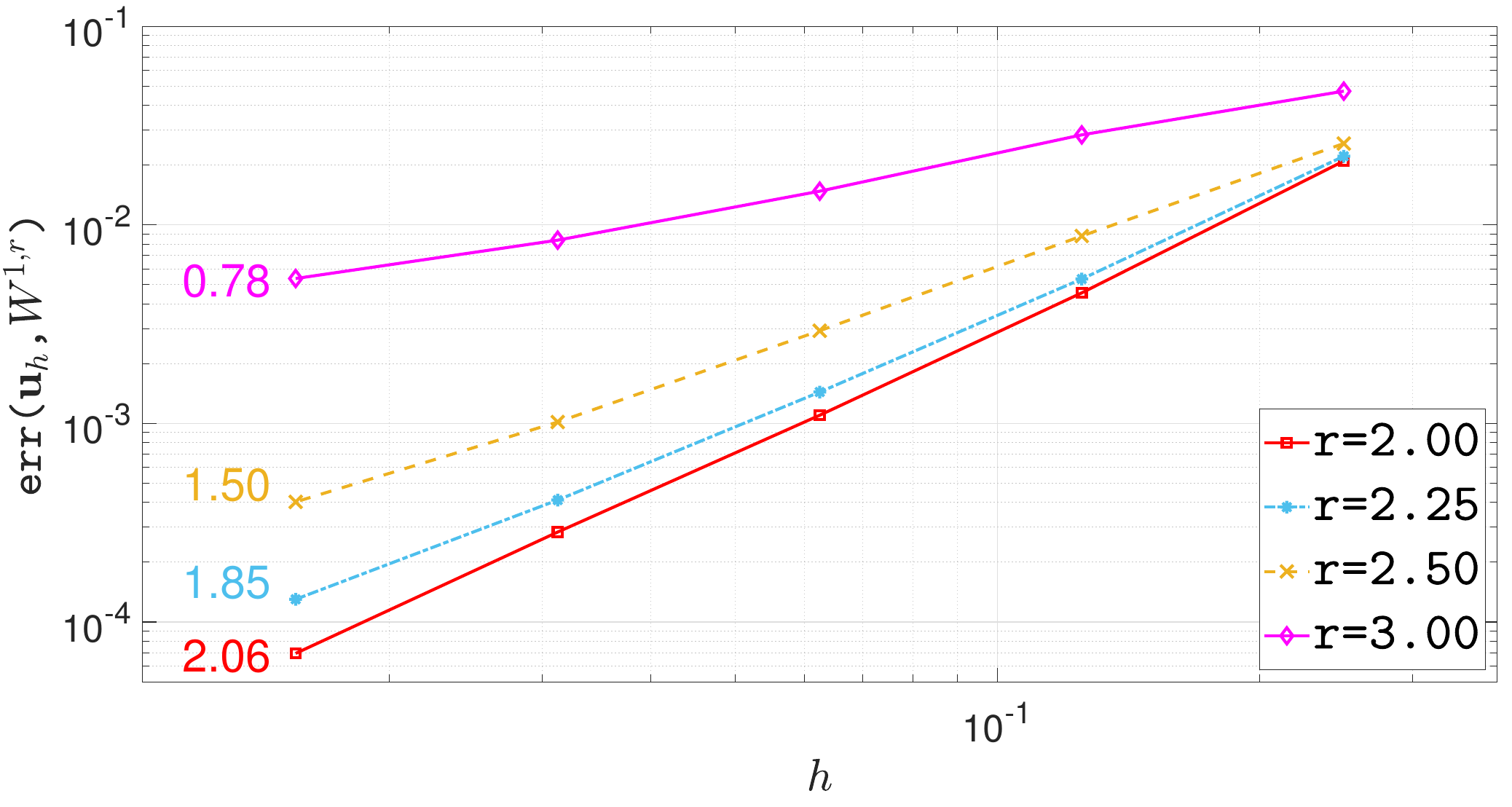}
    \end{subfigure}

    \vspace{0.5cm}

    % Riga 3
    \begin{subfigure}{0.45\textwidth}
        \centering
        \includegraphics[width=\linewidth]{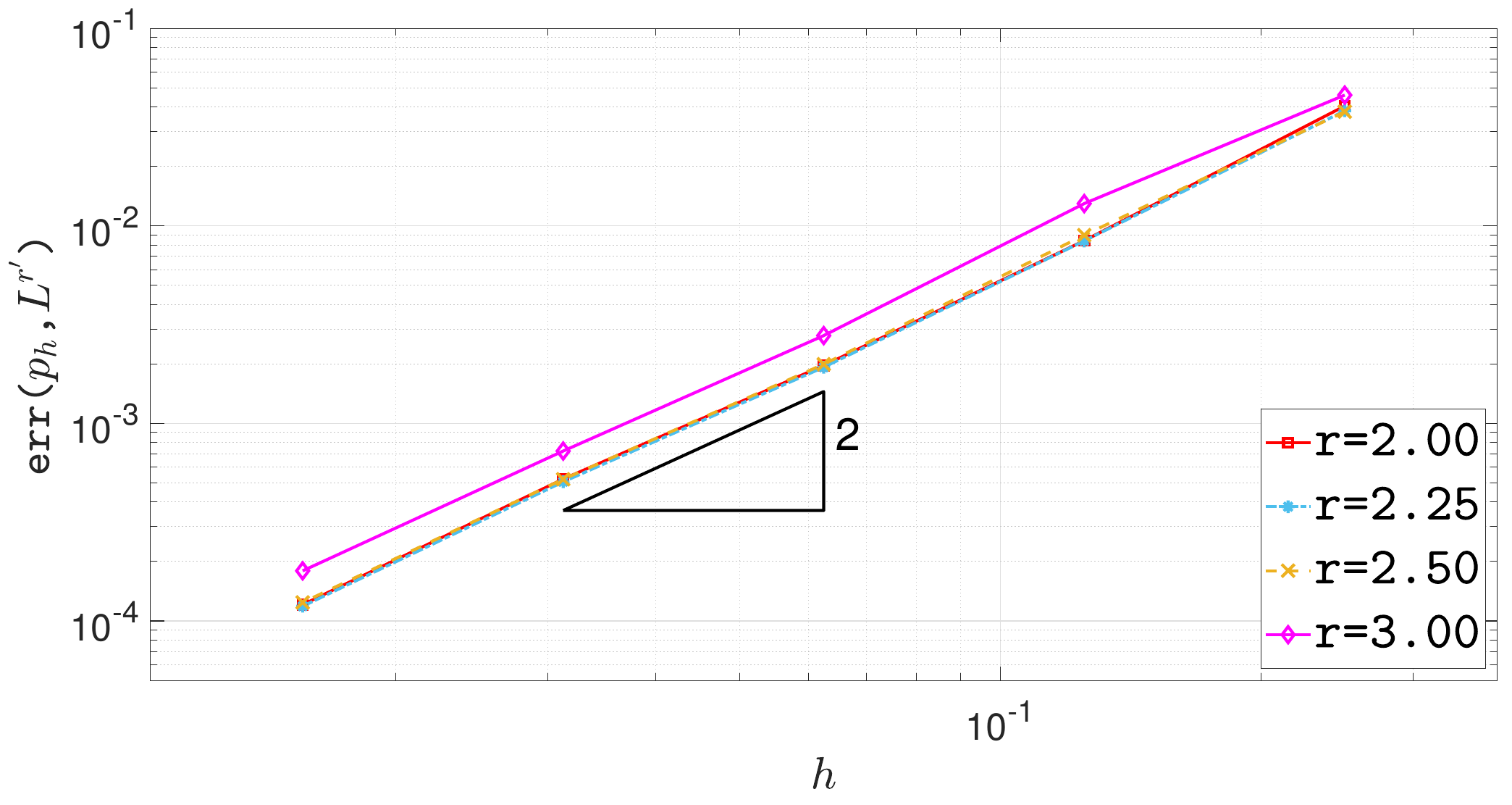}
    \end{subfigure}\qquad
    \begin{subfigure}{0.45\textwidth}
        \centering
        \includegraphics[width=\linewidth]{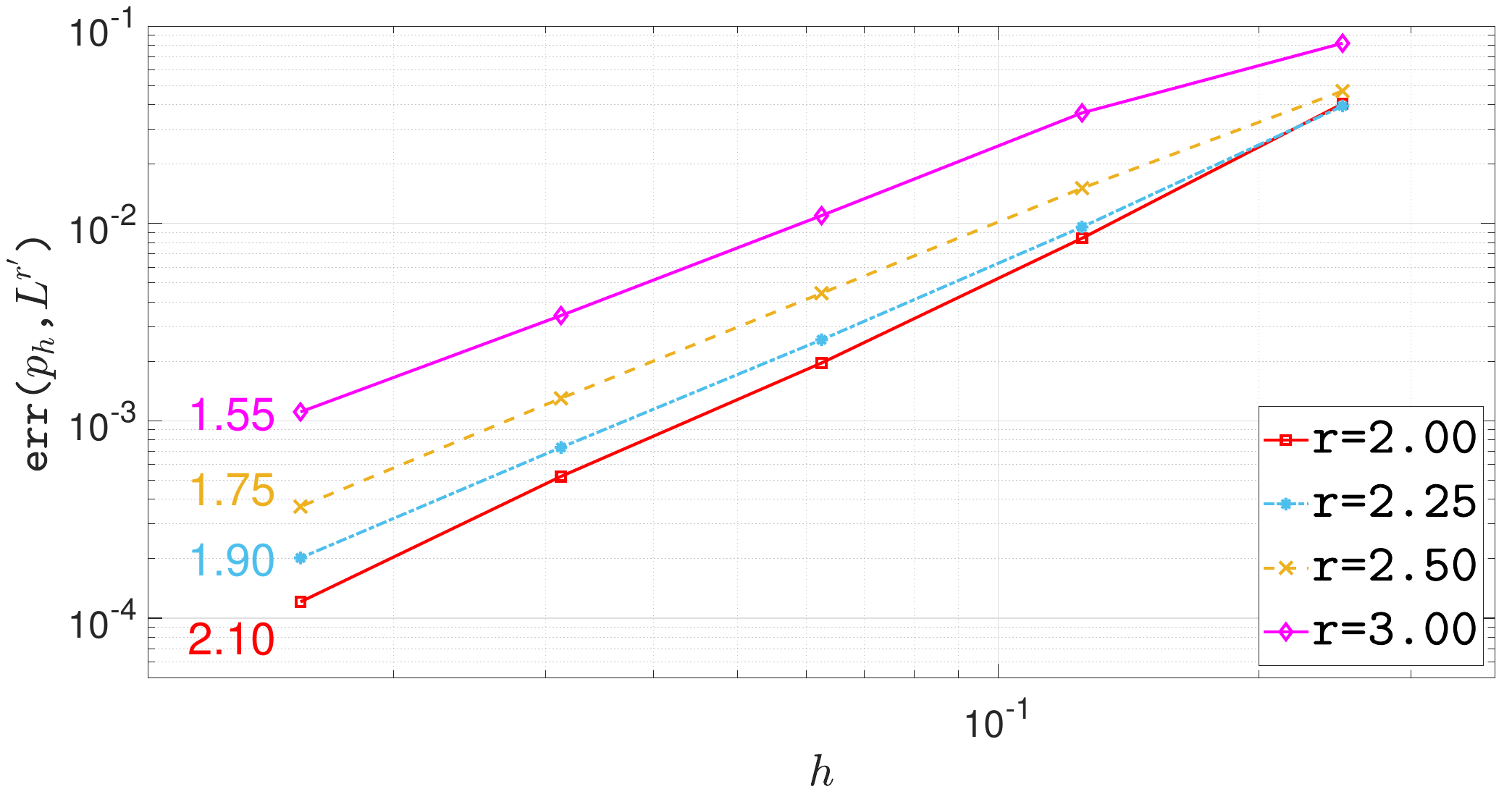}
    \end{subfigure}

    \vspace{0.5cm}
    
     % Riga 4
    \begin{subfigure}{0.45\textwidth}
        \centering
        \includegraphics[width=\linewidth]{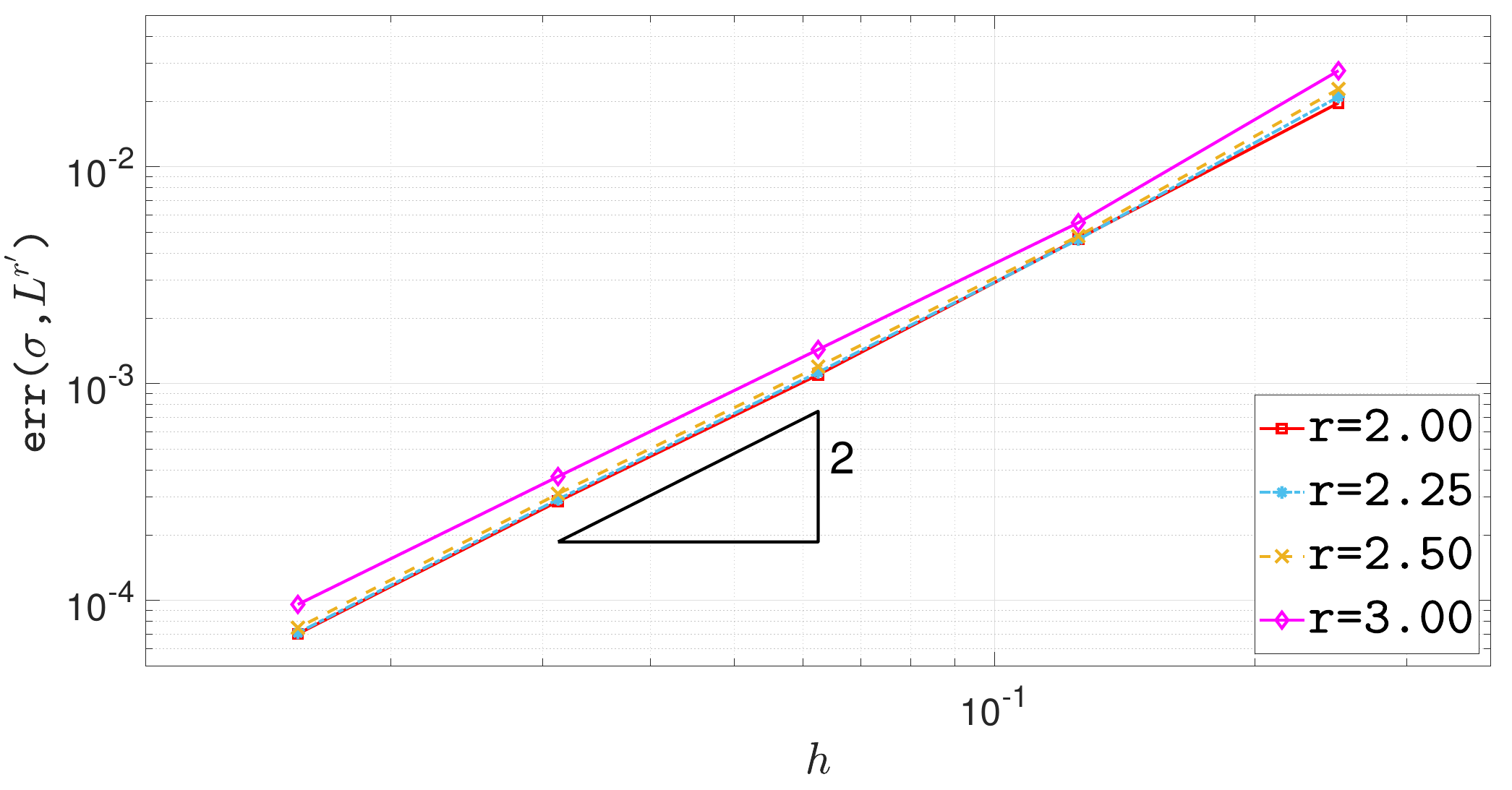}
    \end{subfigure}\qquad
    \begin{subfigure}{0.45\textwidth}
        \centering
        \includegraphics[width=\linewidth]{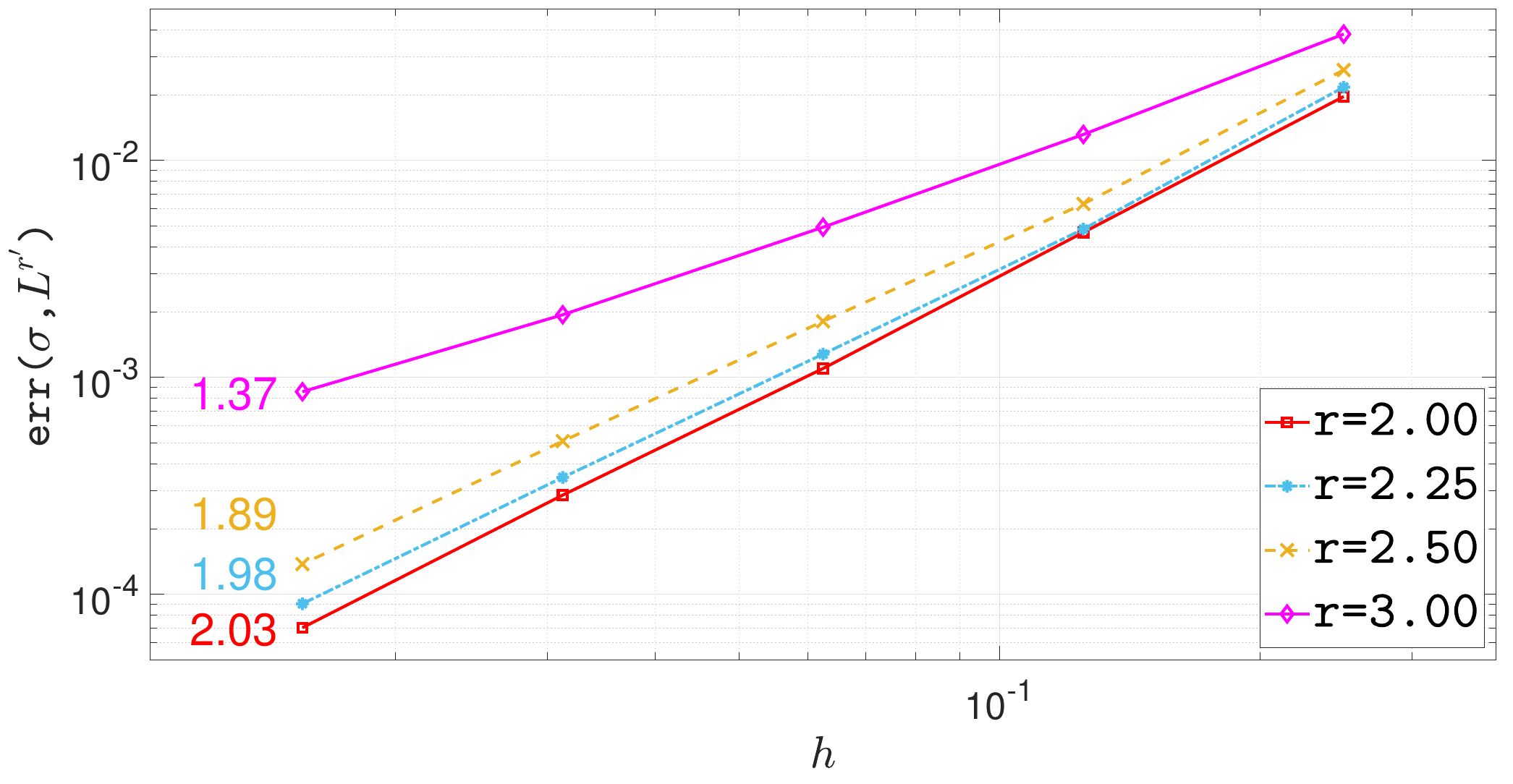}
    \end{subfigure}
   \caption{Test 1. Computed errors defined as in \eqref{eq:err_quant} as a function of the mesh size (log-log scale), for the mesh family \texttt{RANDOM}. Left panel: $\delta=1$, right panel: $\delta=0$.}
    \label{fig:test1-R}
\end{figure}
%%%%%%%%%%%%%%%%%%
In order to interpret the results illustrated in Fig.~\ref{fig:test1-Q} and Fig.~\ref{fig:test1-R} with respect to the theoretical estimates established in Section~\ref{sec:error_analysis}, Table~\ref{tab:test1-err} reports the expected convergence orders corresponding to the different sources of error derived in Theorems~\ref{theo:main} and~\ref{theo:main:2}, specifically the two terms appearing on the right in equations \eqref{eq:bound_u_uh} and \eqref{eq:bound_p_ph}, respectively (see also Remark~\ref{rem:orders}). 
In particular, we report the interpolation errors  
$\b u_{\rm ex} - \b u_I$ and $p_{\rm ex} - p_I$,
as well as the terms \(h^{k/(r-1)}\) and \(h^{2k/r}\) appearing in the bounds \eqref{eq:pressF} and \eqref{eq:proof_press_b}. 
To analyze the stress errors $\texttt{err}(\b \sigma, L^{r'})$, we also show the quantities $\Vert \b \sigma(\cdot, \b \epsilon (\b u_{\rm ex})) - \b \sigma(\cdot, \b \epsilon (\b u_I)) \Vert_{L^{r'}}$ (denoted by $\Vert \b \sigma_{\rm ex} - \b \sigma_I \Vert_{L^{r'}}$).
%%%%%%%%%% TEST 1 - TABLE
\begin{table}[!htbp]
\centering
\begin{small}
\begin{tabular}{l|cccccc}
\toprule
  % void
\texttt{r} & 
$\tri{\b u_{\rm ex} - \b u_I}\tri_{r}$ &
$\Vert \b u_{\rm ex} - \b u_I \Vert_{W^{1,r}}$ &
$\Vert  p_{\rm ex} - p_I \Vert_{L^{r'}}$ &
$\Vert  \b \sigma_{\rm ex} - \b \sigma_I \Vert_{L^{r'}}$ &
$2/(r-1)$ &
$4/r$
\\
\midrule
\texttt{2.00} &
\texttt{2.00} &
\texttt{2.00} &
\texttt{2.00} &
\texttt{2.00} &
\texttt{2.00} &
\texttt{2.00} 
\\
\texttt{2.25} &
\texttt{2.00} &
\texttt{2.00} &
\texttt{2.00} &
\texttt{2.00} &
\texttt{1.60} &
\texttt{1.77} 
\\
\texttt{2.50} &
\texttt{2.00} &
\texttt{2.00} &
\texttt{2.00} &
\texttt{2.00} &
\texttt{1.33} &
\texttt{1.60} 
\\
\texttt{3.00} &
\texttt{2.00} &
\texttt{2.00} &
\texttt{2.00} &
\texttt{2.00} &
\texttt{1.00} &
\texttt{1.33} 
\\
\bottomrule
\end{tabular}
\end{small}
\caption{Test 1. Expected orders of convergence for the terms appearing in the \emph{a priori} error estimates in Section~\ref{sec:error_analysis}.}
\label{tab:test1-err}
\end{table}
It can be observed that, for $\delta = 1$, the interpolation errors dominate all the error quantities defined in \eqref{eq:err_quant}. 
For $\delta = 0$ the results are less pronounced compared to the case $\delta = 1$.
Let us analyze the velocity errors in the discrete norm. 
For $r=\texttt{2.25}$ the averaged rates \texttt{1.80} and  \texttt{1.87} are close to the rate $4/r$. 
For $r=\texttt{2.50}$  velocity errors have rates \texttt{1.50} and  \texttt{1.52}, which fall between the rates $2/(r-1)$ and $4/r$. 
For $r=\texttt{3.00}$ we observe rates \texttt{1.13} and  \texttt{0.85},
with the expected rate $2/(r-1)$ nearly attained.
Similar rates are observed for the continuous norm.
The pressure errors exhibit in general better rates, lying between $4/r$ and $2$.

\FloatBarrier
%%%%%%%%%%%%%%%%%%%%%%%%%%%%%%%%%
%%%%%%%%%%%%%%%%%%%%%%%%%%%%%%%%%
\subsection{Test 2. Polynomial solution}

To further investigate how the different sources of error combine in the error estimates, 
we consider Problem~\eqref{eq:stokes.continuous} on $\Omega=(0,1)^2$ where the Dirichlet datum and the loading term  are chosen in accordance with the exact solution
\[
\b u_{\rm ex}(x_1,x_2) = 
\begin{bmatrix}
x_1^2 + x_2^2 + 3x_1 + 5
\\
-2x_1 x_2 - x_1^2 - 3x_2 + 7
\end{bmatrix} \,,
\qquad
p_{\rm ex}(x_1,x_2) =  0 \,.
\]
We notice that $\b u_{\rm ex} \in [\Pk_2(\Omega)]^2 \subseteq \VDG$, hence, by Theorem~\ref{theo:main}, we have $R_1 = 0$. As a consequence the asymptotically dominant contribution to the  error arising from the approximation of $\b \sigma(\cdot, \epsilon(\b u_{\rm ex}))$ should be better appreciated (with less or no influence by the other terms).
In Fig.~\ref{fig:test2-Q} we show the error quantities in \eqref{eq:err_quant} (for the pressures we plot the absolute errors) for the sequence of \texttt{QUADRILATERAL} meshes and parameters $r$ and $\delta$ in \eqref{eq:r-delta}.

\begin{figure}[htbp]
    \centering
    {\texttt{QUADRILATERAL MESHES}}
\\
    % Riga 1
    \begin{subfigure}{0.45\textwidth}
        \centering
        \includegraphics[width=\linewidth]{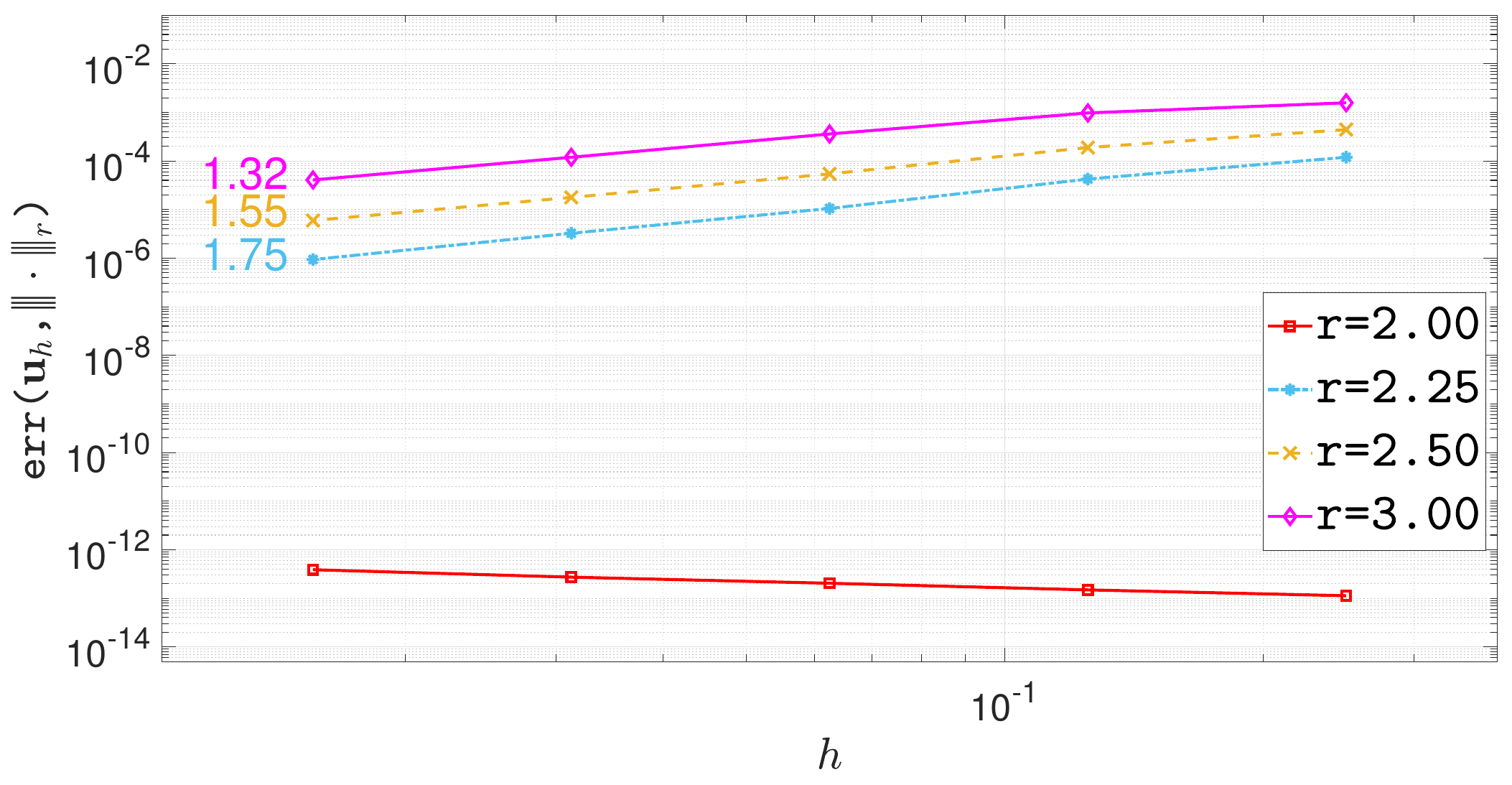}
        
    \end{subfigure}\qquad
    \begin{subfigure}{0.45\textwidth}
        \centering
        \includegraphics[width=\linewidth]{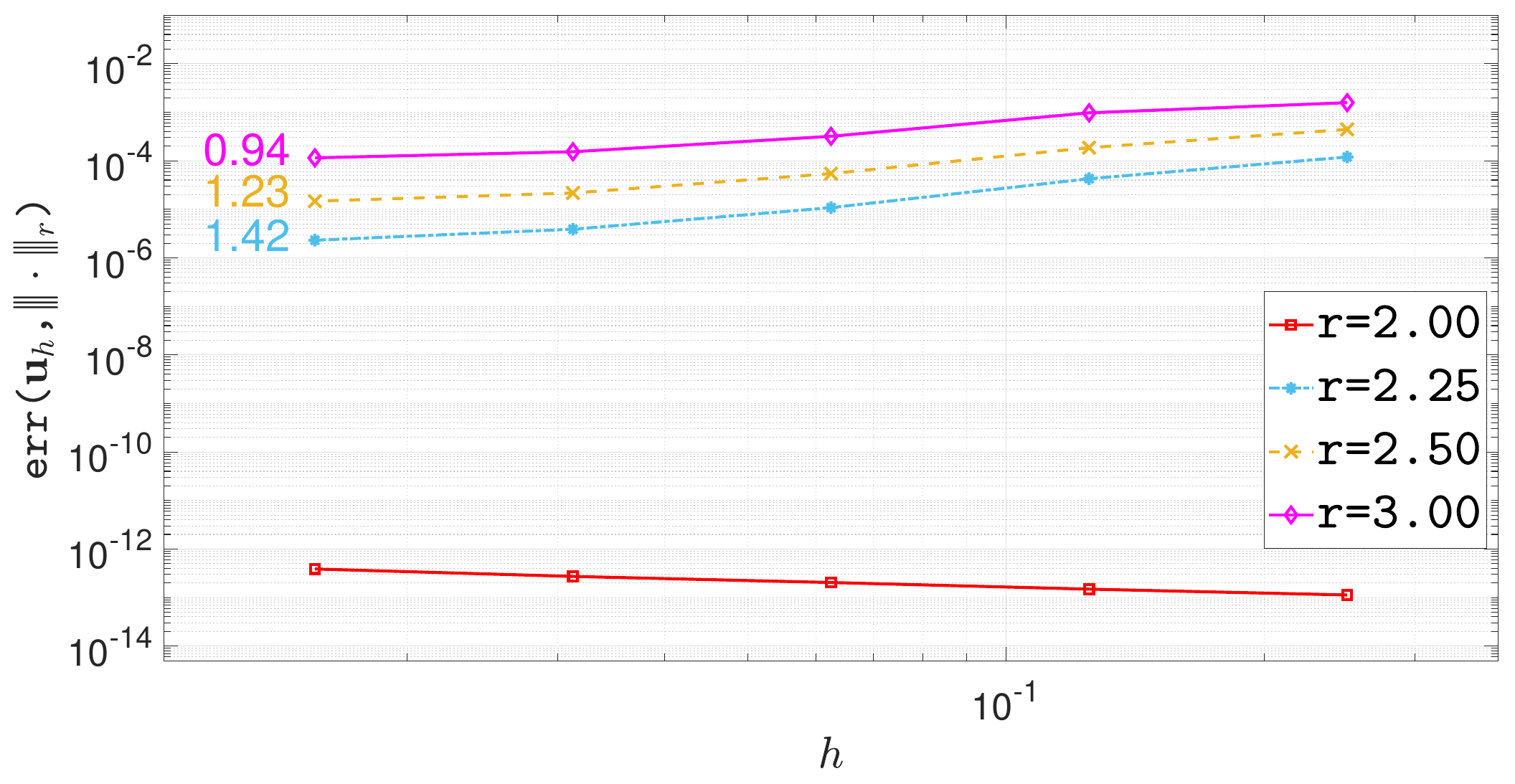}
    \end{subfigure}

    \vspace{0.5cm}

    % Riga 2
    \begin{subfigure}{0.45\textwidth}
        \centering
        \includegraphics[width=\linewidth]{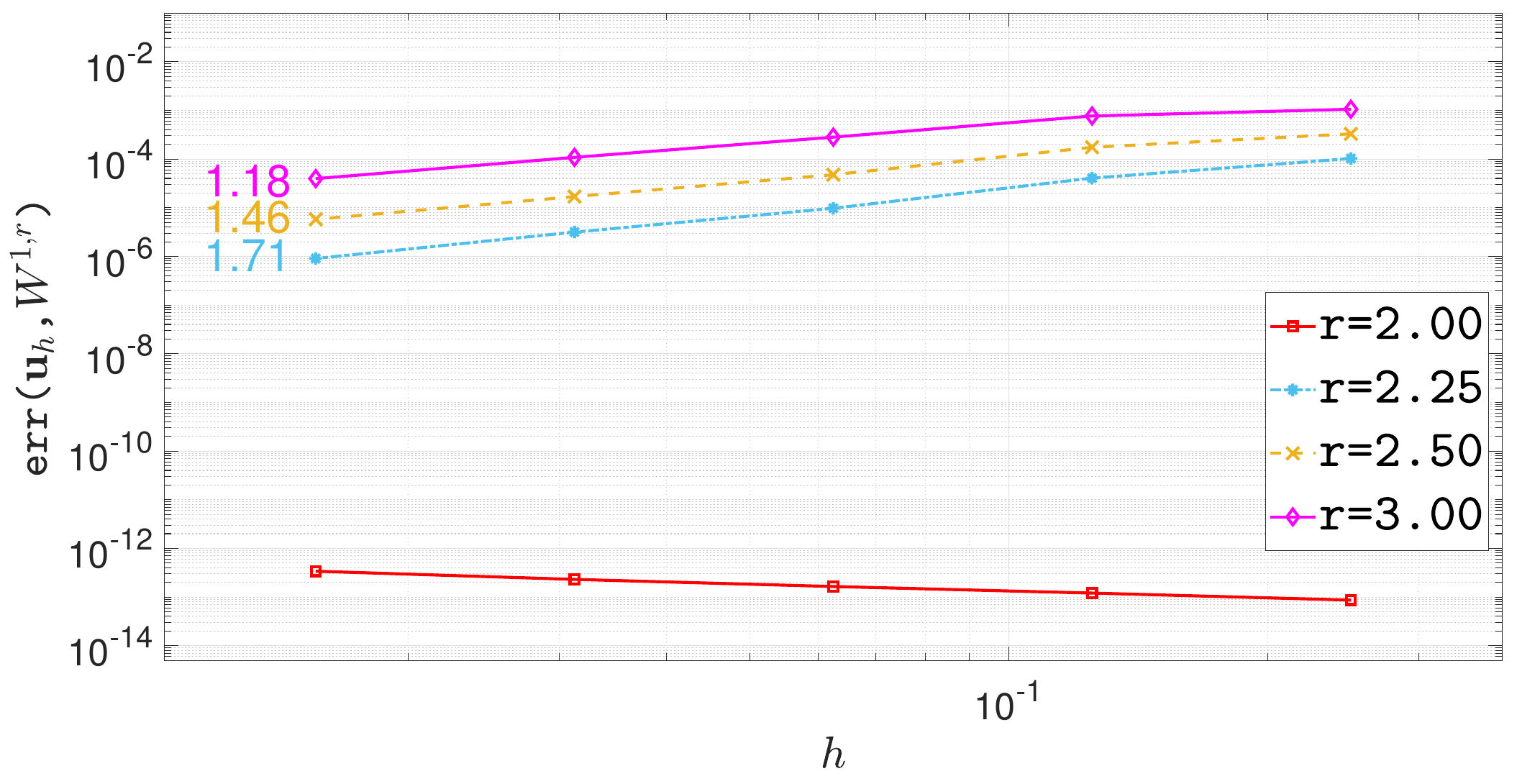}
    \end{subfigure}\qquad
    \begin{subfigure}{0.45\textwidth}
        \centering
        \includegraphics[width=\linewidth]{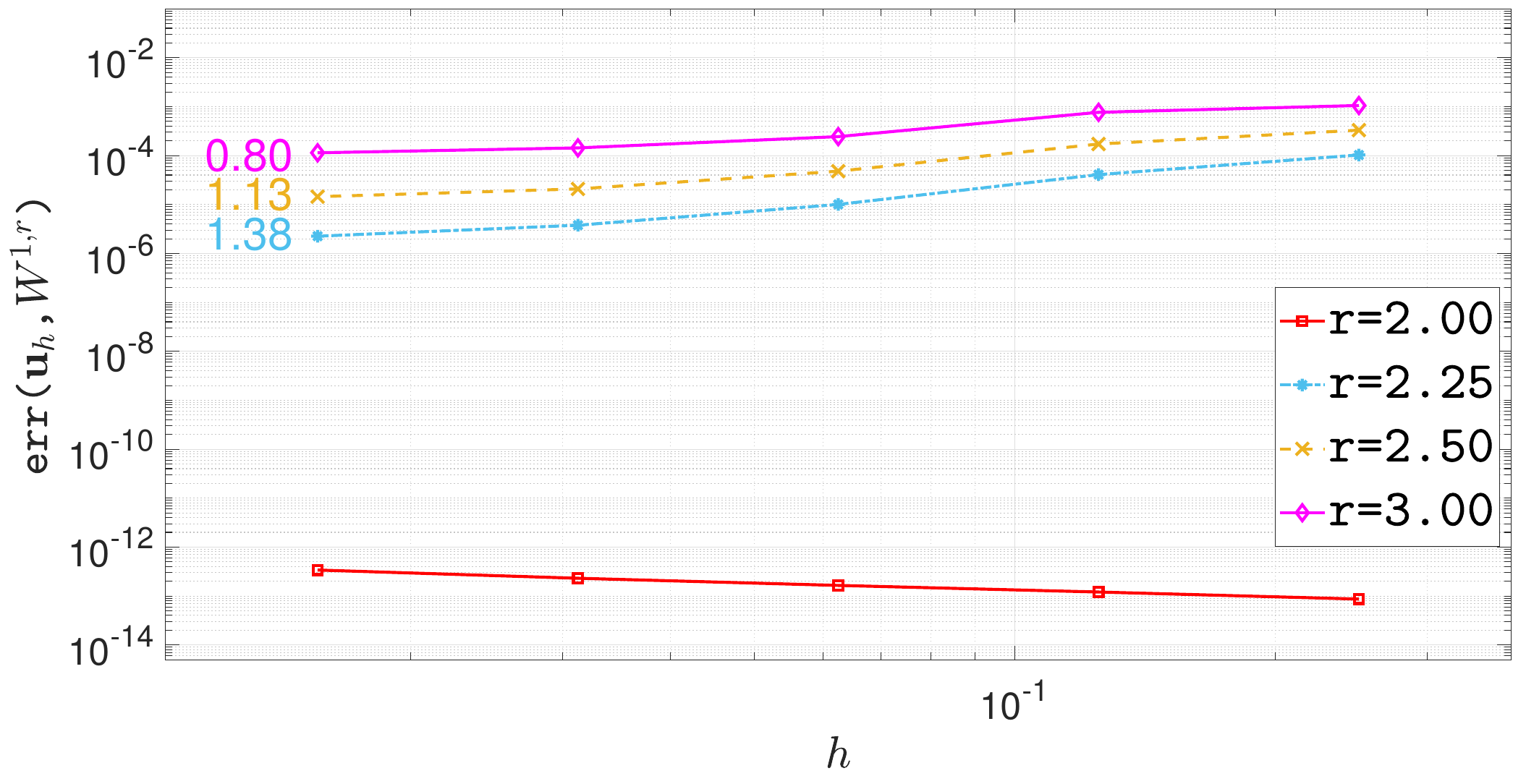}
    \end{subfigure}

    \vspace{0.5cm}

    % Riga 3
    \begin{subfigure}{0.45\textwidth}
        \centering
        \includegraphics[width=\linewidth]{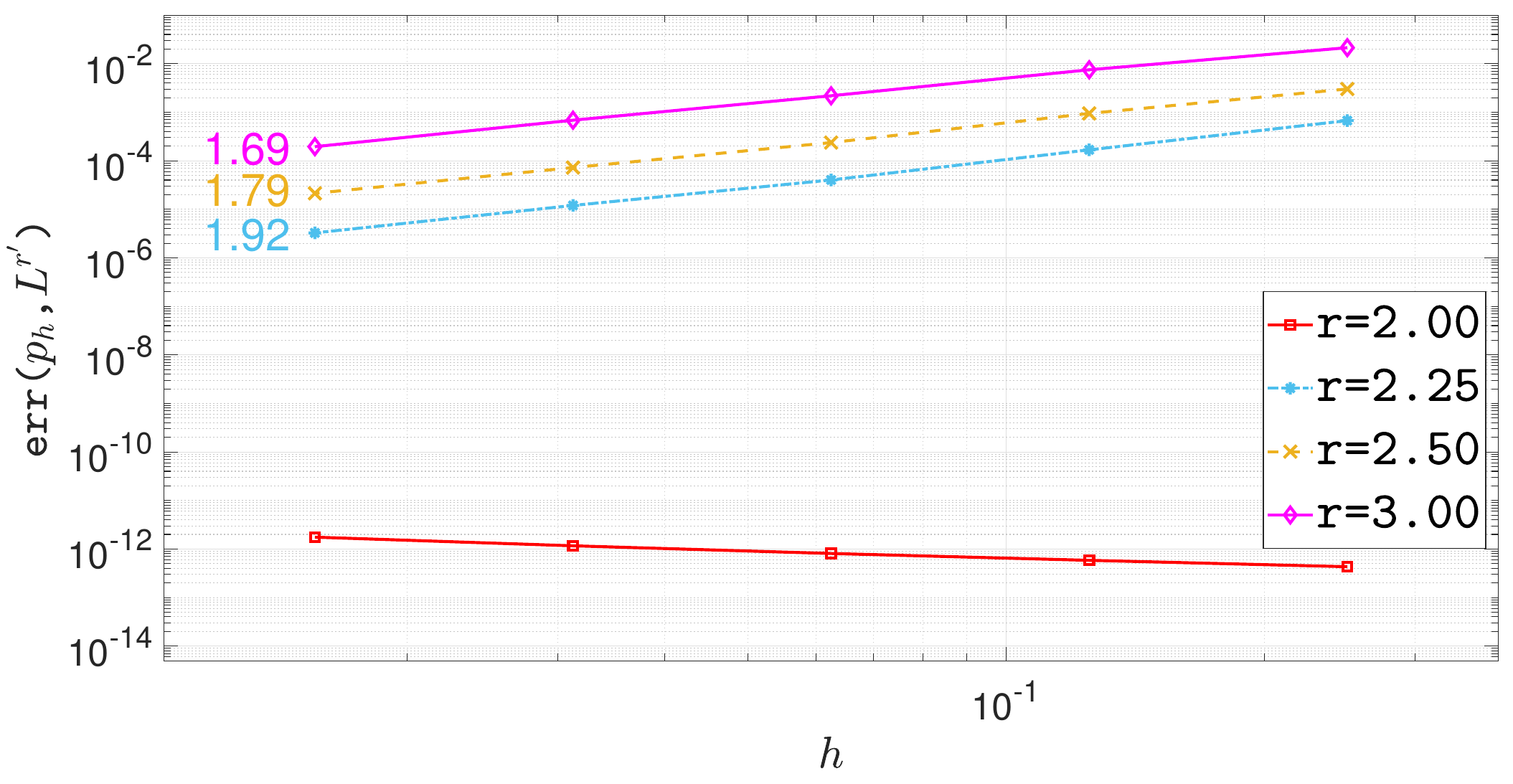}
    \end{subfigure}\qquad
    \begin{subfigure}{0.45\textwidth}
        \centering
        \includegraphics[width=\linewidth]{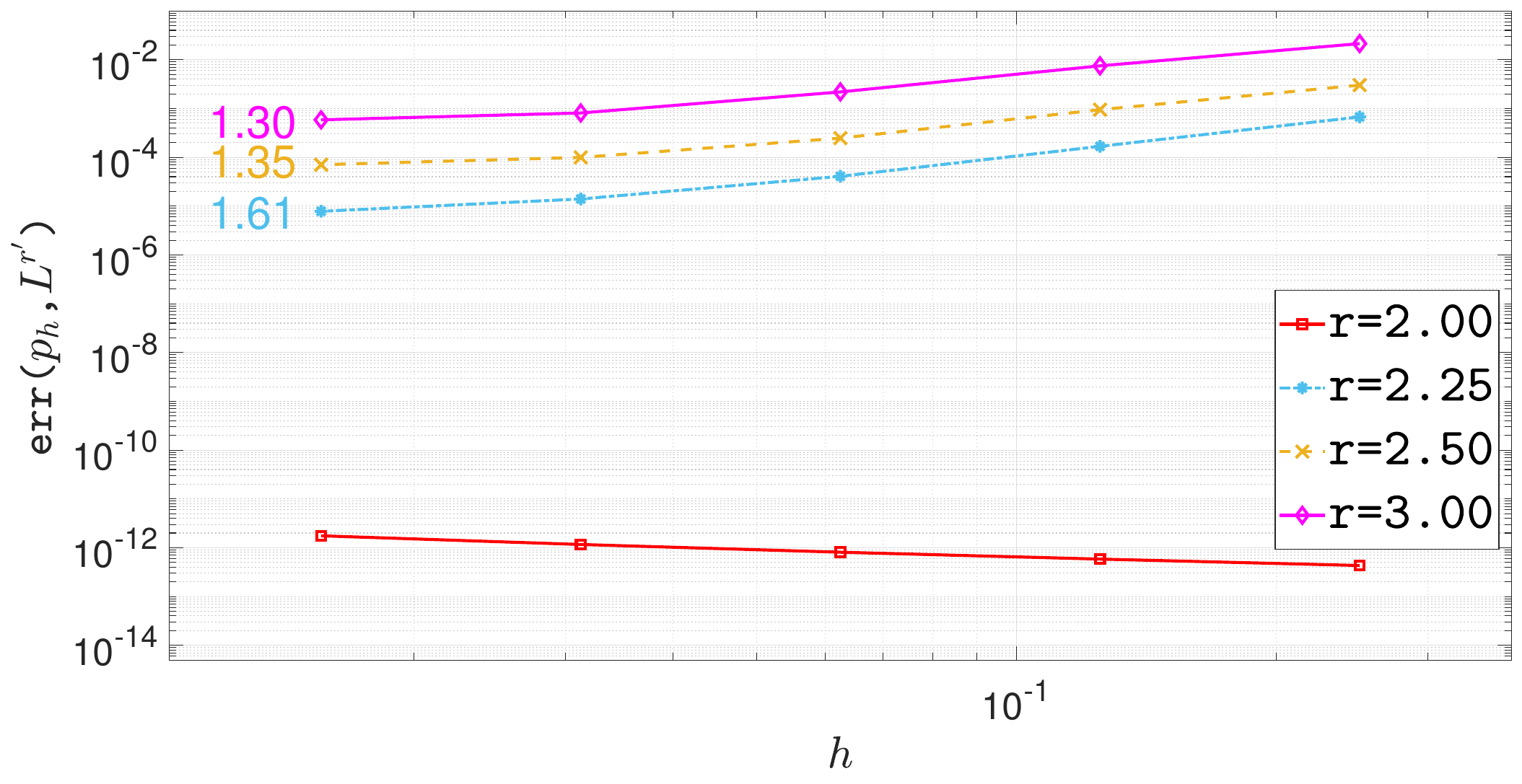}
    \end{subfigure}

    \vspace{0.5cm}
    
     % Riga 4
    \begin{subfigure}{0.45\textwidth}
        \centering
        \includegraphics[width=\linewidth]{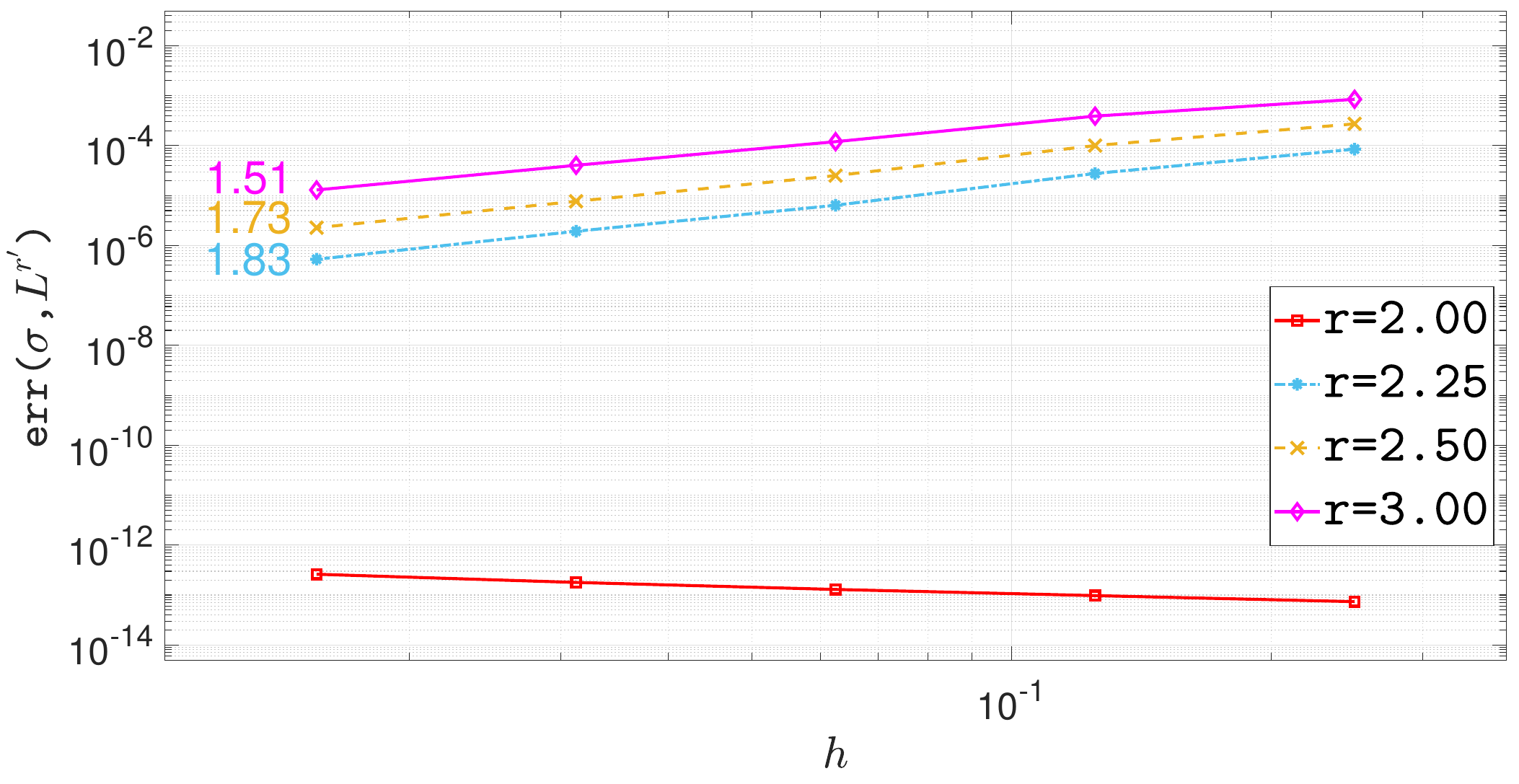}
    \end{subfigure}\qquad
    \begin{subfigure}{0.45\textwidth}
        \centering
        \includegraphics[width=\linewidth]{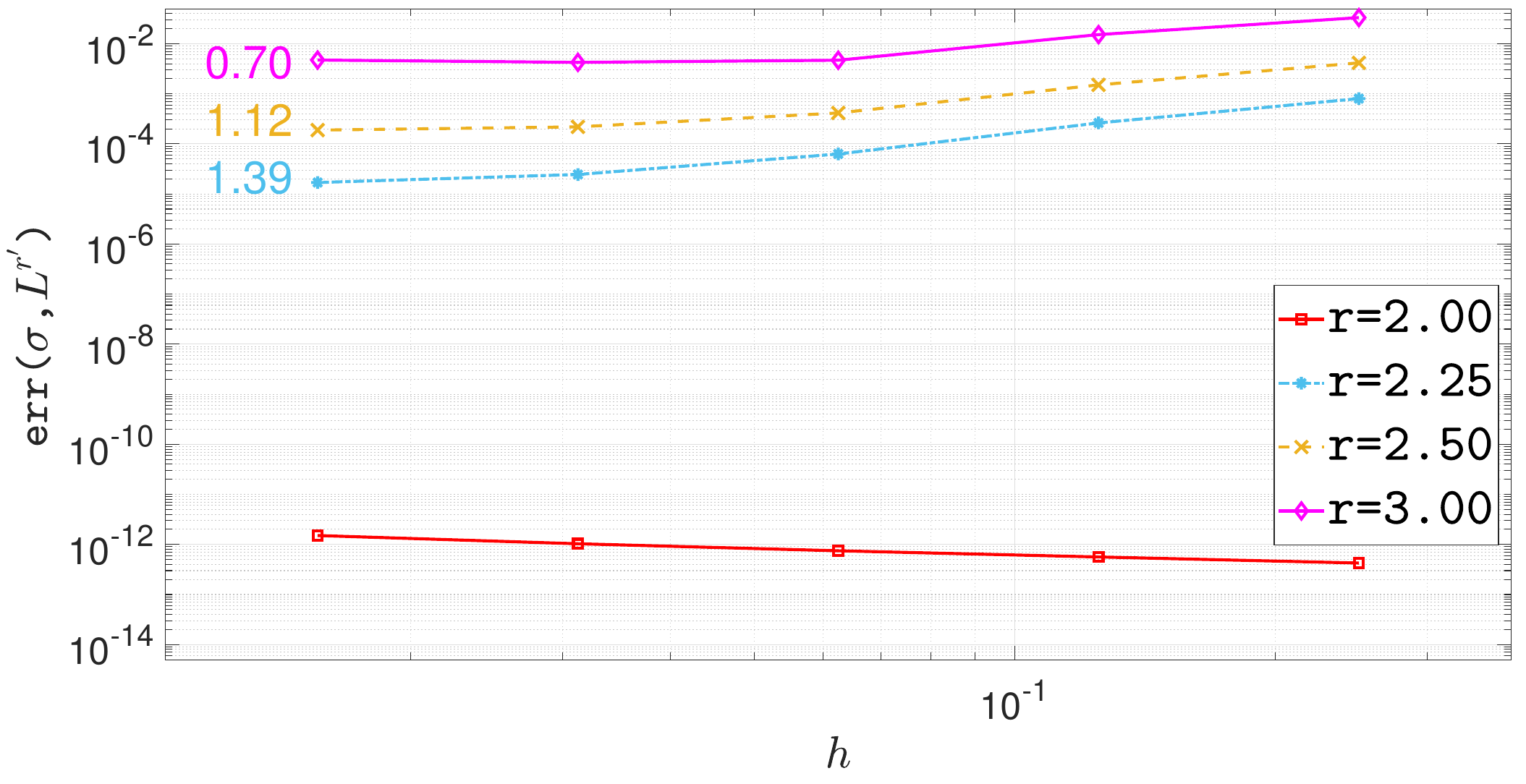}
    \end{subfigure}
   \caption{Test 2. Computed errors defined as in \eqref{eq:err_quant} as a function of the mesh size (loglog scale), for the mesh family \texttt{QUADRILATERAL}. Left panel: $\delta=1$, right panel: $\delta=0$.} 
    \label{fig:test2-Q}
\end{figure}
As expected, for $r=\texttt{2}$, we recover the so-called ``patch test'', i.e. the discrete solution and the exact solution coincide up to machine precision. 
For $r>2$ the average experimental order of convergence \texttt{AEOC} of the $\texttt{err}(\b u_h, \tri{\b \cdot}\tri_{r})$ is in good agreement with rates predicted by Corollary \ref{cor:main} for $\delta=1$ and Theorem \ref{theo:main} for $\delta=0$, namely $4/r$ and $2/(r-1)$ respectively (cf. Table \ref{tab:test1-err}). 
The pressure errors exhibit in both cases better rates.
\FloatBarrier
\FloatBarrier
%%%%%%%%%%%%%%%%%%%%%%%%%%%%%%%%%
%%%%%%%%%%%%%%%%%%%%%%%%%%%%%%%%%
\subsection{Test 3. Singular solution}
The purpose of this test is to assess the performance of the method in the presence of solutions with low Sobolev regularity. 
%
%%%%%%%%%%%%% NEW TABLES delta=1
\begin{table}[!htbp]
\centering
\begin{small}
$\texttt{err}(\b u_h, \tri{\b \cdot}\tri_{r})$\\
\begin{tabular}{p{2cm}|cccc}
\toprule
  % void
& \multicolumn{4}{c}{\texttt{r}}
\\
%\midrule
   \texttt{1/h}
&  $\texttt{2.00}$ 
&  $\texttt{2.25}$      
&  $\texttt{2.50}$ 
&  $\texttt{3.00}$       
\\
\midrule          
  \texttt{4}
& \texttt{7.577256e-04}
& \texttt{1.461758e-03}
& \texttt{3.500512e-03}
& \texttt{1.119863e-02}
\\
  \texttt{8}
& \texttt{3.772397e-04}
& \texttt{8.214921e-04}
& \texttt{2.262381e-03}
& \texttt{8.787189e-03}
\\
  \texttt{16}
& \texttt{1.874426e-04}
& \texttt{4.634299e-04}
& \texttt{1.468054e-03}
& \texttt{6.907361e-03}
\\
  \texttt{32}
& \texttt{9.308978e-05}
& \texttt{2.622244e-04}
& \texttt{9.550462e-04}
& \texttt{5.434720e-03}
\\
\midrule     
{\texttt{AEOC}}
& \texttt{1.008326e+00}
& \texttt{8.262771e-01}
& \texttt{6.246411e-01}
& \texttt{3.476817e-01}
\\
\midrule     
\texttt{$\frac{4}{r^2}$}
& \texttt{1.00}
& \texttt{0.79}
& \texttt{0.64}
& \texttt{0.44}
\\
\bottomrule
\end{tabular}
\\[0.5cm]
\end{small}
$\texttt{err}(\b u_h, W^{1,r})$\\
\begin{small}
\begin{tabular}{p{2cm}|cccc}
\toprule
  % void
& \multicolumn{4}{c}{\texttt{r}}
\\
%\midrule
   \texttt{1/h}
&  $\texttt{2.00}$ 
&  $\texttt{2.25}$      
&  $\texttt{2.50}$ 
&  $\texttt{3.00}$       
\\
\midrule          
  \texttt{4}
& \texttt{7.576044e-04}
& \texttt{1.457011e-03}
& \texttt{3.487738e-03}
& \texttt{1.117410e-02}
\\
  \texttt{8}
& \texttt{3.772245e-04}
& \texttt{8.207718e-04}
& \texttt{2.260225e-03}
& \texttt{8.782263e-03}    
\\
  \texttt{16}
& \texttt{1.874407e-04}
& \texttt{4.633219e-04}
& \texttt{1.467695e-03}
& \texttt{6.906383e-03}
\\
  \texttt{32}
& \texttt{9.308954e-05}
& \texttt{2.622083e-04}
& \texttt{9.549867e-04}
& \texttt{5.434526e-03}
\\
\midrule     
{\texttt{AEOC}}
& \texttt{1.008251e+00}
& \texttt{8.247422e-01}
& \texttt{6.229130e-01}
& \texttt{3.466444e-01}
\\
\midrule     
\texttt{$\frac{4}{r^2}$}
& \texttt{1.00}
& \texttt{0.79}
& \texttt{0.64}
& \texttt{0.44}
\\
\bottomrule
\end{tabular}
\end{small}
\\[0.5cm]
$\texttt{err}(p_h, L^{r'})$\\
\begin{small}
\begin{tabular}{p{2cm}|cccc}
\toprule
  % void
& \multicolumn{4}{c}{\texttt{r}}
\\
%\midrule
   \texttt{1/h}
&  $\texttt{2.00}$ 
&  $\texttt{2.25}$      
&  $\texttt{2.50}$ 
&  $\texttt{3.00}$       
\\
\midrule          
  \texttt{4}
& \texttt{1.173547e-01}
& \texttt{1.170544e-01}
& \texttt{1.238788e-01}
& \texttt{1.354350e-01}
\\
  \texttt{8}
& \texttt{5.832495e-02}
& \texttt{5.822885e-02}
& \texttt{6.167576e-02}
& \texttt{6.752532e-02}
\\
  \texttt{16}
& \texttt{2.896772e-02}
& \texttt{2.892814e-02}
& \texttt{3.065185e-02}
& \texttt{3.358190e-02}
\\
  \texttt{32}
& \texttt{1.438465e-02}
& \texttt{1.436589e-02}
& \texttt{1.522427e-02}
& \texttt{1.668436e-02}
\\
\midrule     
{\texttt{AEOC}}
& \texttt{1.009424e+00}
& \texttt{1.008820e+00}
& \texttt{1.008161e+00}
& \texttt{1.007010e+00}
\\
\midrule     
\texttt{$\Vert p_{\rm ex} - p_I \Vert_{L^{r'}}$}
& \texttt{1.00}
& \texttt{1.00}
& \texttt{1.00}
& \texttt{1.00}
\\
\bottomrule
\end{tabular}
\end{small}
\caption{Test 3. Computed errors $\texttt{err}(\b u_h, \tri{\b \cdot}\tri_{r})$ (top), $\texttt{err}(\b u_h, W^{1,r})$ (middle), and $\texttt{err}(p_h, L^{r'})$ (bottom) as in \eqref{eq:err_quant} for the mesh family \texttt{CARTESIAN}: $\delta=1$.}
\label{tab:test3-d1}
\end{table}
%%%%%%%%%%%%%%%%%%%%%%%%%%%%%%%%%%%%%%%
To this end, we examine the behavior of the proposed method for the benchmark
test introduced in \cite[Section~7]{Belenki.Berselli.ea:12}.
We consider Problem~\eqref{eq:stokes.continuous} on the square domain $\Omega = (-1,1)^2$, where
the forcing term $\b f$ (depending on $r$ and $\delta$ in \eqref{eq:Carreau}) and the Dirichlet boundary conditions prescribed on $\partial\Omega$
are chosen in accordance with the exact solution
\[
\b u_{\rm ex}(x_1,x_2) = 
\vert \b x \vert^{0.01}
\begin{bmatrix}
x_2
\\
-x_1
\end{bmatrix} \,,
\qquad
p_{\rm ex}(x_1,x_2) =  -\vert \b x \vert^{\gamma} + c_\gamma \,,
\]
where $\gamma=\frac{2}{r} - 1 + 0.01$ and $c_\gamma$ is s.t. $p_{\rm ex}$ is zero averaged.
We note that for all $r \in [2, \infty)$
\[
\b u_{\rm ex} \in \b W^{2/r+1, r}(\Omega) \,,
\quad
\b{\sigma}(\cdot, \b\epsilon(\b u_{\rm ex})) \in \mathbb{W}^{2/r',r'}(\Omega) \,,
\quad
\b f \in \b {W}^{2/r' -1,r'}(\Omega)\,,
\quad
p_{\rm ex} \in W^{1, r'}(\Omega) \,,
\]
therefore, with the notation of Theorem~\ref{theo:main} and Theorem~\ref{theo:main:2}: 
\[
k_1= \frac{2}{r} \,,
\quad
k_2= \frac{2}{r'} \,,
\quad
k_3= \frac{2}{r'} -2 \,,
\quad
k_4= 1 \,.
\]
The domain $\Omega$ is partitioned with a sequence of \texttt{CARTESIAN} meshes with diameter $\texttt{h}$ =$\texttt{1/4}$, $\texttt{1/8}$, $\texttt{1/16}$, $\texttt{1/32}$ (see Fig.\ref{fig:meshes}).
Table~\ref{tab:test3-d1} presents the computed errors $\texttt{err}(\b u_h, \tri{\b \cdot}\tri_{r})$ (top panel), $\texttt{err}(\b u_h, W^{1,r})$ (middle panel), and $\texttt{err}(p_h, L^{r'})$ (bottom panel), cf. \eqref{eq:err_quant}, and the associated average experimental orders of convergence  \texttt{AEOC},
for the values of $r=\texttt{2.00,2.25,2.50,3.00}$ and $\delta=1$. The corresponding results for  $\delta=0$ are shown in Table~\ref{tab:test3-d0}. %In Table~\ref{tab:test3-d1} and Table~\ref{tab:test3-d0}.
Notice that, according to Theorem~\ref{theo:main} and Corollary~\ref{cor:W1r.estimate}, the expected rate of convergence for the velocity in both the discrete and continuous norm is $2k/r^2$
%. test reported in \cite[Section 7]{Belenki.Berselli.ea:12}). 
Finally, for the error $\texttt{err}(\b \sigma, L^{r'})$, linear convergence is observed.
%%%%%%%%%%%%%%%%%%%%%%%%%%%%%%%%%%%%%%%
%%%%%%%%%%%%% NEW TABLES delta=0
\begin{table}[!htbp]
\centering
$\texttt{err}(\b u_h, \tri{\b \cdot}\tri_{r})$\\
\begin{small}
\begin{tabular}{p{2cm}|cccc}
\toprule
  % void
& \multicolumn{4}{c}{\texttt{r}}
\\
%\midrule
   \texttt{1/h}
&  $\texttt{2.00}$ 
&  $\texttt{2.25}$      
&  $\texttt{2.50}$ 
&  $\texttt{3.00}$       
\\
\midrule          
  \texttt{4}
& \texttt{7.577256e-04}
& \texttt{3.757320e-03}
& \texttt{1.784934e-02}
& \texttt{1.572178e-01}
\\
  \texttt{8}
& \texttt{3.772397e-04}
& \texttt{2.139041e-03}
& \texttt{1.114599e-02}
& \texttt{1.064344e-01}
\\
  \texttt{16}
& \texttt{1.874426e-04}
& \texttt{1.222004e-03}
& \texttt{6.981838e-03}
& \texttt{7.427671e-02}
\\
  \texttt{32}
& \texttt{9.308978e-05}
& \texttt{6.988137e-04}
& \texttt{4.376918e-03}
& \texttt{5.255378e-02}
\\
\midrule     
{\texttt{AEOC}}
& \texttt{1.008326e+00}
& \texttt{8.0890813-01}
& \texttt{6.759612e-01}
& \texttt{5.269660e-01}
\\
\midrule     
\texttt{$\frac{4}{r^2}$}
& \texttt{1.00}
& \texttt{0.79}
& \texttt{0.64}
& \texttt{0.44}
\\
\bottomrule
\end{tabular}
\end{small}
\\[0.5cm]
$\texttt{err}(\b u_h, W^{1,r})$\\
\begin{small}
\begin{tabular}{p{2cm}|cccc}
\toprule
  % void
& \multicolumn{4}{c}{\texttt{r}}
\\
%\midrule
   \texttt{1/h}
&  $\texttt{2.00}$ 
&  $\texttt{2.25}$      
&  $\texttt{2.50}$ 
&  $\texttt{3.00}$       
\\
\midrule          
  \texttt{4}
& \texttt{7.576044e-04}
& \texttt{3.732379e-03}
& \texttt{1.774680e-02}
& \texttt{1.564706e-01}
\\
  \texttt{8}
& \texttt{3.772245e-04}
& \texttt{2.135423e-03}
& \texttt{1.112967e-02}
& \texttt{1.063065e-01}
\\
  \texttt{16}
& \texttt{1.874407e-04}
& \texttt{1.221485e-03}
& \texttt{6.979270e-03}
& \texttt{7.425567e-02}
\\
  \texttt{32}
& \texttt{9.308954e-05}
& \texttt{6.987394e-04}
& \texttt{4.376515e-03}
& \texttt{5.255063e-02}
\\
\midrule     
{\texttt{AEOC}}
& \texttt{1.008251e+00}
& \texttt{8.057564e-01}
& \texttt{6.732348e-01}
& \texttt{5.247039e-01}
\\
\midrule     
\texttt{$\frac{4}{r^2}$}
& \texttt{1.00}
& \texttt{0.79}
& \texttt{0.64}
& \texttt{0.44}
\\
\bottomrule
\end{tabular}
\end{small}
\\[0.5cm]
$\texttt{err}(p_h, L^{r'})$\\
\centering
\begin{small}
\begin{tabular}{p{2cm}|cccc}
\toprule
  % void
& \multicolumn{4}{c}{\texttt{r}}
\\
%\midrule
   \texttt{1/h}
&  $\texttt{2.00}$ 
&  $\texttt{2.25}$      
&  $\texttt{2.50}$ 
&  $\texttt{3.00}$       
\\
\midrule          
  \texttt{4}
& \texttt{1.173547e-01}
& \texttt{1.170279e-01}
& \texttt{1.237627e-01}
& \texttt{1.349763e-01}
\\
  \texttt{8}
& \texttt{5.832495e-02}
& \texttt{5.821788e-02}
& \texttt{6.162369e-02}
& \texttt{6.731919e-02}
\\
  \texttt{16}
& \texttt{2.896772e-02}
& \texttt{2.892321e-02}
& \texttt{3.062604e-02}
& \texttt{3.348035e-02}
\\
  \texttt{32}
& \texttt{1.438465e-02}
& \texttt{1.436359e-02}
& \texttt{1.521138e-02}
& \texttt{1.663401e-02}
\\
\midrule     
{\texttt{AEOC}}
& \texttt{1.009424e+00}
& \texttt{1.008788e+00}
& \texttt{1.008117e+00}
& \texttt{1.006832e+00}
\\
\midrule     
\texttt{$\Vert p_{\rm ex} - p_I \Vert_{L^{r'}}$}
& \texttt{1.00}
& \texttt{1.00}
& \texttt{1.00}
& \texttt{1.00}
\\
\bottomrule
\end{tabular}
\end{small}
\caption{Test 3. Computed errors $\texttt{err}(\b u_h, \tri{\b \cdot}\tri_{r})$ (top), $\texttt{err}(\b u_h, W^{1,r})$ (middle), and $\texttt{err}(p_h, L^{r'})$ (bottom) as in \eqref{eq:err_quant} for the mesh family \texttt{CARTESIAN}: $\delta=0$.}
\label{tab:test3-d0}
\end{table}
%%%%%%%%%%%%%
% OLD TABLES IN THE TEX FILE
%%%%%%%%%%%%%
%\FloatBarrier

\section{Conclusions}
\label{sec:conclusions}
We presented a theoretical analysis of Virtual Element discretizations of incompressible non-Newtonian flows governed by the Carreau–Yasuda constitutive law in the shear-thickening regime ($r>2$). Our analysis also covers the degenerate limit ($\delta = 0$), which corresponds to the power-law model. The proposed Virtual Element method is fully compatible with general polygonal meshes and yields an exactly divergence-free discrete velocity field.  To carry out the analysis, we introduced novel theoretical tools, including an inf–sup stability bound in non-Hilbertian norms, a suitably tuned stabilization for the case $r>2$, and discrete norms consistent with the constitutive law. We presented numerical results to demonstrate the theoretical findings and assess the practical performance of the proposed method.  The present results extend and complete those in \cite{Antonietti_et_al_2024}, which covered the case $1<r<2$ (shear-thinning regime), demonstrating that the VEM provides a robust discretization framework for Carreau–Yasuda non-Newtonian flows in both shear-thickening and shear-thinning regimes.  

\section*{Acknowledgments}
This research has been partially funded by the European Union (ERC, NEMESIS, project number 101115663). Views and opinions expressed are however those of the author(s) only and do not necessarily reflect those of the European Union or the European Research Council Executive Agency. Neither the European Union nor the granting authority can be held responsible for them.
The present research is part of the activities of ``Dipartimento di Eccellenza 2023-2027''. PA and MV also acknowledge MUR--PRIN/PNRR 2022 grant n. P2022BH5CB, funded by MUR.
GV has been partially funded by  
PRIN2022 n. \verb+2022MBY5JM+\emph{``FREYA - Fault REactivation: a hYbrid numerical Approach''} research grant, and
PRIN2022PNRR n. \verb+P2022M7JZW+\emph{``SAFER MESH - Sustainable mAnagement oF watEr Resources: ModEls and numerical MetHods''} research grant,
funded by the Italian Ministry of Universities and Research (MUR) and  by the European
Union through Next Generation EU, M4C2.
GV acknowledges financial support of INdAM-GNCS through project
CUP\verb+E53C24001950001+ \emph{``VEM per la trattazione di problemi definiti su domini parametrici o randomici''}.
The authors are members of INdAM-GNCS.
\bibliographystyle{plain}
\bibliography{references}
\end{document}